\documentclass[12pt]{amsart}

\usepackage[matrix,arrow,curve,frame]{xy}
\usepackage{amsmath,amsthm,amssymb,enumerate}
\usepackage{latexsym}
\usepackage{amscd}
\usepackage[colorlinks=false]{hyperref}
\usepackage{euscript}
\usepackage{appendix}

\newtheorem{thm}{Theorem}[section]

\newtheorem{cor}{Corollary}[section]
\newtheorem{lemma}{Lemma}[section]

\theoremstyle{definition}
\newtheorem{defn}{Definition}[section]
\newtheorem{conj}{Conjecture}[section]

\theoremstyle{remark}
\newtheorem{remark}{Remark}[section]

\numberwithin{equation}{section}

\def\cO{{\cal O}}

\def\cA{{\cal A}}

\def\ra{\rightarrow}
\def\bra{\langle}
\def\ket{\rangle}

\def\cA{{\mathcal A}}
\def\cB{{\mathcal B}}
\def\cC{{\mathcal C}}
\def\cD{{\mathcal D}}
\def\cE{{\mathcal E}}
\def\cF{{\mathcal F}}

\def\cH{{\mathcal H}}
\def\cI{{\mathcal I}}
\def\cJ{{\mathcal J}}

\def\cM{{\mathcal M}}
\def\cN{{\mathcal N}}
\def\cO{{\mathcal O}}

\def\cR{{\mathcal R}}
\def\cS{{\mathcal S}}

\def\cV{{\mathcal V}}
\def\cW{{\mathcal W}}

\def\cZ{{\mathcal Z}}

\def\gg{{\mathfrak g}}

\def\gl{{\mathfrak l}}

\def\go{{\mathfrak o}}
\def\gp{{\mathfrak p}}

\def\gs{{\mathfrak s}}

\DeclareMathOperator{\Hom}{Hom}

\DeclareMathOperator{\sgn}{sgn}

\newfont{\german}{eufm10}

\begin{document}
\pagestyle{plain}

\title
{Cosets of free field algebras via arc spaces}

\author{Andrew R. Linshaw}
\address{Department of Mathematics, University of Denver, C. M. Knudson Hall, 2390 S. York St. Denver, CO 80210}
\email{andrew.linshaw@du.edu}
\thanks{A. Linshaw is supported by Simons Foundation Grant \#635650 and NSF Grant DMS-2001484. He thanks T. Creutzig for helpful discussions and for suggesting the duality appearing in Theorem \ref{thm:levelranktypeB}.}

\author{Bailin Song}
\address{School of Mathematical Sciences, University of Science and Technology of China, Jinzhai Road 96, Hefei 230026, Anhui, P. R. China}
\email{bailinso@ustc.edu.cn}
\thanks{B. Song is supported by NSFC No. 12171447}


{\abstract \noindent Using the invariant theory of arc spaces, we find minimal strong generating sets for certain cosets of affine vertex algebras inside free field algebras that are related to classical Howe duality. These results have several applications. First, for any vertex algebra $\cV$, we have a surjective homomorphism of differential algebras $\mathbb{C}[J_{\infty}(X_{\cV})] \rightarrow \text{gr}^F(\cV)$; equivalently, the singular support of $\cV$ is a closed subscheme of the arc space of the associated scheme $X_{\cV}$. We give many new examples of classically free vertex algebras (i.e., this map is an isomorphism), including $L_k(\gs\gp_{2n})$ for all positive integers $n$ and $k$. We also give new examples where the kernel of this map is nontrivial but is finitely generated as a differential ideal. Next, we prove a coset realization of the subregular $\cW$-algebra of $\gs\gl_n$ at critical level that was previously conjectured by Creutzig, Gao, and the first author. Finally, we give some new level-rank dualities involving affine vertex superalgebras.}

\maketitle
\section{Introduction}

\subsection{Invariant theory of arc spaces}
In a series of papers \cite{LS1,LS2,LS3}, we have proven the arc space analogues of the first and second fundamental theorems of invariant theory for the general linear, special linear, and symplectic groups over an arbitrary algebraically closed field $K$. We briefly recall these results. First, given an irreducible scheme $X$ of finite type over $K$, the arc space $J_\infty(X)$ is determined by its functor of points. For every $K$-algebra $A$, we have a bijection
$$\Hom(\text{Spec}\ A, J_\infty(X))\cong\Hom(\text{Spec}\ A[[t]], X).$$ If $i: X\to Y$ is a morphism of schemes, we get a morphism of schemes $i_{\infty}:J_{\infty}(X)\to J_{\infty}(Y)$. 

Given an algebraic group $G$ over $K$, $J_{\infty}(G)$ is again an algebraic group. If $V$ is a finite-dimensional $G$-module, there is an induced action of $J_{\infty}(G)$ on $J_{\infty}(V)$, and the invariant ring $K[J_{\infty}(V)]^{J_{\infty}(G)}$ was studied in our earlier paper \cite{LSS1} with Schwarz. The quotient morphism $V\to V/\!\!/G$ induces a morphism $J_\infty(V)\to J_\infty(V/\!\!/G)$, so we have a morphism \begin{equation} \label{equ:invariantmap} J_\infty(V)/\!\!/J_\infty(G)\to J_\infty(V/\!\!/G).\end{equation}
In particular, we have a ring homomorphism
\begin{equation} \label{equ:invariantringmap} K[J_{\infty}(V/\!\!/G)] \rightarrow K[J_{\infty}(V)]^{J_{\infty}(G)}.\end{equation} In the case $K = \mathbb{C}$, if $V/\!\!/G$ is smooth or a complete intersection, and $K[V]$ has no nontrivial one-dimensional $G$-invariant subspaces, it was shown in \cite{LSS1} that \eqref{equ:invariantringmap} is an isomorphism, although in general it is neither injective nor surjective. The following results were proven in \cite{LS1,LS2,LS3}.

\begin{thm}\label{arcspaceinvt}  \begin{enumerate}
\item For $n\geq 1$, let $G = GL_n(K)$ and let $W = K^{\oplus n}$ be the standard module. Let $V = W^{\oplus p} \oplus (W^*)^{\oplus q}$ be the sum of $p$ copies of $W$ and $q$ copies of the dual module $W^*$. Then for all $p,q$, \eqref{equ:invariantringmap} is an isomorphism.
\item For $n\geq 2$ an even integer, let $G = Sp_n(K)$ and let $W = K^{\oplus n}$ be the standard module. Let $V = W^{\oplus p}$ be the sum of $p$ copies of $W$. Then for all $p$, \eqref{equ:invariantringmap} is an isomorphism.
\item For $n\geq 2$, let $G = SL_n(K)$ and let $W = K^{\oplus n}$ be the standard module. Let $V = W^{\oplus p} \oplus (W^*)^{\oplus q}$ be the sum of $p$ copies of $W$ and $q$ copies of the dual module $W^*$. Then 
\begin{itemize}
\item[(i)] If $p,q \leq n+2$, \eqref{equ:invariantringmap} is an isomorphism.
\item[(ii)] If $\text{max}\{p,q\} > n+2$, \eqref{equ:invariantringmap} is surjective but not injective.  When $\text{char} \ K = 0$, its kernel coincides with the nilradical $\cN \subseteq K[J_{\infty}(V/\!\!/G)]$, and an explicit finite generating set for $\cN$ as a differential ideal is given by Corollary 4.4 of \cite{LS3}.
\end{itemize}
\end{enumerate}
\end{thm}

These results were proven by constructing a standard monomial basis for the invariant spaces which extends the standard monomial basis in the classical setting. In this paper, we present some applications of these results to vertex algebras, and throughout the paper we will assume that $K = \mathbb{C}$.

\subsection{Vertex algebra coset problem} Given a vertex algebra $\cA$ and a subalgebra $\cV \subseteq \cA$, the coset $\cC = \text{Com}(\cV, \cA)$ is the subalgebra of $\cA$ which commutes with $\cV$. If $\cV$ is a homomorphic image of an affine vertex algebra $V^k(\gg)$ for some Lie algebra $\gg$, $\cC$ is called an affine coset and it is just the invariant space $\cA^{\gg[t]}$. Many interesting vertex algebras can be realized as affine cosets, including the principal $\cW$-algebras of types $A$, $B$, $C$, $D$ as well as principal $\cW$-superalgebras of $\go\gs\gp_{1|2n}$ \cite{ACL2,CL4}. There is a large class of vertex superalgebras $\cA^k$ which depend continuously on the parameter $k$ and admit a homomorphism $V^k(\gg)\rightarrow \cA^k$, such that the coset $\cC^k =  \text{Com}(V^k(\gg), \cA^k)$ can be described for generic values of $k$ by passing to the large $k$ limit, which is isomorphic to a certain orbifold of a free field algebra. This method was developed by the first author and Creutzig in \cite{CL1,CL3,CL4} and applies when $\cA^k$ is any $\cW$-algebra $\cW^k(\gg,f)$ where $\gg$ is a simple Lie superalgebra and $f$ is an even nilpotent element of $\gg$.

However, there are many affine cosets that cannot be studied using these methods. For example, given a finite-dimensional Lie algebra $\gg$ and finite-dimensional $\gg$-modules $V$ and $W$, there are induced homomorphisms $$V^{-k}(\gg)\rightarrow \cS(V),\qquad V^l(\gg) \rightarrow \cE(W),\qquad V^{-k+l}(\gg) \rightarrow \cS(V) \otimes \cE(W).$$ Here $\cS(V)$ and $\cE(W)$ denote the $\beta\gamma$-system and $bc$-system associated to $V$ and $W$, respectively, and $k,l$ are certain positive rational numbers. We denote the images of these affine vertex algebras by $\tilde{V}^{-k}(\gg)$, $\tilde{V}^{l}(\gg)$, and $\tilde{V}^{-k+l}(\gg)$, respectively. When $\gg$ is one of the classical Lie algebras and $V,W$ are sums of copies of the standard module, cosets of the form 
\begin{equation} \label{intro:cosets} \begin{split} & \text{Com}(\tilde{V}^{-k}(\gg),\cS(V)) = \cS(V)^{\gg[t]},\qquad \text{Com}(\tilde{V}^{l}(\gg),\cE(W)) = \cE(W)^{\gg[t]},
\\ & \text{Com}(\tilde{V}^{-k+l}(\gg),\cS(V) \otimes \cE(W)) = (\cS(V) \otimes \cE(W))^{\gg[t]},\end{split} \end{equation} are related to classical Howe duality \cite{H}, and have been studied by several authors \cite{LSS2,AKMPP,Gai}. One of the difficulties in describing $\cS(V)^{\gg[t]}$ is that when $k < h^{\vee}$, $\tilde{V}^{-k}(\gg)$ can be a quotient of $V^{-k}(\gg)$ which is not the simple quotient, and it need not act semisimply on $\cS(V)$; the same can hold for the other cosets in \eqref{intro:cosets}.

 A method for studying such cosets using the invariant theory of arc spaces was introduced in our joint paper with Schwarz \cite{LSS2}. First, $\cS(V)$ has a good increasing filtration such that the associated graded algebra 
 $$\text{gr}(\cS(V)) \cong \mathbb{C}[J_{\infty}(V\oplus V^*)],$$ and $\text{gr}(\cS(V))^{\gg[t]} \cong \mathbb{C}[J_{\infty}(V\oplus V^*)]^{J_{\infty}(G)}$. Here $G$ is a connected Lie group whose action is infinitesimally generated by the action of $\gg$. Via the inclusion $f: \cS(V)^{\gg[t]} \hookrightarrow \cS(V)$, the filtration on $\cS(V)$ induces a filtration on $\cS(V)^{\gg[t]}$, and we denote its associated graded algebra by $\text{gr}_f( \cS(V)^{\gg[t]})$. There is a homomorphism
\begin{equation} \label{intro:inclusion} \text{gr}_f(\cS(V)^{\gg[t]}) \hookrightarrow \text{gr}(\cS(V))^{\gg[t]}, \end{equation} which need not be surjective. However, if both \eqref{intro:inclusion} and the map
$\mathbb{C}[J_{\infty}((V\oplus V^*)/\!\!/G)] \rightarrow \mathbb{C}[J_{\infty}(V\oplus V^*)]^{J_{\infty}(G)}$ given by \eqref{equ:invariantringmap} are surjective, the generators of $\mathbb{C}[V\oplus V^*]^G$ will give rise to strong generators for $\cS(V)^{\gg[t]}$ as a vertex algebra. Finally, suppose that $U$ is a $G$-module such that for all $m\geq 1$, \eqref{equ:invariantringmap} is surjective for $V = (U\oplus U^*)^{\oplus m}$. Then for $V= U^{\oplus m}$ and $W = U^{\oplus r}$, we can use this approach to find strong generating sets for $\cE(W)^{\gg[t]}$ and $(\cS(V) \otimes \cE(W))^{\gg[t]}$ as well.

In \cite{LSS2}, we considered the cases where $\gg$ is one of the classical Lie algebras and $V$ is a sum of $m$ copies of the standard module. We were able to describe $\cS(V)^{\gg[t]}$ in all cases when $(V \oplus V^*)/\!\!/ G$ is an affine space, and all cases where it is a complete intersection and $\gg = \gs\gl_n$ or $\gs\gp_{2n}$, namely,
\begin{enumerate}
\item $\gg = \gs\gl_n$ and $V = (\mathbb{C}^n)^{\oplus m}$ for all $m \leq n$.
\item $\gg = \gs\go_{n}$ and $V = (\mathbb{C}^{n})^{\oplus m}$ for all $m < \frac{n}{2}$.
\item $\gg = \gs\gp_{2n}$ and $V = (\mathbb{C}^{2n})^{\oplus m}$ for all $m \leq n+1$.
\item $\gg = \gg\gl_n$ and $V = (\mathbb{C}^n)^{\oplus m}$ for all $m < n$.
\end{enumerate}

Now that Theorem \ref{arcspaceinvt} has been established, we improve upon these results by finding minimal strong generating sets for $\cS(V)^{\gg[t]}$ in the following cases:
\begin{enumerate}
\item $\gg = \gs\gl_n$ and $V = (\mathbb{C}^n)^{\oplus m}$ for all $m > n$.
\item $\gg = \gs\gp_{2n}$ and $V = (\mathbb{C}^{2n})^{\oplus m}$ for all $m > n+1$.
\item $\gg = \gg\gl_n$ and $V = (\mathbb{C}^n)^{\oplus m}$ for all $m \geq  n$.
\end{enumerate}
In case (3), \eqref{intro:inclusion} fails to be surjective, so we cannot use Theorem \ref{arcspaceinvt} directly; instead, we make use of the structure of $\cS(V)^{\gs\gl_n[t]}$. Unfortunately, the case $\gg = \gs\go_n$ and $V = (\mathbb{C}^n)^{\oplus m}$ for $m\geq \frac{n}{2}$ cannot be studied using these methods because \eqref{intro:inclusion} fails to be surjective \cite{LSS2}.

We expect that in case (1), $\cS(V)^{\gs\gl_n[t]}$ can be identified with vertex algebras appearing in other contexts, such as $\cW$-algebras. For example, when $n=m$, so that $V$ is the space of $n\times n$ matrices, $\cS(V)$ has two commuting actions of $V^{-n}(\gs\gl_n)$. It was conjectured in \cite{CGL} that $\cS(V)^{\gs\gl_n[t] \oplus \gs\gl_n[t]}$ is isomorphic to the Feigin-Semikhatov algebra $\cW^{(2)}_n$ at critical level $-n$ \cite{FS}, (which is isomorphic to the $\cW$-algebra $\cW^{-n}(\gs\gl_n, f_{\text{subreg}})$ associated to $\gs\gl_n$ with its subregular nilpotent \cite{G}), and this was proven for $n=2,3,4$. We will prove this conjecture for all $n$. This implies that $\cS(V)^{\gs\gl_n[t]}$ is isomorphic to a certain quotient of $V^{-n}(\gs\gl_n) \otimes \cW^{-n}(\gs\gl_n, f_{\text{subreg}})$.

We will also find minimal strong generating sets for $\cE(W)^{\gg[t]}$ in the following cases:
\begin{enumerate}
\item $\gg = \gs\gl_n$, $W  = (\mathbb{C}^n)^{\oplus r}$ for all $r \geq 1$.
\item $\gg = \gs\gp_{2n}$, $W = (\mathbb{C}^{2n})^{\oplus r}$ for all $r\geq 1$.
\item $\gg = \gg\gl_n$, $W  = (\mathbb{C}^n)^{\oplus r}$ for all $r \geq 1$.
\end{enumerate}
Finally, we find minimal strong generating sets for $(\cS(V) \otimes \cE(W))^{\gg[t]}$ in the following cases:
\begin{enumerate}
\item $\gg = \gs\gl_n$, $V = (\mathbb{C}^n)^{\oplus m}$, and $W  = (\mathbb{C}^n)^{\oplus r}$ for all $m,r \geq 1$.
\item $\gg = \gg\gl_n$, $V = (\mathbb{C}^n)^{\oplus m}$, and $W  = (\mathbb{C}^n)^{\oplus r}$ for all $m,r \geq 1$.
\item $\gg = \gs\gp_{2n}$, $V = (\mathbb{C}^{2n})^{\oplus m}$, and $W  = (\mathbb{C}^{2n})^{\oplus r}$ for all $m,r\geq 1$.
\end{enumerate}

\subsection{Singular support and associated scheme}
There are two well-known functors from the category of VOAs to the category of commutative rings. The first, which was introduced by Zhu \cite{Z}, assigns to a vertex algebra $\cV$ a commutative ring $R_{\cV}$. It is defined as the vector space quotient of $\cV$ by the span of all elements of the form $a_{(-2)} b$ for all $a,b \in \cV$. The normally ordered product on $\cV$ descends to a commutative, associative product on $R_{\cV}$. Strong generators for $\cV$ give rise to generators for $R_{\cV}$; in particular, $R_{\cV}$ is finitely generated if and only if $\cV$ is strongly finitely generated. Recall that $\cV$ is said to be {\it $C_2$-cofinite} if $R_{\cV}$ is finite-dimensional as a vector space. This is a key starting assumption in Zhu's work on modularity of characters of modules for rational vertex algebras \cite{Z}, and it implies that $\cV$ has finitely many simple $\mathbb{Z}_{\geq 0}$-graded modules.

The second functor comes from Li's canonical decreasing filtration $F^{\bullet} \cV$ that is defined on any VOA $\cV$, such that the associated graded algebra $\text{gr}^F(\cV)$ is a differential graded commutative ring \cite{Li}. If $\cV$ has finite-dimensional weight spaces (which is the case in all our examples), $\cV$ is linearly isomorphic to $\text{gr}^F(\cV)$, and a strong generating set for $\cV$ gives rise to a generating set for $\text{gr}^F(\cV)$ as a differential algebra. In fact, $R_{\cV}$ can be identified with the zeroth graded component of $\text{gr}^F(\cV)$, so $R_{\cV}$ generates $\text{gr}^F(\cV)$ as a differential algebra. If $\cV$ is {\it freely generated} by a set of fields $\{\alpha_i\}$, then $\text{gr}^F(\cV)$ is just the differential polynomial algebra generated by $\{\alpha_i\}$. However, if $\cV$ is not freely generated, it is a difficult and important problem to find all differential algebraic relations in $\text{gr}^F(\cV)$.

Following Arakawa \cite{Ar2,Ar3}, we define the {\it associated scheme} of $\cV$ to be
$$X_{\cV} = \text{Spec} \ R_{\cV},$$
which is an affine Poisson scheme. The {\it singular support} of $\cV$ is defined to be
$$\text{SS}(\cV) = \text{Spec} \ \text{gr}^F(\cV),$$ 
which is a vertex Poisson scheme. Let $(R_{\cV})_{\infty}$ denote the affine coordinate ring of the arc space $J_{\infty}(X_{\cV})$. By its universal property, there is a surjective homomorphism of differential rings
\begin{equation} \label{eq:ss} \Phi: (R_{\cV})_{\infty} \rightarrow \text{gr}^F(\cV).\end{equation} Equivalently, there is a closed embedding $\text{SS}(\cV)\hookrightarrow J_{\infty}(X_{\cV})$.
In the terminology of van Ekeren and Heluani \cite{EH1}, $\cV$ is called {\it classically free} if \eqref{eq:ss} is an isomorphism. This property plays an important role in their computations of chiral homology. It is easy to see that any freely generated vertex algebra is classically free; this includes all free field algebras, universal affine vertex algebras, and universal $\cW$-algebras. If $\cV$ is not freely generated, the phenomenon is much more subtle, but examples are known including the simple affine vertex algebras $L_k(\gs\gl_2)$ for $k \in \mathbb{N}$, and the Virasoro minimal models $\text{Vir}_{2,q}$ for all odd $q\geq 3$ \cite{EH1}. Further examples also appear in recent work of Li and Milas \cite{L,LM}. 

However, not all vertex algebras are classically free. For example, $\text{Vir}_{p,q}$ is classically free if and only if $p = 2$ \cite{EH1}. The singlet algebras of type $\cW(2,2p-2)$ for $p\geq 2$ introduced by Adamovi\'c \cite{A} are not classically free; this was shown in \cite{AL} in the cases $p=2,3$ and the general case is similar. We say that $\cV$ is {\it classically free at the level of varieties} if \eqref{eq:ss} induces an isomorphism of reduced schemes. This is the case in all known examples where $\cV$ is simple, and was proven by Arakawa and Moreau whenever $\cV$ is simple and quasi-lisse \cite{AM}. If \eqref{eq:ss} fails to be injective, its kernel is always a differential ideal and one can ask whether it is finitely generated as a differential ideal \cite{AL}. There is currently only one example in the literature where this has been proven, namely $\text{Vir}_{3,4}$; it was shown in \cite{AEH} that $\text{ker} \ \Phi$ is generated as a differential ideal by one element.

In this paper, we will describe $\text{gr}^F(\cV)$ for certain affine cosets of the form $\cS(V)^{\gg[t]}$ and $\cE(V)^{\gg[t]}$, and we will compute the kernel of \eqref{eq:ss}. As a result, we give many new examples of classically free vertex algebras which are not freely generated. For example, we will show that $L_n(\gs\gp_{2r})$ is classically free for all $r,n \in \mathbb{N}$, generalizing and providing a conceptual proof of van Ekeren and Heluani's result for $L_n(\gs\gl_2)$ \cite{EH1}. This implies that Theorem 10.2.1 of \cite{EH2}, namely the vanishing of the first chiral homology $H_1^{\text{ch}}(V)$, holds for $V = L_n(\gs\gp_{2r})$. We also give many new examples where \eqref{eq:ss} is not injective, where we have an explicit finite generating set for $\text{ker} \ \Phi$ as a differential ideal.

\subsection{Level-rank dualities involving affine vertex superalgebras} 
It is well known that for positive integers $n,m,r$, there is an embedding $L_m(\gs\gl_n) \otimes L_m(\gs\gl_r) \otimes \cH \hookrightarrow L_m(\gs\gl_{r+n})$ and the following level-rank duality holds:
$$\text{Com}(L_{n+r}(\gs\gl_m), L_n(\gs\gl_m) \otimes L_{r}(\gs\gl_m)) \cong \text{Com}(L_m(\gs\gl_n) \otimes L_m(\gs\gl_r) \otimes \cH, L_m(\gs\gl_{r+n})).$$
This duality has appeared in \cite{ABI,F,NT,Wal} and was proven by Jiang and Lin \cite{JL} as well as \cite{ACL2}. It is natural to ask if there is a similar duality where the positive integer $r$ is replaced with $-r$, and $L_m(\gs\gl_{r+n})$ is replaced with the affine Lie superalgebra $L_m(\gs\gl_{r|n})$. A weaker statement of this kind was proven by Creutzig, Riedler and the first author in \cite{CLR}. It says that 
\begin{equation} \label{thm:CLR} \text{Com}(\tilde{V}^{-n+r}(\mathfrak{sl}_m), A^{-n}(\mathfrak{sl}_{m}) \otimes L_r(\mathfrak{sl}_m))\cong \text{Com}(\tilde{V}^{-m}(\mathfrak{sl}_n) \otimes L_m(\mathfrak{sl}_r) \otimes \cH, A^m(\mathfrak{sl}_{r|n})).\end{equation}
In this notation, $$A^{-n}(\mathfrak{sl}_{m}) = \cS(nm)^{\gg\gl_n[t]},$$ which is an extension of $\tilde{V}^{-n}(\gs\gl_m)$, and 
$$A^m(\mathfrak{sl}_{r|n}) =  (\cS(mn) \otimes \cE(mr))^{\gg\gl_m[t]},$$ 
which is an extension of $\tilde{V}^m(\mathfrak{sl}_{r|n})$. Also, $\tilde{V}^{-n+r}(\mathfrak{sl}_m)$ denotes the image of $V^{-n+r}(\mathfrak{sl}_m)$ under the diagonal map $V^{-n+r}(\mathfrak{sl}_m) \rightarrow A^{-n}(\mathfrak{sl}_{m}) \otimes L_r(\mathfrak{sl}_m)$.

The question was raised in \cite{CLR} whether this can be improved by replacing $A^{-n}(\mathfrak{sl}_{m})$ and $A^m(\mathfrak{sl}_{r|n})$ with $\tilde{V}^{-n}(\mathfrak{sl}_{m})$ and $\tilde{V}^m(\mathfrak{sl}_{r|n})$, respectively. We will see that $A^{-n}(\mathfrak{sl}_{m}) = \tilde{V}^{-n}(\gs\gl_m)$ for all $m < n$ and $m\geq 2n+1$, and $A^m(\mathfrak{sl}_{r|n}) = \tilde{V}^m(\mathfrak{sl}_{r|n})$ for all $m,r,n\geq 1$, so we can indeed improve this result. In the case $m=n$, $A^{-n}(\mathfrak{sl}_{m})$ cannot be replaced with $\tilde{V}^{-n}(\gs\gl_m)$. We are not able to determine this in the range $n<m < 2n+1$.

We conclude by giving the analogous level-rank dualities in types $B$, $C$, and $D$.

\section{Vertex algebras} \label{section:VOAs}
In this paper, we will follow the formalism developed in \cite{K}. A {\it vertex algebra} is the data $(\cA, Y, L_{-1}, 1)$, where
\begin{enumerate} 
\item $\cA = \cA_{\bar{0}} \oplus \cA_{\bar{1}}$ is a $\mathbb Z_2$-graded vector space over $\mathbb{C}$. For $a \in \cA_i$, $|a| = i$ for $i = \bar{0}, \bar{1}$.
\item $Y$ is an even linear map
$$Y : \cA \rightarrow \text{End} (\cA) [[z, z^{-1}]],\qquad Y(a) = a(z) = \sum_{n\in \mathbb{Z}}a_{(n)}z^{-n-1}.$$ Here $z$ is a formal variable and $a(z)$ is called the field corresponding to $a$.  
\item $1 \in \cA$ is a distinguished element called the vacuum vector. 
\item $L_{-1}$ is an even endomorphism of $\cA$.
\end{enumerate}
They satisfy the following axioms:
\begin{itemize}
\item \textit{Vacuum axiom}: $L_{-1}1 =0$; $1(z)= \text{Id}$; for $a\in \cA$, $n\geq 0$, $a_{(n)}1=0$ and $a_{(-1)}1=a$,
\item \textit{Translation invariance axiom}: For $a\in \cA$, $[L_{-1}, Y(a)] =\partial a(z)$,
\item \textit{Locality axiom}: Let $z,w$ be formal variables. For homogeneous $a, b \in \cA$, $(z-w)^k [a(z), b(w)]=0$ for some $k\geq 0$, where $[a(z), b(w)] = a(z) b(w) - (-1)^{|a| |b|} b(w) a(x)$. \end{itemize}

For $a,b\in \cA$ and $n\in \mathbb Z_{\geq 0}$, the $n^{\text{th}}$ product is denoted by $a_{(n)}b$, and the operator product expansion (OPE) is given by
$$a(z)b(w)\sim \sum_{n\geq 0} (a_{(n)}b)(w)(z-w)^{-n-1}.$$ Here $\sim$ means modulo the terms which are regular at $z= w$.
The Wick product (or normally ordered product) of $a(z)$ and $b(z)$ is
$$:a(z)b(z):  \ =(a_{(-1)}b)(z) = a(z)_-b(z)\ +\ (-1)^{|a||b|} b(z)a(z)_+,$$ where $a(z)_-=\sum_{n<0}a_{(n)}z^{-n-1}$ and $a(z)_+=\sum_{n\geq 0}a_{(n)}z^{-n-1}$.  The other negative products are given by
$$:(\partial^na(z))b(z):\ =n!(a_{(-n-1)}b)(z),\qquad \partial = \frac{d}{dz}.$$
For $a_1, \dots , a_k\in \cA$, their iterated Wick product is defined to be
$$:a_1(z)\cdots a_k(z):\ =\ :a_1(z)b(z):,\quad \quad b(z)=\ :a_2(z)\cdots a_k(z):.$$
We often omit the formal variable $z$ when no confusion can arise.

A vertex algebra $\cA$ is said to be {\it generated} by a subset $S=\{\alpha^i|\ i\in I\}$ if $\cA$ is spanned by words in the letters $\alpha^i$, and all products, for $i\in I$ and $n\in\mathbb{Z}$. We say that $S$ {\it strongly generates} $\cA$ if $\cA$ is spanned by words in the letters $\alpha^i$, and all products for $n<0$. Equivalently, $\cA$ is spanned by $$\{ :\partial^{k_1} \alpha^{i_1}\cdots \partial^{k_m} \alpha^{i_m}:| \ i_1,\dots,i_m \in I,\ k_1,\dots,k_m \geq 0\}.$$ 

Suppose that $S$ is an ordered strong generating set $\{\alpha^1, \alpha^2,\dots\}$ for $\cA$ which is at most countable. We say that $S$ {\it freely generates} $\cA$, if $\cA$ has a Poincar\'e-Birkhoff-Witt basis consisting of all normally ordered monomials 
\begin{equation} \label{freegen} \begin{split} & :\partial^{k^1_1} \alpha^{i_1} \cdots \partial^{k^1_{r_1}}\alpha^{i_1} \partial^{k^2_1} \alpha^{i_2} \cdots \partial^{k^2_{r_2}}\alpha^{i_2}
 \cdots \partial^{k^n_1} \alpha^{i_n} \cdots \partial^{k^n_{r_n}} \alpha^{i_n}:,\qquad 
 1\leq i_1 < \dots < i_n,
 \\ & k^1_1\geq k^1_2\geq \cdots \geq k^1_{r_1},\quad k^2_1\geq k^2_2\geq \cdots \geq k^2_{r_2},  \ \ \cdots,\ \  k^n_1\geq k^n_2\geq \cdots \geq k^n_{r_n},
 \\ &  k^{t}_1 > k^t_2 > \dots > k^t_{r_t} \ \ \text{if} \ \ \alpha^{i_t}\ \ \text{is odd}. 
 \end{split} \end{equation}

A conformal structure with central charge $c$ is a Virasoro vector $L(z) = \sum_{n\in \mathbb{Z}} L_n z^{-n-2} \in \cA$ satisfying
\begin{equation} \label{virope} L(z) L(w) \sim \frac{c}{2}(z-w)^{-4} + 2 L(w)(z-w)^{-2} + \partial L(w)(z-w)^{-1},\end{equation} such that in addition, $L_{-1} \alpha = \partial \alpha$ for all $\alpha \in \cA$, and $L_0$ acts diagonalizably on $\cA$. We say that $\alpha$ has conformal weight $d$ if $L_0(\alpha) = d \alpha$, and we denote the conformal weight $d$ subspace by $\cA[d]$. In all our examples, the conformal weight grading will be either by $\mathbb{Z}_{\geq 0}$ or $\frac{1}{2} \mathbb{Z}_{\geq 0}$.

We recall some identities that hold in any vertex algebra $\cA$.  For any fields $a,b,c \in \cA$,
\begin{equation} \label{deriv} (\partial a)_{(n)} b = -na_{(n-1)} b\qquad \forall n\in \mathbb{Z},\end{equation}
\begin{equation} \label{commutator} a_{(n)} b  = (-1)^{|a||b|} \sum_{p \in \mathbb{Z}} (-1)^{p+1} (b_{(p)} a)_{(n-p-1)} 1,\qquad \forall n\in \mathbb{Z},\end{equation}
\begin{equation} \label{nonasswick} :(:ab:)c:\  - \ :abc:\ 
=  \sum_{n\geq 0}\frac{1}{(n+1)!}\big( :(\partial^{n+1} a)(b_{(n)} c):\ +
(-1)^{|a||b|} (\partial^{n+1} b)(a_{(n)} c):\big)\ .\end{equation}
\begin{equation} \label{ncw} a_{(n)}
(:bc:) -\ :(a_{(n)} b)c:\ - (-1)^{|a||b|}\ :b(a_{(n)} c): \ = \sum_{i=1}^n
\binom{n}{i} (a_{(n-i)}b)_{(i-1)}c, \qquad \forall n \geq 0.
\end{equation}
Given $a,b,c \in  \cA$ and integers $m,n \geq 0$, the following identities are known as Jacobi relations of type $(a,b,c)$. 
\begin{equation} \label{jacobi} a_{(r)}(b_{(s)} c) = (-1)^{|a||b|} b_{(s)} (a_{(r)}c) + \sum_{i =0}^r \binom{r}{i} (a_{(i)}b)_{(r+s - i)} c.\end{equation}

\subsection{Affine vertex algebras} Let $\gg$ be a simple, finite-dimensional Lie (super)algebra. The {\it universal affine vertex algebra} $V^k(\gg)$ is freely generated by fields $X^{\xi}$ which are linear in $\xi \in \gg$ and satisfy 
\begin{equation}
X^{\xi}(z) X^{\eta}(w) \sim k ( \xi, \eta) (z-w)^{-2} + X^{[\xi,\eta]}(w)(z-w)^{-1}.
\end{equation}
Here $(\cdot ,\cdot )$ denotes the normalized Killing form $\frac{1}{2h^{\vee}} \langle \cdot,\cdot \rangle$. For all $k\neq -h^{\vee}$, $V^k(\gg)$ has the Sugawara Virasoro vector
\begin{equation} \label{sugawara} L^{\gg}  = \frac{1}{2(k+h^{\vee})} \sum_{i=1}^n :X^{\xi_i} X^{\xi'_i}: \end{equation} of central charge $c = \frac{k\ \text{sdim}(\gg)}{k+h^{\vee}}$. Here $\xi_i$ runs over a basis of $\gg$, and $\xi'_i$ is the dual basis with respect to $(\cdot,\cdot)$.

As a module over $\widehat{\gg} = \gg[t,t^{-1}] \oplus \mathbb{C}$, $V^k(\gg)$ is isomorphic to the vacuum $\widehat{\gg}$-module. For generic $k$, $V^k(\gg)$ is a simple vertex algebra, but for certain rational values of $k \geq -h^{\vee}$ which were classified by Gorelik and Kac \cite{GK}, $V^k(\gg)$ is not simple, and we denote by $L_k(\gg)$ its simple graded quotient.

In the case $\gg = \go\gs\gp_{m|2n}$, we take the dual Coxeter number to be $$h^{\vee} = \frac{2n+2-m}{2}.$$ Then the bilinear form on $\go\gs\gp_{m|2n}$ is normalized so that it coincides with the usual bilinear form on $\gs\gp_{2n}$, and we have the embedding \begin{equation} \label{osp:convention} V^{k}(\gs\gp_{2n}) \otimes V^{-2k}(\gs\go_m) \rightarrow V^k(\go\gs\gp_{m|2n}).\end{equation}

\subsection{$\beta\gamma$-system} 
Let $V$ be a finite-dimensional complex vector space. The $\beta\gamma$-system $\cS(V)$ was introduced in \cite{FMS}. It is freely generated by even fields $\beta^{x}$, $\gamma^{x'}$ which are linear in $x\in V$, $x'\in V^*$, and satisfy
\begin{equation}\begin{split} & \beta^x(z)\gamma^{x'}(w)\sim\langle x',x\rangle (z-w)^{-1}, \qquad \gamma^{x'}(z)\beta^x(w)\sim -\langle x',x\rangle (z-w)^{-1},
\\ & \beta^x(z)\beta^y(w)\sim 0, \qquad \gamma^{x'}(z)\gamma^{y'}(w)\sim 0.\end{split} \end{equation} 
Here $\bra,\ket$ denotes the pairing between $V^*$ and $V$. If we fix a basis $x_1,\dots, x_n$ for $V$ and dual basis $x'_1,\dots, x'_n$ for $V^*$, we often denote the corresponding fields by $\beta^1,\dots, \beta^n$ and $\gamma^1,\dots, \gamma^n$, so that \begin{equation} \label{eq:betagammaope} \begin{split} \beta^i(z)\gamma^{j}(w) &\sim \delta_{i,j} (z-w)^{-1},\quad \gamma^{i}(z)\beta^j(w)\sim -\delta_{i,j} (z-w)^{-1},\\  \beta^i(z)\beta^j(w) &\sim 0,\qquad\qquad\qquad \gamma^i(z)\gamma^j (w)\sim 0,\end{split} \end{equation} and we denote $\cS(V)$ by $\cS(n)$. It has Virasoro element $$L^{\cS} = \frac{1}{2} \sum_{i=1}^n \big(:\beta^{i}\partial\gamma^{i}: - :\partial\beta^{i}\gamma^{i}:\big)$$ of central charge $-n$, under which $\beta^{i}$, $\gamma^{i}$ are primary of weight $\frac{1}{2}$. The symplectic group $Sp_{2n}$ is the full automorphism group of $\cS(n)$ preserving $L^{\cS}$. In fact, there is a homomorphism $L_{-1/2}(\gs\gp_{2n}) \rightarrow \cS(n)$ whose zero modes infinitesimally generate the action of $Sp_{2n}$. 
There is an additional $\mathbb{Z}$-grading on $\cS(n)$ which we call the {\it charge}. Define 
\begin{equation}\label{defchargebg} e = \sum_{i=1}^n : \beta^{i}\gamma^{i}:.\end{equation} The zero mode $e(0)$ acts diagonalizably on $\cS(n)$. The charge grading is just the eigenspace decomposition of $\cS(n)$ under $e(0)$, and $\beta^{i}$, $\gamma^{i}$ have charges $-1$, $1$, respectively.

\subsection{$bc$-system} There is a similar vertex superalgebra $\cE(V)$ known as a $bc$-system \cite{FMS}. It is freely generated by odd fields $b^{x}$, $c^{x'}$ which are linear in $x\in V$, $x'\in V^*$, and satisfy
\begin{equation}\begin{split} & b^x(z) c^{x'}(w)\sim\langle x',x\rangle (z-w)^{-1}, \qquad c^{x'}(z) b^x(w)\sim \langle x',x\rangle (z-w)^{-1},
\\ & b^x(z) b^y(w)\sim 0, \qquad c^{x'}(z) c^{y'}(w)\sim 0.\end{split} \end{equation} 
If we fix a basis $x_1,\dots, x_n$ for $V$ and dual basis $x'_1,\dots, x'_n$ for $V^*$, we often denote the corresponding fields by $b^1,\dots, b^n$ and $c^1,\dots, c^n$, so that 
\begin{equation} \label{eq:betagammaope} \begin{split} b^i(z) c^{j}(w) &\sim \delta_{i,j} (z-w)^{-1},\quad c^{i}(z) b^j(w)\sim \delta_{i,j} (z-w)^{-1},\\  b^i(z) b^j(w) &\sim 0,\qquad\qquad\qquad c^i(z)c^j (w)\sim 0,\end{split} \end{equation} and we denote $\cE(V)$ by $\cE(n)$. It has Virasoro element 
$$L^{\cE}= \frac{1}{2} \sum_{i=1}^n \big(-:b^{i}\partial c^{i}: + :\partial b^{i} c^{i}:\big)$$
 of central charge $n$, under which $b^{i}$, $c^{i}$ are primary of weight $\frac{1}{2}$. The orthogonal group $O_{2n}$ is the full automorphism group of $\cE(n)$ preserving $L^{\cE}$, and there is a homomorphism $L_{1}(\gs\go_{2n}) \rightarrow \cE(n)$ which infinitesimally generates the action of $O_{2n}$. 
As above, $\cE(n)$ has an additional $\mathbb{Z}$-grading called charge, given by the eigenvalue of the zero mode of the operator
\begin{equation}\label{defchargebc} e = -\sum_{i=1}^n : b^{i}c^{i}:.\end{equation} Then $b^{i}$, $c^{i}$ have charges $-1$, $1$, respectively.

\subsection{Free fermion algebra} If $V$ is a vector space with a symmetric, nondegenerate form $\langle, \rangle$, one can also associate to $V$ the {\it free fermion algebra} $\cF(V)$. It is a vertex superalgebra with odd generators $\phi^u$ which are linear in $u \in V$, and satisfy
 $$\phi^u(z) \phi^v(w) \sim \langle u, v \rangle (z-w)^{-1}.$$ 
 If we fix an orthonormal basis $v_1,\dots, v_n$ for $V$ relative to $\langle, \rangle$, we denote the corresponding fields by $\phi^{1},\dots, \phi^n$, and they satisfy $\phi^{i} (z) \phi^{j}(w) \sim \delta_{i,j} (z-w)^{-1}$. We often denote $\cF(V)$ by $\cF(n)$. We have the Virasoro element  $$L^{\cF} =  -\frac{1}{2} \sum_{i=1}^n  :\phi^i \partial \phi^i:$$ of central charge $\frac{n}{2}$, under which $\phi^i$ is primary of weight $\frac{1}{2}$. The full automorphism group of $\cF(n)$ is the orthogonal group $O_n$, and there is a homomorphism $L_1(\gs\go_n) \rightarrow \cF(n)$ which infinitesimally generates the $O_n$-action. Also, note that $\cE(n) \cong \cF(2n)$ as vertex algebras.

 \subsection{Affine vertex algebra actions on free field algebras}
In this section, we recall several well-known homomorphisms from an affine vertex (super)algebra $V^k(\gg)$ to some free field algebras $\cB$. We often use the notation $\tilde{V}^k(\gg)$ to denote the image of such a homomorphism, which need not be either the universal algebra $V^k(\gg)$ or its simple quotient $L_k(\gg)$.

Let $\gg$ be a simple finite-dimensional Lie algebra, and let $V$ be a finite-dimensional $\gg$-module via $\rho_1:\gg\rightarrow \gg\gl(V)$. There is an induced homomorphism
\begin{equation} \label{def:rho12even} V^{-k}(\gg)\ra \cS(V),\qquad X^{\xi} \mapsto  - \sum_{i=1}^{\text{dim} \ V} :\gamma^{x'_i}\beta^{\rho_1(\xi)(x_i)}:. \end{equation} Here $k$ is given as follows: the bilinear form $\text{Tr}(\rho_1(\xi)\rho_1(\eta))$ on $\gg$ is equal to $k$ times  the normalized Killing form. Then $\{\beta^{x_i}\}$ and $\{\gamma^{x'_i}\}$ transform under $\gg$ as $V$ and $V^*$, respectively.

An important case is $\gg = \gg\gl_n$ and $V$ the standard module $\mathbb{C}^n$; this gives the embedding 
\begin{equation} \label{glnlevel-1} L_{-1}(\gg\gl_n) = \cH \otimes L_{-1}(\gs\gl_n) \rightarrow \cS(n).\end{equation} More generally, for $V = (\mathbb{C}^n)^{\oplus m}$, $\cS(V) = \cS(nm)$ admits a homomorphism
\begin{equation} \label{conformalemb:slnslmbg} V^{-m}(\gs\gl_{n}) \otimes V^{-n}(\gs\gl_m) \otimes \cH \rightarrow \cS(nm), \end{equation} whose image $\tilde{V}^{-m}(\gs\gl_{n}) \otimes \tilde{V}^{-n}(\gs\gl_m) \otimes \cH$ is conformally embedded in $\cS(nm)$.

Next, $\gs\gp_{2n}\subseteq \gg\gl_{2n}$ consists of block matrices of the form $$\bigg[ \begin{matrix}  A & B \cr  C & -A^T \end{matrix} \bigg], \qquad A,B,C\in \gg\gl_n, \qquad B=B^T, \qquad C = C^T.$$ In terms of the basis $\{e_{i,j}|\ 1\leq i\leq 2n, \ 1\leq j\leq 2n\} \subseteq \gg\gl_{2n}$, a standard basis for $\gs\gp_{2n}$ is $$e_{j,k+n} + e_{k,j+n}, \qquad  -e_{j+n,k} - e_{k+n,j}, \qquad e_{j,k} - e_{n+k, n+j}, \qquad 1\leq j, \qquad k\leq n.$$ 

There is a homomorphism $V^{-m/2}(\gs\gp_{2n}) \rightarrow \cS(nm)$ given by
\begin{equation} \begin{split} & X^{e_{j,k+n} + e_{k,j+n}} \mapsto \sum_{i=1}^m : \gamma^{x'_{i,j}} \gamma^{x'_{i,k}}:, 
\\ & X^{-e_{j+n,k} - e_{k+n,j}} \mapsto \sum_{i=1}^m :\beta^{x_{i,j}}\beta^{x_{i,k}}:,
\\ & X^{e_{j,k} - e_{n+k, n+j}} \mapsto \sum_{i=1}^m :\gamma^{x'_{i,j}} \beta^{x'_{i,k}}:. \end{split} \end{equation}
Note that for each $i = 1,\dots, m$, $U_i = \text{span}\{\gamma^{i,1},\dots, \gamma^{i,n}, \beta^{i,1},\dots, \beta^{i,n}\}$ is a copy of the standard $\gs\gp_{2n}$-module. In fact, there is a commuting homomorphism $$V^{-2n}(\gs\go_m) \rightarrow \cS(nm)$$ which is a special case of \eqref{def:rho12even} with $\gg = \gs\go_m$ and $V = (\mathbb{C}^m)^{\oplus n}$. Combining these maps yields a homomorphism
\begin{equation} \label{conformalemb:sp2nsombg}  V^{-m/2}(\gs\gp_{2n}) \otimes V^{-2n}(\gs\go_m) \rightarrow \cS(nm),\end{equation} whose image $\tilde{V}^{-m/2}(\gs\gp_{2n}) \otimes \tilde{V}^{-2n}(\gs\go_m)$ is conformally embedded in $\cS(nm)$.

\begin{remark} One more case of \eqref{def:rho12even} was studied in \cite{LSS2}, namely, $\gg = \gs\gp_{2n}$ and $V = (\mathbb{C}^{2n})^{\oplus m}$. In this case, we have a homomorphism $V^{-m}(\gs\gp_{2n}) \otimes V^{-2n}(\gs\go_{2m}) \rightarrow \cS(2nm)$ whose image is conformally embedded. However, after a change of basis this is equivalent to \eqref{conformalemb:sp2nsombg} with $m$ replaced with $2m$, so we need not consider this as a separate case.
\end{remark}

Next, we consider affine algebra actions on $bc$-systems and free fermion algebras. Let $\gg$ be a simple finite-dimensional Lie algebra, and let $W$ be a finite-dimensional $\gg$-module via $\rho_2:\gg\rightarrow \gg\gl(W)$. There is an induced homomorphism
\begin{equation} \label{def:rho12odd} V^{l}(\gg)\ra \cE(W),\qquad X^{\xi} \mapsto  \sum_{i=1}^{\text{dim} \ W} : c^{y'_i} b^{\rho_2(\xi)(y_i)}:.\end{equation} Here $\ell$ is determined as follows: the bilinear form $\text{Tr}(\rho_2(\xi)\rho_2(\eta))$, on $\gg$ is equal to $l$ times the normalized Killing form. Then $\{b^{y_i}\}$ and $\{c^{y'_i}\}$ transform under $\gg$ as $V$ and $V^*$, respectively. 

For example, if $\gg = \gg\gl_n$ and $W = \mathbb{C}^n$, we have an embedding 
 \begin{equation} \label{glnlevel1} L_{1}(\gg\gl_n) = \cH \otimes L_{1}(\gs\gl_n) \rightarrow \cE(n).\end{equation} Similarly, for $V = (\mathbb{C}^n)^{\oplus m}$, $\cE(V) = \cE(nm)$ admits a homomorphism
\begin{equation} \label{conformalemb:slnslmbc} L_{m}(\gs\gl_{n}) \otimes L_{n}(\gs\gl_m) \otimes \cH \rightarrow \cE(nm), \end{equation} whose image is conformally embedded in $\cE(nm)$. 

Another special case is $\gg = \gs\gp_{2n}$ and $V = (\mathbb{C}^{2n})^{\oplus m}$. In this case, we have a homomorphism
\begin{equation} \label{sp2n-sp2mE} L_{m}(\gs\gp_{2n}) \otimes  L_n(\gs\gp_{2m}) \rightarrow \cE(2nm),\end{equation} whose image is conformally embedded. Note that $\cE(2nm) \cong \cF(4nm)$ which is an extension of $L_1(\gs\go_{4nm})$, and this is equivalent to the well-known conformal embedding $L_{m}(\gs\gp_{2n}) \otimes L_n(\gs\gp_{2m}) \rightarrow L_1(\gs\go_{4nm})$.

Next, there is a conformal embedding
\begin{equation} \label{son-somF} L_m(\gs\go_n) \otimes L_n(\gs\go_m) \rightarrow \cF(nm), \end{equation}  which is equivalent to the well-known conformal embedding $ L_m(\gs\go_n) \otimes L_n(\gs\go_m) \rightarrow L_1(\gs\go_{nm})$.

Finally, we consider certain actions of affine vertex (super)algebras in tensor products of $\beta\gamma$-systems and $bc$-systems. Let $\gg$ be a simple Lie algebra and let $\rho_1:\gg\rightarrow \gg\gl(V)$ and  $\rho_2:\gg\rightarrow \gg\gl(W)$ be finite-dimensional $\gg$-modules, as above. There is then a homomorphism
\begin{equation} \label{def:rhobcbg} V^{-k+l}(\gg)\ra \cS(V) \otimes \cE(W),\qquad X^{\xi} \mapsto  - \sum_{i=1}^{\text{dim} \ V} :\gamma^{x'_i}\beta^{\rho(\xi)(x_i)}: + \sum_{j=1}^{\text{dim} \ W} : c^{y'_j} b^{\rho_2(\xi)(y_j)}:.\end{equation}

We also have the following well-known homomorphisms whose images are conformally embedded.
\begin{equation} \label{super:1} V^{-m+r}(\gg\gl_n) \otimes V^n (\gs\gl_{r|m}) \rightarrow \cS(nm) \otimes \cE(nr),\end{equation}
\begin{equation} \label{super:2}V^{-\frac{m}{2}+r}(\gs\gp_{2n}) \otimes V^{n}(\go\gs\gp_{m|2r}) \rightarrow \cS(nm) \otimes \cE(2nr),\end{equation}
\begin{equation} \label{super:3} V^{-2n+1+ r}(\gs\go_{2m+1}) \otimes V^{-m-\frac{1}{2}}(\go\gs\gp_{r+1|2n}) \rightarrow \cS(n(2m+1)) \otimes \cF(2m+1) \otimes \cF(r(2m+1)).
\end{equation}
In \eqref{super:3}, we have used the homomorphism $V^{-1/2}(\go\gs\gp_{1|2n}) \rightarrow \cS(n) \otimes \cF(1)$, which yields the diagonal map $V^{-k/2}(\go\gs\gp_{1|2n}) \rightarrow \cS(nk) \otimes \cF(k)$ for all $k$.

\subsection{Filtrations}
We recall the canonical decreasing filtration introduced by Li \cite{Li}  that exists on any vertex algebra $\cV$ and is independent of the conformal weight grading. We have
$$F^0(\cV) \supseteq F^1(\cV) \supseteq \cdots,$$ where $F^p(\cV)$ is spanned by elements of the form
$$:\partial^{n_1} a^1 \partial^{n_2} a^2 \cdots \partial^{n_r} a^r:,$$ 
where $a^1,\dots, a^r \in \cV$, $n_i \geq 0$, and $n_1 + \cdots + n_r \geq p$. Note that $\cV = F^0(\cV)$ and $\partial F^i(\cV) \subseteq F^{i+1}(\cV)$. Set $$\text{gr}^F(\cV) = \bigoplus_{p\geq 0} F^p(\cV) / F^{p+1}(\cV),$$ and for $p\geq 0$ let 
$$\sigma_p: F^p(\cV) \ra F^p(\cV) / F^{p+1}(\cV) \subseteq \text{gr}^F(\cV)$$ be the projection. Note that $\text{gr}^F(\cV)$ is a graded commutative algebra with product
$$\sigma_p(a) \sigma_q(b) = \sigma_{p+q}(a_{(-1)} b),$$ for $a \in F^p(\cV)$ and $b \in F^q(\cV)$. We say that the subspace $F^p(\cV) / F^{p+1}(\cV)$ has degree $p$. Note that $\text{gr}^F(\cV)$ has a differential $\partial$ defined by $$\partial( \sigma_p(a) ) = \sigma_{p+1} (\partial a),$$ for $a \in F^p(\cV)$. Also, $\text{gr}^F(\cV)$ has a Poisson vertex algebra structure \cite{Li}; for $n\geq 0$, we define $$\sigma_p(a)_{(n)} \sigma_q(b) = \sigma_{p+q-n} a_{(n)} b.$$ Finally, Zhu's commutative algebra $R_{\cV}$ is isomorphic to the subalgebra $F^0(\cV) / F^1(\cV)\subseteq \text{gr}^F(\cV)$, and $\text{gr}^F(\cV)$ is generated by $R_{\cV}$ as a differential algebra \cite{Li}. 

Next, we recall the notion of a {\it good increasing filtration} $G_{\bullet} \cV$ on a vertex algebra $\cV$. This is a $\mathbb{Z}_{\geq 0}$-filtration
\begin{equation} 0 = G_{-1}\cA \subseteq G_0 \cV \subseteq G_1\cV \subseteq G_2 \cV \subseteq \cdots,  \qquad \cV = \bigcup_{k\geq 0}
G_k \cV, \end{equation} such that $G_0 \cV = \mathbb{C}$, and for all
$a \in G_k \cV$, $b\in G_l \cV$, we have
\begin{equation} \label{goodi} a_{(n)} b\in  G_{k+l} \cV,  \qquad \text{for}\
n<0,\end{equation}
\begin{equation} \label{goodii} a_{(n)} b \in  G_{k+l-1} \cV \qquad \text{for}\
n\geq 0.\end{equation}
Elements $a \in  G_{d} \cV \setminus G_{d-1} \cV$ are said to have degree $d$.

The associated graded algebra $\text{gr}_G(\cV) = \bigoplus_{p\geq 0} G_p \cV / G_{p-1} \cV$ is a
$\mathbb{Z}_{\geq 0}$-graded associative, supercommutative algebra with a
unit $1$ under a product induced by the Wick product on $\cV$. For each $p\geq 1$ we have the projection \begin{equation} \label{projmap}\phi_p: G_p \cV \ra G_p \cV /  G_{p-1} \cV \subseteq \text{gr}_G(\cV).\end{equation} 
Moreover, $\text{gr}_G(\cV)$ has a derivation $\partial$ of degree zero which is induced by the operator $\partial = \frac{d}{dz}$ on $\cV$, and
for each $a\in G_p\cV$ and $n\geq 0$, the operator $a_{(n)}$ on $\cV$
induces a derivation of degree $p-1$ on $\text{gr}_G(\cV)$. These derivations give $\text{gr}_G(\cV)$ the structure of a vertex Poisson algebra.

In fact, if $\cV$ has a grading by conformal weight $\cV = \bigoplus_d \cV[d]$ where $d\in \mathbb{Z}_{\geq 0}$ or $d \in \mathbb{Z}_{\geq 0}$, there is a standard construction of such filtrations \cite{Li}. Suppose that $\cV$ has a strong generating set consisting of fields $\{\alpha^i|\ i \in I\}$ of conformal weight $d_i$. In particular, $\alpha^i_{(n)} \alpha^j$ is a linear combination of normally ordered monomials 
$$ :\partial^{k^1_1} \alpha^{i_1} \cdots \partial^{k^1_{r_1}}\alpha^{i_1} \partial^{k^2_1} \alpha^{i_2} \cdots \partial^{k^2_{r_2}}\alpha^{i_2}
 \cdots \partial^{k^n_1} \alpha^{i_n} \cdots \partial^{k^n_{r_n}} \alpha^{i_n}:$$ where $r_1 d_{i_1} + \cdots+ r_n d_{i_n} \leq p$, then this defines a good increasing filtration on $\cV$.

The following useful observation is due to Arakawa \cite{Ar1}. 
\begin{lemma} \label{Araprop} Let $\cV$ be a conformal vertex algebra, where $\cV_{\Delta}$ is the subspace of conformal weight $\Delta$. Then
$$F^p(\cV_{\Delta}) = G_{\Delta-p} \cV_{\Delta},$$ where $F^p(\cV_{\Delta}) = \cV_{\Delta} \cap F^p(\cV)$ and $G_p(\cV_{\Delta}) = \cV_{\Delta} \cap G_p(\cV)$. Therefore $\text{gr}^F(\cV) \cong \text{gr}_G(\cV)$ as Poisson vertex algebras. In particular, $\text{gr}_G(\cV)$ does not depend on choice of strong generating set used to defined the filtration.
\end{lemma}

Suppose that $\cA$ and $\cB$ are vertex algebras with good increasing filtrations $G_{\bullet} \cA$ and $G_{\bullet} \cB$, and let
$$f: \cA \hookrightarrow \cB,$$ be an injective homomorphism such that $f(G_p \cA) \subseteq G_p\cB$. We then have a homomorphism of Poisson vertex algebras $\text{gr}(f): \text{gr}_G(\cA) \rightarrow \text{gr}_G(\cB)$, but this map need not be injective. For example, consider the embedding  $f: \cH^{\mathbb{Z}_2} \hookrightarrow \cH$, where $\cH^{\mathbb{Z}_2}$ denotes the $\mathbb{Z}_2$-orbifold of the Heisenberg algebra $\cH$. It is well known that $\cH^{\mathbb{Z}_2}$ is strongly generated by the Virasoro field $L$ and a weight $4$ field $W$ \cite{DN}. Then the image $\phi_4(W)$ of $W$ in $\text{gr}_G(\cH^{\mathbb{Z}_2})$ is nilpotent (equivalently, it is nilpotent in $\text{gr}^F(\cH^{\mathbb{Z}_2})$ \cite{AL}), but $\text{gr}_G(\cH)$ is the polynomial ring $\mathbb{C}[\alpha, \partial \alpha, \partial^2 \alpha,\dots]$, so $\phi_4(W)$ lies in the kernel of $\text{gr}(f)$.

The map $f:\cA \rightarrow \cB$ also induces a good increasing filtration $G^{f}_{\bullet} \cA$ on $\cA$ as follows:
$$G^{f}_p \cA = \{ a \in \cA|\ f(a) \in G_p \cB\}.$$ We denote by $\text{gr}_f(\cA)$ the associated graded object. Note that we always have an injective map
$$\text{gr}_f(\cA) \hookrightarrow \text{gr}_G(\cB),$$ but in general $\text{gr}_f(\cA)  \neq \text{gr}_G(\cA)$. In general, it is difficult to determine when $\text{gr}_f(\cA) \cong \text{gr}_G(\cA)$, but the following criterion will suffice for all our examples.

\begin{lemma} \label{lem:filtration} Let $\cA$ and $\cB$ be vertex algebras with good increasing filtrations $G_{\bullet} \cA$ and $G_{\bullet} \cB$, and let $$f: \cA \hookrightarrow \cB$$ be a homomorphism such that $f(G_p\cA) \subseteq G_p\cB$.
Suppose that $$\{\tilde{\omega}_i\} \subseteq \text{gr}_f(\cA) \subseteq \text{gr}_G(\cB)$$ is a generating set for $\text{gr}_f(\cA)$ as a differential algebra with the following properties.
\begin{itemize}
\item[(i)] $\tilde{\omega}_i$ is homogeneous of degree $d_i > 0$ in $\text{gr}_G(\cB)$, for all $i$.
\item[(ii)] There exist fields $\omega_i(z) \in G_{d_i} \cA $ such that $\phi_{d_i}(f(\omega_i)) = \tilde{\omega}_i$, for all $i$.
\end{itemize}

Then 
\begin{enumerate}
\item $\cA$ is strongly generated by $\{\omega_i\}$.
\item $f(\cA) \cap G_p \cB = f(G_p\cA)$. It follows that the map $\text{gr}(f): \text{gr}_G(\cA) \rightarrow \text{gr}_G(\cB)$ is injective, and $\text{gr}_{f}(\cA) \cong \text{gr}_G(\cA)$.
\item If in addition, we have $\text{wt} \ \omega_i = d_i$, then $\text{gr}_G(\cA) \cong \text{gr}^F(\cA)$ as well.
\end{enumerate}
\end{lemma}

 \begin{proof} For (1), let $\alpha \in G_p\cA$ be nonzero. Since $f$ is injective, $f(\alpha) \neq 0$, and by assumption $f(\alpha) \in G_p\cB$. Then there exists $q \leq p$ such that $f(\alpha) \in G_q \cB \setminus  G_{q-1} \cB$, so $\phi_q(f(\alpha))$ can be expressed as a normally ordered polynomial $P(\tilde{\omega}_i)$ in $\tilde{\omega}_i$ and their derivatives, of total degree $q$. Let $\alpha' \in \cA$ be the corresponding normally ordered polynomial in the fields $\omega_i$ and their derivatives. By our assumptions, $\alpha' \in G_q \cA \subseteq G_p \cA$, and $\alpha - \alpha'$ has the property that $f(\alpha - \alpha') \in G_{q-1} \cB$. Continuing this process, we can find some $\alpha'' \in G_{q-1} \cA \subseteq G_p \cA$ which is a normally ordered polynomial in $\omega_i$ and their derivatives, such that $f(\alpha - \alpha ' - \alpha'') = 0$, and hence $\alpha = \alpha' + \alpha''$ since $f$ is injective.

Next, we have $f(G_p\cA) \subseteq f(\cA) \cap G_p \cB$, so to prove (2), we need to show  
$$f(\cA) \cap G_p \cB \subseteq f(G_p\cA).$$ Suppose this holds for all degrees less than $p$. Let $\beta \in ( f(\cA) \cap G_p\cB)\setminus (f(\cA) \cap G_{p-1}\cB)$. Then we can write $\phi_p(\beta)$ as a polynomial $P(\tilde{\omega}_i)$ in $\tilde{\omega}_i$ and their derivatives, of total degree $p$. Let $\alpha \in \cA$ be the corresponding normally ordered polynomial in $\omega_i$ and their derivatives. By assumption, $\alpha \in G_p(\cA)$ and $f(\alpha) - \beta$ lies in $f(\cA) \cap G_{p-1}\cB$. By induction, 
$$ f(\alpha) - \beta \in f(G_{p-1}\cA),$$ that is, $f(\alpha) - \beta   = f(\gamma)$ for some $\gamma \in G_{p-1}\cA$. Then  $\alpha  - \gamma \in G_p\cA$ and $f(\alpha - \gamma) = \beta$, as claimed. 

Finally, (3) is immediate from Lemma \ref{Araprop}. \end{proof}

\begin{remark} The notion of good increasing $\mathbb{Z}_{\geq 0}$-filtrations can easily be modified to include $\frac{1}{2}\mathbb{Z}_{\geq 0}$-filtrations, where $G_{i}\cA \subseteq G_{i+ \frac{1}{2}} \cA$ for all $i \in \frac{1}{2} \mathbb{Z}$. Also, the notion works for vertex superalgebras with no modification, and the statement of Lemma \ref{lem:filtration} continues to hold. \end{remark}

The examples we need are the following. We give $\cS(V)$, $\cE(V)$, and $\cF(V)$ the following good increasing filtrations: \begin{enumerate}
\item $G_{\frac{r}{2}} \cS(V)$ is spanned by the monomials
$$\{:\partial^{k_1} \beta^{x_1} \cdots \partial^{k_s}\beta^{x_s}\partial^{l_1} \gamma^{y'_1}\cdots \partial^{l_t}\gamma^{y'_t}:|\ x_i\in V,\ y'_i\in V^*,\  k_i,l_i\geq 0,\  s+t \leq r\}.$$
\item $G_{\frac{r}{2}} \cE(V)$ is spanned by the monomials
$$\{:\partial^{k_1} b^{x_1} \cdots \partial^{k_s} b^{x_s}\partial^{l_1} c^{y'_1}\cdots \partial^{l_t}c^{y'_t}:|\ x_i\in V,\ y'_i\in V^*,\  k_i,l_i\geq 0,\  s+t \leq r\}.$$
\item $G_{\frac{r}{2}} \cF(V)$  is spanned by the monomials
$$\{:\partial^{k_1} \phi^{x_1} \cdots \partial^{k_s} \phi^{x_s}:|\  x_i\in V, \ k_i \geq 0, \ s \leq r\}.$$
\end{enumerate}
Similarly, for any affine vertex (super)algebra $V^k(\gg)$, we define $G_r V^k(\gg)$ to be the span of all monomials in the generators $X^{\xi}$ and their derivatives of length at most $r$. Then all the homomorphisms $f: V^k(\gg) \rightarrow \cB$ in the previous subsection satisfy $f(G_r V^k(\gg)) \subseteq G_r \cB$.

For $\cV = \cS(V)$, $\cE(V)$, or $\cF(V)$, we have $\text{gr}_G(\cV) \cong \text{gr}^F(\cV)$ by Lemma \ref{Araprop}, so for simplicity of notation, we will always denote $\text{gr}_G(\cV)$ by $\text{gr}(\cV)$.

\section{Coset construction}
\begin{defn} Let $\cA$ be a vertex algebra, and let $\cV$ be a subalgebra. The commutant of $\cV$ in $\cA$, denoted by $\text{Com}(\cV,\cA)$, is the subalgebra of elements $a\in\cA$ such that $[v(z),a(w)] = 0$ for all $v\in\cV$. Equivalently, $v_{(n)} a = 0$ for all $v\in\cV$ and $n\geq 0$.\end{defn}

This was introduced by Frenkel and Zhu \cite{FZ}, and is a standard way to construct new vertex algebras from old ones. If $\cA$ and $\cV$ have Virasoro vectors $L^{\cA}$ and $L^{\cV}$, then $\cC = \text{Com}(\cV,\cA)$ has Virasoro vector $L^{\cA} - L^{\cV}$, and $\cV \otimes \cC\hookrightarrow \cA$ is a conformally embedding. If $\cV$ is a homomorphic image of $V^k(\gg)$, we call $\cC$ an {\it affine coset}, and it is just the invariant space $\cA^{\gg[t]}$.

In this paper, the main examples of affine cosets we study have the following form.
\begin{equation}\label{coset:bg}  \text{Com}(\tilde{V}^{-k}(\gg), \cS(V)) = \cS(V)^{\gg[t]},\end{equation}
\begin{equation} \label{coset:bc} \text{Com}(\tilde{V}^{k}(\gg), \cE(W)) = \cE(W)^{\gg[t]},\end{equation}
\begin{equation} \label{coset:bgbc} \text{Com}(\tilde{V}^{-k+ l}(\gg), \cS(V) \otimes \cE(W)) = (\cS(V) \otimes \cE(W))^{\gg[t]}.\end{equation}

\subsection{The method of arc spaces} We now recall the approach to studying these cosets using the invariant theory of arc spaces that was introduced in \cite{LSS2}. First, with respect to the good increasing filtration $G_{\bullet} \cS(V)$ defined above, we have
\begin{equation} \label{structureofgrs} \text{gr}(\cS(V))\cong \mathbb{C}[\bigoplus_{k\geq 0} (V_k\oplus V^*_k)], \qquad V_k = \{\beta^{x}_{k} |\ x\in V\}, \qquad V^*_k = \{\gamma^{x'}_{k} |\  x'\in V^*\},\end{equation} as commutative algebras. Here $\beta^{x}_{k}$ and $\gamma^{x'}_{k}$ are the images of $\beta^{x}_{(-k-1)} 1 = \frac{1}{k!}\partial^k \beta^{x}$ and $\gamma^{x'}_{(-k-1)} 1 = \frac{1}{k!}\partial^k \gamma^{x'}$ in $\text{gr}(\cS(V))$ under the projection $G_1 \cS(V) \rightarrow G_1\cS(V) / G_0 \cS(V) \subseteq \text{gr}(\cS(V))$, respectively. Note that this notation is slightly different from \cite{LSS2}, and is chosen so that the derivation $\partial$ on $\text{gr}(\cS(V))$ is given by
\begin{equation}\label{actionofpartial} \partial \beta^x_k = (k+1) \beta^x_{k+1}, \qquad \partial \gamma^{x'}_k = (k+1) \gamma^{x'}_{k+1}.\end{equation}
Then \eqref{def:rho12even} induces an action of $\gg[t]$ on $\text{gr}(\cS(V))$ by derivations of degree zero, defined on generators by 
\begin{equation}\label{actiontheta}\xi t^r(\beta^x_i) = \beta^{\rho(\xi)(x)}_{i-r}, \qquad \xi t^r(\gamma^{x'}_i) = \gamma^{\rho^*(\xi)(x')}_{i-r}.\end{equation}
We therefore have an isomorphism of differential algebras $$\text{gr}(\cS(V)) \cong  \mathbb{C}[J_{\infty}(V\oplus V^*)]$$ which is in fact an isomorphism of $\gg[t]$-modules. Here the differential $\partial$ on $\mathbb{C}[J_{\infty}(V\oplus V^*)]$ is normalized as in \cite{LS1,LS2,LS3}.

Next, the inclusion map $f: \cS(V)^{\gg[t]} \hookrightarrow \cS(V)$ induces a filtration $G^f_{\bullet}  \cS(V)^{\gg[t]}$ on $\cS(V)^{\gg[t]}$ where $$G^f_{p}  \cS(V)^{\gg[t]} =  \cS(V)^{\gg[t]} \cap G_p \cS(V).$$ 
We have an induced injective map
\begin{equation}\label{injgamm1} \text{gr}_f(\cS(V)^{\gg[t]})\hookrightarrow \text{gr}(\cS(V))^{\gg[t]} \cong  \mathbb{C}[J_{\infty}(V\oplus V^*)]^{\gg[t]} \cong  \mathbb{C}[J_{\infty}(V\oplus V^*)]^{J_{\infty}(G)},\end{equation} 
where $G$ is a connected Lie group whose action on $\cS(V)$ is infinitesimally generated by the action of $\gg$. In general, \eqref{injgamm1} can fail to be surjective. If generators for $\mathbb{C}[J_{\infty}(V\oplus V^*)]^{J_{\infty}(G)}$ as a differential algebra are known, to check the surjectivity of \eqref{injgamm1} it suffices to check that the generators lie in the image. Finally, there is always a map
\begin{equation}\label{injgamm2}  \mathbb{C}[J_{\infty}((V\oplus V^*)/\!\!/G)] \rightarrow \mathbb{C}[J_{\infty}(V\oplus V^*)]^{J_{\infty}(G)}.\end{equation}
If \eqref{injgamm2} is surjective, the generators of $\mathbb{C}[V\oplus V^*]^G$ will generate $\mathbb{C}[J_{\infty}(V\oplus V^*)]^{J_{\infty}(G)}$ as a differential algebra. If both \eqref{injgamm1} and \eqref{injgamm2} are surjective, we therefore obtain a strong finite generating set for $\cS(V)^{\gg[t]}$ as a vertex algebra.

We next recall how certain cosets of the form $\cE(W)^{\gg[t]}$ and $(\cS(V) \otimes \cE(W))^{\gg[t]}$ given by \eqref{coset:bc} and \eqref{coset:bgbc} can be studied using similar methods. First, given an algebraic group $G$ and finite-dimensional $G$-modules $\tilde U$ and $U$, let $\tilde U_j\cong \tilde U^*$ for $j\geq 0$, and fix a basis $\{x_{1,j},\dots, x_{m,j}\}$ for $\tilde U_j$. Let $S^{\tilde U} = \mathbb{C}[\bigoplus_{j\geq 0} \tilde U_j]$. The map $\mathbb{C}[J_{\infty}(\tilde U)]  \rightarrow S^{\tilde U}$ sending $x_i^{(j)}\mapsto x_{i,j}$ is an isomorphism of differential algebras, where the differential $D$ on $S^{\tilde U}$ is given by $D(x_{i,j}) = (j+1) x_{i,j+1}$.

For $j\geq 0$, let $U_j\cong U^*$ and let $L^U= \bigwedge \bigoplus_{j\geq 0} U_j$. Fix a basis $\{y_{1,j},\dots, y_{n,j}\}$ for $U_j$ and extend the differential on $S^{\tilde U}$ to an even differential $D$ on $S^{\tilde U}\otimes L^U$, defined on generators by $D(y_{i,j}) = (j+1) y_{i,j+1}$. There is an action of $J_{\infty}(G)$ on $S^{\tilde U}\otimes L^U$, and we may consider the invariant ring $(S^{\tilde U}\otimes L^U)^{J_{\infty}(G)}$. 
Let $S^{\tilde U}_0 = \mathbb{C}[{\tilde U}_0] \subseteq S^{\tilde U}$ and $L^U_0 = \bigwedge (\tilde{U}_0)\subseteq L$, and let $\bra (S^{\tilde U}_0\otimes L^U_0)^G\ket$ be the differential algebra generated by $(S^{\tilde U}_0\otimes L^U_0)^G$, which lies in $(S^{\tilde U}\otimes L^U)^{J_{\infty}(G)}$.

Since $G$ acts on the direct sum $\tilde U\oplus U^{\oplus k}$ of $k$ copies of $U$, we have a map 
\begin{equation} \label{kcopiesofv}\mathbb{C}[J_{\infty}(\tilde U\oplus U^{\oplus k} /\!\!/ G)] \rightarrow \mathbb{C}[J_{\infty}(\tilde U\oplus U^{\oplus k})]^{J_{\infty}(G)}.\end{equation}

\begin{thm} \label{oddgen} 
\begin{enumerate}

\item Suppose that \eqref{kcopiesofv} is surjective for all $k\geq 0$. Then $$(S^{\tilde U}\otimes L^U)^{J_{\infty}(G)} =\bra (S^{\tilde U}_0\otimes L^U_0)^G\ket,$$ which generalizes Theorem 7.1 of \cite{LSS2}. In particular, if we fix a generating set $\{\alpha_1,\dots, \alpha_k\}$ for $(S^{\tilde U}_0\otimes L^U_0)^G$, then  $\{\alpha_1,\dots, \alpha_k\}$ generates $(S^{\tilde U}\otimes L^U)^{J_{\infty}(G)}$ as a differential algebra.

\item Suppose that \eqref{kcopiesofv} is an isomorphism for all $k\geq 0$. Let $\{f_1,\dots, f_r\}$ be a generating set for the ideal of relations among $\{\alpha_1,\dots, \alpha_k\}$. Then the ideal of relations among $\{\alpha_1,\dots, \alpha_k\}$ and their derivatives is generated as a differential ideal by $\{f_1,\dots, f_r\}$.
\end{enumerate}
\end{thm}

\begin{proof} For each integer $a \geq 1$, let $S^{U,a}$ be the copy of $S^U$ with generators $z^{a}_{i,j}$ for $i=1,\dots, n$, and $j\geq 0$.  
For $A=\{a_1,\dots, a_d\}$ with $a_i<a_{i+1}$, let 
 $$S^A= S^{\tilde U} \otimes (S^{U,a_1}\otimes \cdots \otimes S^{U,a_d}),\qquad S^A_0= S_0^{\tilde U} \otimes (S_0^{U,a_1}\otimes \cdots \otimes S_0^{U,a_d}).$$
 $S^A$ is $\mathbb Z_{\geq 0}^A$-graded if we give each generator $z^{a_r}_{i,j}$ the multidegree $(0,\dots, 1,\dots, 0)$ with $1$ in the $a_r^{\text{th}}$ position. Let  $T^A$ be the subspace of $S^A$ which is linear in $S^{U,a_1},\dots, S^{U,a_d}$, that is, $T^A$ consists of elements of the form 
\begin{equation} \label{deftildep} \overline{p} = \sum_{|I|,|J|} f_{|I|, |J|}  \otimes (z^{a_1}_{i_1,j_1}\otimes \cdots \otimes z^{a_d}_{i_d,j_d}).\end{equation} In this notation, $|I| = (i_1,\dots, i_d)$ and $|J| = (j_1,\dots, j_d)$ are ordered lists, $f_{|I|, |J|}$ are elements of $S^{\tilde U}$, and the above sum is finite. We will suppress the index of summation and use the shorthand $\overline{p} = \sum f  \otimes (z^{a_1}_{i_1,j_1}\otimes \cdots \otimes z^{a_d}_{i_d,j_d})$. Observe that  
\begin{equation} \label{iso:commalge}S^A \cong \mathbb{C}[J_{\infty}(\tilde U \oplus (\bigoplus_{a=1}^d U^a)],\end{equation} 
which is a commutative algebra. If $B=\{b_1,\dots, b_e\}\subseteq A$ with $b_i< b_{i+1}$, we have the natural embedding 
$$\iota_{B,A}: S^B\to S^A, \qquad f\otimes f^{b_1}\otimes\cdots\otimes f^{b_e}\mapsto f\otimes f^{a_1}\otimes\cdots\otimes f^{a_d},$$
with $f^a=1$ if $a\notin B$. If $C=A\backslash B$, $\bar p\in S^B$ and $\bar q\in S^{C}$ then 
$\iota_{B,A}(\bar p)\iota_{C,A}(\bar q)\in S^A$.

Observe next that the permutation group $\mathfrak{S}_d$ acts on $S^A$, and preserves $T^A$.  For $\sigma \in \mathfrak{S}_d$,
 $$ \sigma(\overline{p}) = \sum f  \otimes (z^{a_1}_{i_{\sigma(1)},j_{\sigma(1)}}\otimes \cdots \otimes z^{a_d}_{i_{\sigma(d)},j_{\sigma(d)}}).$$ Let $T^A_{\sgn} \subseteq T^A$ denote the subspace  $$\{\overline{p} \in T^A|\  \sigma(\overline{p}) = \sgn(\sigma) \overline{p},\ \text{for all}\ \sigma \in \mathfrak{S}_d\}.$$ 
We have a retraction
$$\psi_A: T^A \to T^A_{\sgn},\quad \psi_A(p)=\frac 1 {d!}\sum_{\sigma\in \mathfrak{S}_d} \sgn(\sigma) \sigma(\overline{p}).$$
For each $d \geq 0$, let $S^{\tilde U} \otimes L^{U,d} \subseteq S^{\tilde U}\otimes L^{U}$ be the subspace of degree $d$ in the variables $y_{i,j}$. Then we have linear isomorphisms
$$\phi_A: T^A_{\sgn} \rightarrow S^{\tilde U} \otimes L^{U,d}, \qquad  \sum f  \otimes (z^{a_1}_{i_1,j_1}\otimes \cdots \otimes z^{a_d}_{i_{d},j_{d}}) \mapsto \sum f  \otimes (y_{i_1,j_1}\wedge \cdots \wedge y_{i_{d},j_{d}}).$$
We have $\phi_A D=D\phi_A$ .
Let $\Phi_A=\phi_A\circ\psi_A:T^A\to S^{\tilde U}\otimes L^{U,d}$, which has the following properties:
\begin{enumerate}
	\item $\Phi_A(\phi^{-1}_A(p))=p$, for any $p\in S^{\tilde U}\otimes L^{U,d}$;
	\item $\Phi_A(D\bar p)=D\Phi_A(\bar p)$, for $\bar p\in T^A$.
	\item $\Phi_A(\sigma(\bar p))=\sgn(\sigma) \Phi_A(\bar p)$, for $\bar p\in T^A$.
	\item  Suppose $\sigma\in \mathfrak S _d $ with $a_{\sigma(1)}<\cdots<a_{\sigma(e)} $ and $a_{\sigma(e+1)}<\cdots<a_{\sigma(d)}$. Let $B=\{a_{\sigma(1)}, \cdots, a_{\sigma(e)} \}$,  $C=\{a_{\sigma(e+1)}, \cdots, a_{\sigma(d)} \}$, $\bar p\in T^B$, and $\bar q\in T^{C}$. Then 
	$\Phi_A(\iota_{B,A}(\bar p)\iota_{C,A}(\bar q))=\sgn(\sigma) \Phi_B(\bar p)\Phi_C(\bar q)$.
\end{enumerate}
The following properties of the action of $J_\infty(G)$ on $S^A$ are apparent:
\begin{enumerate}
	\item The action of $J_\infty(G)$ preserves the grading of $\mathbb Z_{\geq 0}^A$. 
	\item $T^A$ is a $J_\infty(G)$-invariant subspace;
	\item the actions of $\mathfrak{S}_d$ and $J_\infty(G)$ commute;
	\item $\iota_{B,A}$, $\phi_A$, $\psi_A$ and $\Phi_A$ are $J_\infty(G)$-equivariant.
\end{enumerate}

Suppose  \eqref{kcopiesofv} is surjective for all $k\geq 0$, so that $(S^A)^{J_\infty(G)}$ is generated by $(S^A_0)^G$ as a differential algebra. 
We can choose homogeneous generators $\{\beta_1,\dots, \beta_u,\beta_{u+1},\dots, \beta_{u+v}\}$ for $ (S^A_0)^G$. In this notation, the first $u$ generators $\{\beta_1,\dots, \beta_u\}$ are linear in each copy of $S^{U,a} $ which appears, and the remaining generators $\beta_{u+1},\dots, \beta_{u+v}$ are at least quadratic in one of these copies. Then there is $B_i\subseteq A$, $1\leq i\leq u$, $\tilde \beta_i\in T^{B_i}$, such that $\beta_i=\iota_{B_i,A}(\tilde \beta_i)$.

For statement (1), if $p\in S^{\tilde U}\otimes L^{U,d}$ is a $J_\infty(G)$-invariant element, then $\phi^{-1}_A(p)$ is $J_\infty(G)$-invariant.  We have
$$\phi^{-1}_A(p)=\sum c\,D^{k_1}\beta_{i_1}\cdots D^{k_l}\beta_{i_l}.$$
with   $1\leq i_j\leq u$ since $\phi^{-1}_A(p)\in T^A_{\sgn} \subseteq T^A$. So
$$\phi^{-1}_A(p)=\sum c\,D^{k_1}\iota_{B_{i_1},A}(\tilde \beta_{i_1})\cdots D^{k_l}\iota_{B_{i_l}, A}(\tilde \beta_{i_l}),$$ with $\cup_j B_{i_j}=A$  and $B_{i_j}\cap B_{i_{j'}}=\emptyset$.  Thus 
$$p=\Phi_A(\phi^{-1}_A(p))=\sum\pm c\, D^{k_1}\Phi_{B_{i_1}}(\tilde \beta_{i_1})\cdots D^{k_l}\Phi_{B_{i_l}}(\tilde \beta_{i_l}).$$ 
Since $\Phi_{B_{i_j}}(\beta_{i_j})$ are $G$-invariant element in $S_0^{\tilde U}\otimes L_0^{U}$, $(S^{\tilde U}\otimes L^U)^{J_{\infty}(G)} =\bra (S^{\tilde U}_0\otimes L^U_0)^G\ket$.

For statement (2), 
we may assume that the generators $\alpha_t$ of $(S^{\tilde U}_0\otimes L^U_0)^G$ are homogeneous of degree $e_t$ in the variables $y_{i,0}$. Note that $\alpha_t$ is odd (respectively even) if and only if $e_t$ is odd (respectively even). Let $P$ be the differential polynomial superalgebra on generators $X_1,\dots, X_k$ and their derivatives, with appropriate parity, so that $(S^{\tilde U}\otimes L^U)^{J_{\infty}(G)}$ is the quotient of $P$ by the homomorphism of differential superalgebras $$\pi_P: P \rightarrow (S^{\tilde U}\otimes L^U)^{J_{\infty}(G)},\qquad  
D^k X_t \mapsto D^k \alpha_t.$$ We will view relations among $\alpha_1,\dots, \alpha_k$ and their derivatives, as elements of $\text{ker} \ \pi_P$. Note that $P$ has a compatible grading if we assign $X_t$ the degree $e_t$, and we denote by $P^e$ the subspace of degree $e$.

For $t = 1,\dots, k$, let 
 \begin{equation}\label{betasigma}\beta_t^\sigma=\sigma(\iota_{C_t,A}(\phi_{C_t}^{-1}(\alpha_t))), \quad  C_t=\{a_1,\dots, a_{e_t}\},\quad \sigma\in  \mathfrak{S}_d.
 	\end{equation}
 Now $\{\beta^\sigma_t|\ t = 1,\dots, k,\ \sigma \in  \mathfrak{S}_d\}\cup \{\beta_1,\dots, \beta_{u+v}\}$ is a set of generators of $(S_0^A)^G$.
We may assume that 
 \begin{equation} \label{signinvt} \Phi_{B_i}(\tilde \beta_i)= 0, \quad 1\leq i\leq u .\end{equation}
This is because $\phi_{B_i}$  is a linear isomorphism, so $\iota_{B_i,A}(\psi_{B_i}(\tilde \beta_i))$ can be generated by $\beta^\sigma_t$ since  $\Phi_{B_i}(\tilde \beta_i)$ is generated by $\alpha_t$.
Therefore we can replace $\beta_i$ by $\beta_i-\iota_{B_i,A}(\psi_{B_i}(\tilde \beta_i))$. 

Let $Q$ be the differential polynomial algebra on generators $Y^\sigma_t$, $Y_1,\dots, Y_{u+v}$, and their derivatives. Then $Q$ is $\mathbb Z_{\geq 0}^A$-graded such that the multidegree of $D^kY_t$ is the same as the multidegree of $D^k\beta_t$. So $(S^A)^{J_{\infty}(G)}$ is the quotient of $Q$ by the homomorphism of differential algebras 
$$\pi_Q: Q  \rightarrow (S^A)^{J_{\infty}(G)}, \qquad D^k Y^\sigma_t \mapsto D^k \beta^\sigma_t,\quad D^k Y_i \mapsto D^k \beta_i.$$ 
Then relations among $\beta^\sigma_t$ and $\beta_i$ and their derivatives are just elements of the kernel of $\pi_Q$. Under the assumption that \eqref{kcopiesofv} is an isomorphism for all $k\geq 1$, $\text{ker}\ \pi_Q$ is generated as a differential ideal by polynomials in $Y^\sigma_t$ and $Y_i$ , i.e., elements with no derivatives, which are homogeneous in the $\mathbb Z_{\geq 0}^A$-grading.

For $B\subseteq A$, let  $g_B\in \mathbb Z_{\geq 0}^A$  with $g_B(a)=0$  if $a\notin B$ and $g_B(a)=1$ if $a\in B$.  Let $Q^B \subseteq Q$ denote the homogeneous subspace of multidegree $g_B$. Clearly $Q^B$ is a differential subspace of $Q$, i.e., it is closed under the action of $D$. Let $e$ be the number of elements in $B$.
We have a linear map of differential spaces 
$$\Psi_{B}: Q^B \rightarrow P^e,$$ defined on monomials as follows:
$\Psi_B(M)=0$ if $M$ is a monomial containing some $Y_i$, and 
$$\Psi_B(D^{k_1}Y^{\sigma_1}_{t_1}\cdots D^{k_l}Y^{\sigma_l}_{t_l})=\sgn(\sigma) D^{k_1}X_{t_1}\cdots D^{k_l} X_{t_l},$$
where $\sigma$ is a permutation of $B$ such that the sequence
$$\sigma(\sigma_1(a_1)),\dots, \sigma(\sigma_1(a_{e_{t_1}})),\dots, \sigma(\sigma_l(a_{e_{t_l}}))$$
is in increasing order.
The map $\Psi_B$ has the following properties:
\begin{enumerate}
	\item $\Psi_B(DM)=D\Psi_B(M)$;
	\item $\pi_P \circ \Psi_B = \Phi_{B} \circ \pi_Q$;
	\item  Let $\sigma\in \mathfrak S _{d} $ with $a_{\sigma(1)}<\cdots<a_{\sigma(e)} $ and $a_{\sigma(e+1)}<\cdots<a_{\sigma(d)}$. Let $B=\{a_{\sigma(1)}, \dots, a_{\sigma(e)} \}$, $C=\{a_{\sigma(e+1)}, \dots, a_{\sigma(d)} \}$, $\bar p\in T^B$, $\bar q\in T^{C}$, $M_1 = \pi_Q^{-1}(\bar p)$, and $M_2 = \pi_Q^{-1}(\bar q)$. Then $\Psi_B (M_1)\Psi_C(M_2)=\sgn(\sigma)\Psi_A(M_1M_2)$.
	\end{enumerate}

Now if $R\in \text{ker} \ \pi_P$ is a relation of $D^k \alpha_t$, we can assume $R$ is homogeneous of degree $d$.
Assume 
$$R=\sum c \, D^{k_1}X_{i_1}\cdots D^{k_l}X_{i_l}.$$
Let $d_j=e_{i_1}+\cdots +e_{i_j}$ and $\sigma_j\in\mathfrak  S_d$ a permutation such that $\sigma_j(i)=d_{j-1}+i \mod d$.
Consider 
$$\tilde R=\sum c  \, D^{k_1}\beta^{\sigma_1}_{i_1}\cdots D^{k_l}\beta^{\sigma_l}_{i_l}\in T^A.$$
\begin{equation}\label{eqnR} \psi_A(\tilde R)=\frac 1{d!}\sum_{\sigma\in \mathfrak S_d}\sgn(\sigma) \sum c  \, D^{k_1}(\beta^{\sigma\sigma_1}_{i_1})\cdots D^{k_l}(\beta^{\sigma \sigma_l}_{i_l}).
\end{equation}
 Replacing  $\beta^\sigma_t$ by $Y^\sigma_t$ in the expression \eqref{eqnR}, we get a polynomial 
 $$r=\frac 1{d!}\sum_{\sigma\in \mathfrak S_d}\sgn(\sigma) \sum c  \, D^{k_1}\sigma(Y^{\sigma\sigma_1}_{i_1})\cdots D^{k_l}(Y^{\sigma \sigma_l}_{i_l})$$
  in $Q^A$ with $\Psi_A(r)= R$ .
Since $$\Phi_A(\tilde R)=\sum c\, D^{k_1}\alpha_{i_1} \cdots D^{k_l}\alpha_{i_l}=\pi_P(R)=0,$$ 
$\pi_Q(r)=\psi_A(\tilde R)=0$.
Thus $r=\sum f_j D^{k_j} r_j$, and $r_j$ is a polynomial in $R^\sigma_t$ and $Y_i$, which are homogeneous with respect to the $\mathbb Z_{\geq 0}^A$-grading. Since $r\in Q^A$, $f_j$ can be chosen to be  homogeneous as well. There are $B_j\subseteq A$, $C_j=A\backslash B_j$ such that $f_j\in Q^{B_j}$ and $r_j\in Q^{C_j}$.
Then $R=\Psi_A(r)=\sum \pm \Psi_{B_i}(f_j)D^{k_i}\Psi_{C_i}(r_j)$, and 
$\Psi_{C_i}(r_j)\in \text{ker}\ \pi_P$ is an element of level zero (i.e. involving no derivatives). This completes the proof. \end{proof}

In view of Theorem \ref{arcspaceinvt}, Theorem \ref{oddgen}(2) applies to the case $G = GL_n$ and $\tilde U =V^{\oplus  k},\  U=V^{\oplus l}$ with $V=\mathbb{C}^n \oplus (\mathbb{C}^n)^*$, and the case $G = Sp_{2n}$ and $\tilde U =V^{\oplus  k} , U=V^{\oplus l}$ with $V=\mathbb{C}^{2n}$. Theorem \ref{oddgen}(1) applies to the case $G=SL_n$ and $\tilde U =V^{\oplus  k}, \ U=V^{\oplus l}$ with $V=\mathbb{C}^n \oplus (\mathbb{C}^n)^*$.

\section{The case $\gg = \gs\gl_n$}
For $n\geq 2$, let $\gg = \gs\gl_n$ and let $V =(\mathbb{C}^n)^{\oplus m}$ be the sum of $m$ copies of the standard module $\mathbb{C}^n$. We regard $V$ as the space of $n \times m$ matrices, so that the homomorphism \eqref{conformalemb:slnslmbg} corresponds to the left and right actions of $\gs\gl_n$ and $\gs\gl_m$ on $V$, respectively. From now on, we use the generators $\beta^{ij}, \gamma^{ij}$ for $i= 1,\dots , n$ and $j = 1, \dots, m$, satisfying $\beta^{ij}(z) \gamma^{kl}(w) \sim \delta_{i,k} \delta_{j,l} (z-w)^{-1}$. The generator of $\cH$ is then $e = \sum_{i=1}^n \sum_{j=1}^m :\beta^{ij} \gamma^{ij}:$.

\begin{thm}  \label{main:sln} For all $n\geq 2$ and $m\geq 1$, $\text{Com}(\tilde{V}^{-m}(\gs\gl_n), \cS(nm))= \cS(nm)^{\gs\gl_{n}[t]}$ is an extension of $\tilde{V}^{-n}(\gg\gl_m)$. It is strongly generated by  
$$X^{ij} = \sum_{k=1}^n :\beta^{ki} \gamma^{kj}:\ \in \tilde{V}^{-n}(\gg\gl_m),$$ together with $2\binom{m}{n}$ additional fields of conformal weight $\frac{n}{2}$, if $m\geq n$:
$$D_{j_1,\dots, j_n} = \left|\begin{matrix} \beta^{1j_1}& \cdots &  \beta^{1j_n} \cr  \vdots  & & \vdots  \cr   \beta^{n j_1}  & \cdots &  \beta^{n j_n} \end{matrix} \right|, \qquad D'_{j_1,\dots, j_n} =  \left|\begin{matrix} \gamma^{1j_1} & \cdots & \gamma^{1j_n} \cr  \vdots  & & \vdots  \cr  \gamma^{n j_1} & \cdots & \gamma^{n j_n} \end{matrix} \right|,$$ for all sets $ \{j_1,\dots, j_n\}\subseteq \{1,\dots, m\}$ of distinct indices.
\end{thm}

\begin{proof} The case $1 \leq m < n$ is given by Theorem 4.1 of \cite{LSS2}; in this case $\tilde{V}^{-n}(\gg\gl_m) = V^{-n}(\gg\gl_n)$, which is simple. The case $n=2$ follows from Theorem 4.3 of \cite{LSS2}. In this case, $\cS(2m)^{\gs\gl_{2}[t]}$ is a homomorphic image of $V^{-2}(\gs\go_{2m})$; see also Proposition 8.1 of \cite{AKMPP}.

In the general case, \eqref{injgamm2} is surjective by Theorem \ref{arcspaceinvt} (3). Since $D_{j_1,\dots, j_n}$ depends only on the $\beta^{ij}$, there are no double contractions with the generators of $\tilde{V}^{-m}(\gs\gl_n)$, so $D_{j_1,\dots, j_n}$ is $\gs\gl_n[t]$-invariant. Similarly, $D'_{j_1,\dots, j_n}$ is $\gs\gl_n[t]$-invariant. Therefore \eqref{injgamm1} is surjective as well, so the generators of the classical invariant ring $\mathbb{C}[(\mathbb{C}^n \oplus (\mathbb{C}^n)^*)^{\oplus m}]^{SL_n}$ give rise to a generating set for $\text{gr}_f(\cS(nm)^{\gs\gl_n[t]})$ as a differential algebra. By Lemma \ref{lem:filtration}, the corresponding fields strongly generate $\cS(nm)^{\gs\gl_n[t]}$ as a vertex algebra.
\end{proof}

The next statement follows immediately from Theorem \ref{main:sln} together with Theorem 3.5 of \cite{LSS2}.

\begin{cor} The Zhu algebra $A(\cS(nm)^{\gs\gl_n[t]})$ is isomorphic to the ring of invariant differential operators $\cD(V)^{SL_n}$ for $V =(\mathbb{C}^n)^{\oplus m}$. \end{cor}

The proof of the next statement is the same as the proof of Theorem 3.5 of \cite{LSS2}. 
\begin{cor} The Zhu commutative algebra $R_{\cS(nm)^{\gs\gl_n[t]}}$ is isomorphic to the ring of invariant polynomial functions $\mathbb{C}[V \oplus V^*]^{SL_n}$.
\end{cor}

Theorem \ref{main:sln} allows us to give a complete description of the singular support of $\cS(nm)^{\gs\gl_n[t]}$.

\begin{cor}  \label{main:slnclassicalfree} \begin{enumerate}
\item For all $n\geq 2$ and $ m \leq n+2$, $\cS(nm)^{\gs\gl_n[t]}$ is classically free. 
\item For $m > n+2$, $\cS(nm)^{\gs\gl_n[t]}$ is not classically free. However, the kernel of the homomorphism \begin{equation} \label{AM:casesln} \mathbb{C}[J_{\infty}((V \oplus V^*)/\!\!/SL_n)] \rightarrow \text{gr}^F(\cS(nm)^{\gs\gl_n[t]}),\end{equation} coincides with the nilradical of $\mathbb{C}[J_{\infty}((V \oplus V^*)/\!\!/SL_n)]$. Therefore $\cS(nm)^{\gs\gl_n[t]}$ is classically free at the level of varieties, and $\text{SS}(\cS(nm)^{\gs\gl_n[t]})$ is just the reduced scheme of $J_{\infty}((V \oplus V^*)/\!\!/SL_n)$.
\item The kernel of \eqref{AM:casesln} is finitely generated as a differential ideal.
\end{enumerate}
\end{cor}

\begin{proof} In the case $m<n$, since $\cS(nm)^{\gs\gl_n[t]} \cong \tilde{V}^{-n}(\gg\gl_m)$ which is a universal affine vertex algebra, there is nothing to prove since all universal affine vertex algebras are classically free. For all $m\geq n$, it follows from Theorem \ref{main:sln} and Lemma \ref{lem:filtration} that $$\text{gr}_f(\cS(nm)^{\gs\gl_n[t]}) \cong \text{gr}(\cS(nm)^{\gs\gl_n[t]}) \cong \text{gr}^F(\cS(nm)^{\gs\gl_n[t]}) \cong  \mathbb{C}[J_{\infty}(V \oplus V^*)]^{J_{\infty}(SL_n)}.$$

For $m \leq n+2$, it follows from Theorem 1.1 (3) that for $V = (\mathbb{C}^n)^{\oplus m}$ the map \eqref{injgamm2} is an isomorphism, so the above map is an isomorphism as well. Finally, the kernel is given explicitly and the statement about the nilradical is proven in Corollary 4.4 of \cite{LS3}.
\end{proof}

There is only one other example in the literature where the kernel of \eqref{eq:ss} is nontrivial and is known to be differentially finitely generated, namely $\text{Vir}_{3,4}$ \cite{AEH}. The examples given by Corollary \ref{main:slnclassicalfree} are the first examples which are not $C_2$-cofinite.

In the cases $m\geq n$, it is an interesting question whether $\cS(nm)^{\gs\gl_n[t]}$ can be identified with vertex algebras appearing in other contexts, such as $\cW$-algebras. We now consider the cases $m = n$ and $m = n+1$.

\subsection{The case $m = n$}
In this case, $\cS(n^2)^{\gs\gl_n[t]}$ is strongly generated by the generators of $V^{-n}(\gs\gl_n)$, the Heisenberg field $e = \sum_{i,j=1}^n :\beta^{ij} \gamma^{ij}:$, together with two fields 
$$D^+ = D_{1,\dots,n}= \left|\begin{matrix} \beta^{11}& \cdots &  \beta^{1n} \cr  \vdots  & & \vdots  \cr   \beta^{n1}  & \cdots &  \beta^{nn} \end{matrix} \right|, \qquad D^-  = D'_{1,\dots,n}=  \left|\begin{matrix} \gamma^{11} & \cdots & \gamma^{1n} \cr  \vdots  & & \vdots  \cr  \gamma^{n1} & \cdots & \gamma^{nn} \end{matrix} \right|,$$ of conformal weight $\frac{n}{2}$. For convenience, we replace the Heisenberg field $e$ with $J = -\frac{1}{n} e$. We have the following OPEs:
\begin{equation} \begin{split} & J(z) J(w) \sim -(z-w)^{-2},
\\ & J(z) D^{\pm}(w) = \pm  D^{\pm}(w)(z-w)^{-1}.
\\ & D^+(z) D^-(w) \sim n! (z-w)^{-n} -  n! J(w)(z-w)^{-n+1} + \dots.
\end{split} \end{equation}
The remaining terms in $D^+_{(n-r-1)} D^-$ for $r = 2,\dots, n-1$ depend only on $J$ and elements of the center of $V^{-n}(\gs\gl_n)$. Note also that the images of the two commuting copies of $V^{-n}(\gs\gl_n)$ in $\cS(n^2)$ have the same center, which is the differential polynomial algebra with generators $\nu_2,\dots, \nu_n$ of weights $2,\dots, n$.

In \cite{CGL}, for $n\geq 2$, the coset 
$$\text{Com}(V^{-n}(\gs\gl_n) \otimes V^{-n}(\gs\gl_n), \cS(n^2)) = \cS(n^2)^{\gs\gl_n[t] \oplus \gs\gl_n[t]}$$ was considered, and it was shown to be freely generated by the Heisenberg field $J$, the fields $D^{\pm}$, and the fields $\nu_2,\dots, \nu_{n-1}$ when $n\geq 3$. The reason that the remaining central element $\nu_n$ is not needed is that there exists a relation of weight $n$:
\begin{equation} \label{relation:omega_n} :D^+ D^-: - P(J, \nu_2,\dots, \nu_n) = 0, \end{equation} where $P$ is a normally ordered polynomial in $\{J, \nu_2,\dots, \nu_n\}$ and their derivatives, and the coefficient of $\nu_n$ is nonzero. Therefore $\nu_n$ can be eliminated from the strong generating set.

In particular, $\cS(n^2)^{\gs\gl_n[t] \oplus \gs\gl_n[t]}$ has the same strong generating type as the universal $\cW$-algebra $\cW^k(\gs\gl_n, f_{\text{subreg}})$ associated to $\gs\gl_n$ with its subregular nilpotent element, which was shown by Genra \cite{G} to be isomorphic to the Feigin-Semikhatov algebra $\cW^{(2)}_n$. Note that in the case $n=2$, $f_{\text{subreg}} = 0$ so $\cW^k(\gs\gl_2, f_{\text{subreg}}) \cong V^k(\gs\gl_2)$. The following result was conjectured in \cite{CGL}, and was proven for $n=2,3,4$ by direct computation.

\begin{thm} \label{casen=m} For all $n\geq 2$, $$\cS(n^2)^{\gs\gl_n[t] \oplus \gs\gl_n[t]} \cong \cW^{-n}(\gs\gl_n, f_{\text{subreg}}).$$
\end{thm}

The idea of the proof is that $\cS(n^2)^{\gs\gl_n[t] \oplus \gs\gl_n[t]}$ and $\cW^{-n}(\gs\gl_n, f_{\text{subreg}})$ share some properties such as strong generating type, graded character and a few features of the OPE relations. Then we will show that there is a unique vertex algebra satisfying these properties, up to isomorphism.

Although the full OPE algebra of $\cW^{-n}(\gs\gl_n, f_{\text{subreg}})$ is given explicitly in \cite{GKu}, we only need the following more qualitative statement, which follows from Theorem 3.14, Proposition 4.2, and Theorem 4.4 of \cite{GKu}.
\begin{lemma} \label{properties:subreg} $\cW^{-n}(\gs\gl_n, f_{\text{subreg}})$ has the following features:
\begin{enumerate} 
\item It is freely generated by a Heisenberg field $J$, fields $G^{\pm}$ of weight $\frac{n}{2}$, and central fields $\omega_2,\dots, \omega_{n-1}$ of weights $2,\dots, n-1$ when $n\geq 3$.
\item These fields satisfy
\begin{equation} \label{eq:FS1} \begin{split} & J(z) J(w) \sim -(z-w)^{-2},
\\ & G^+(z) G^-(w) \sim n! (z-w)^{-n} -  n! J(w)(z-w)^{-n+1} + \dots,
\\ & J(z) G^{\pm}(w) = 
\pm G^{\pm}(w)(z-w)^{-1}.
\end{split} \end{equation}
\item $G^{\pm}$ generate $\cW^{-n}(\gs\gl_n, f_{\text{subreg}})$ as a vertex algebra. Equivalently, for all $i = 2,\dots, n-1$,
$$G^+_{(n-i-1)} G^- = \mu_i \omega_i + \cdots,$$ where $\mu_i \neq 0$ and the remaining terms are normally ordered monomials in $J,\omega_2,\dots, \omega_{i-1}$ and their derivatives.
\end{enumerate}
\end{lemma}

\begin{lemma} $\cS(n^2)^{\gs\gl_n[t] \oplus \gs\gl_n[t]}$ is generated by $D^{\pm}$ as a vertex algebra. Equivalently, for all $i = 2,\dots, n-1$,
$$D^+_{(n-i-1)} D^- = \lambda_i \nu_i + \cdots,$$ where $\lambda_i \neq 0$ and the remaining terms are normally ordered monomials in $J,\nu_2,\dots, \nu_{i-1}$ and their derivatives.
\end{lemma}

\begin{proof} First, the projection $\phi_{n}(:D^+ D^-:)$ onto the component of degree $n$ (which has degree $2n$ in the variables $\beta^{ij}_k$ and $\gamma^{ij}_k$), actually lies in the subalgebra
$$\mathbb{C}[\beta^{ij}_0, \gamma^{ij}_0]^{GL_n \times GL_n} \cong \cZ(\gg\gl_n),$$ where $\cZ(\gg\gl_n)$ denotes the center of $U(\gg\gl_n)$. As such, we can write $\phi_{n}(:D^+ D^-:)$ as a polynomial in the Casimirs $J,  \nu_2, \dots, \nu_n$. Here by abuse of notation, we do not distinguish between the fields $J, \nu_2,\dots, \nu_n \in \cS(n^2)^{\gs\gl_n[t]\oplus \gs\gl_n[t]}$ and their images in $\cZ(\gg\gl_n)$. As above, the coefficient of $\nu_n$ is nonzero.

Let $x^{ij} = \phi_1(X^{ij})$ be the image of $X^{ij}$ in the degree $1$ part of $\text{gr}(\cS(n^2))$, and observe that $\nu_n$ can be characterized as the only monomial of degree $n$ in $J, \nu_2,\dots, \nu_n$ which contains the term $x^{12} x^{23} \cdots x^{n-1,n} x^{n,1}$ with nonzero coefficient.

Now for all $l$ with $2< l < n$, the projection $\phi_{l}(D^+_{(n-l-1)} D^-)$ will contain the degree $l$ element 
 $$d_{i_1,\dots, i_l} d'_{i_1,\dots, i_l} = \phi_l(:D_{i_1,\dots, i_l} D'_{i_1,\dots, i_l}:)$$ with coefficient $\pm (n-l)!$. As above, $d_{i_1,\dots, i_l} d'_{i_1,\dots, i_l}$ can be expressed as a polynomial in the generators  $J', \nu_2',\dots, \nu'_l$ of $\cZ(\gg\gl_l)$, and the coefficient of $\nu'_l$ is nonzero.

Next, the embedding $i: \gg\gl_l\hookrightarrow \gg\gl_n$ gives a restriction map $i^*: \gg\gl_n^*\rightarrow \gg\gl_l^*$ and induced surjection
$$i^*:\cZ(\gg\gl_n) \cong S(\gg\gl_n)^{GL_n} \rightarrow S(\gg\gl_l)^{GL_l} \cong \cZ(\gg\gl_l).$$
This map is also injective when restricted to the component of degree $l$. In particular, given a polynomial $p(J, \nu_2,\dots, \nu_l)$ of degree $l$ in $S(\gg\gl_n)^{GL_n}$, the coefficient of $\nu_l$ in $p$ nonzero if and only if the coefficient of $x^{12} x^{23} \cdots x^{l-1,l} x^{l,1}$ in $i^*(p)$ is nonzero. Since $x^{12} x^{23} \cdots x^{l-1,l} x^{l,1}$ has nonzero coefficient in $d_{i_1,\dots, i_l} d'_{i_1,\dots, i_l}$ and cannot appear in any other terms in $i^*(\phi_l (D^+_{(n-l-1)} D^-))$, $\nu_l$ must appear with nonzero coefficient in $D^+_{(n-l-1)} D^-$, as claimed.
\end{proof}

{\it Proof of Theorem \ref{casen=m}.}
It suffices to prove that any two vertex algebras satisfying the properties of Lemma \ref{properties:subreg} must be isomorphic. So let $\cA$ be another such vertex algebra which is freely generated by a Heisenberg field $\tilde{J}$, fields $\tilde{G}^{\pm}$ of weight $\frac{n}{2}$, and central fields $\tilde{\omega}_2,\dots, \tilde{\omega}_{n-1}$ with OPE relations
\begin{equation}\label{partialOPE:generalA}  \begin{split} &\tilde{J}(z) \tilde{J}(w) \sim -n^2 (z-w)^{-2},\qquad \tilde{G}^+(z) \tilde{G}^-(w) \sim n! (z-w)^{-n} - n! \tilde{J} (w)(z-w)^{-n+1} + \dots,
\\ & J(z) \tilde{G}^{\pm}(w) =  \pm \tilde{G}^{\pm}(w) (z-w)^{-1}.
\end{split} \end{equation}
Suppose furthermore that for all $2\leq r \leq n-1$, $\tilde{G}^+_{(n-r-1)} \tilde{G}^{-} = \tilde{\mu}_{r} \tilde{\omega}_r + \cdots$, where the remaining terms depend only on $J, \tilde{\omega}_2,\dots, \tilde{\omega}_{n-1}$ and their derivatives, and $\tilde{\mu}_r\neq 0$.

To show that $ \cW^{-n}(\gs\gl_n, f_{\text{subreg}}) \cong \cA$, it suffices to show that we can choose new generators $\tilde{\omega}'_2,\dots, \tilde{\omega}'_{n-1}$ for the central algebra generated by $\tilde{\omega}_2,\dots, \tilde{\omega}_{n-1}$ such that OPEs agree, that is, the map $\cW^{-n}(\gs\gl_n, f_{\text{subreg}}) \rightarrow  \cA$ defined by
\begin{equation} \label{correspondence} J \mapsto \tilde{J}, \qquad G^{\pm} \mapsto \tilde{G}^{\pm}, \qquad \omega_i \mapsto \tilde{\omega}'_i,\end{equation}
preserves OPEs and is therefore an isomorphism of vertex algebras since both sides are freely generated by the given fields.

For $2\leq i\leq n-1$, let $\cM^i \subseteq \cW^{-n}(\gs\gl_n, f_{\text{subreg}})$ denote the set of all normally ordered monomials of weight $i$ in the variables $\omega_2,\dots, \omega_{n-1}$ and their derivatives, but not including the term $\omega_i$ itself. Let $\cJ^i$ be the set of normally ordered monomials in $J$ and its derivatives of weight $i$.

Similarly, let $\tilde{\cM}^i$ and $\tilde{\cJ}^i$ be the corresponding sets in $\cA$ consisting of normally ordered monomials in $\tilde{\omega}_2,\dots, \tilde{\omega}_{n-1}$ and their derivatives not including $\tilde{\omega}_i$ itself, and normally ordered monomials in $\tilde{J}$ and its derivatives, respectively.

For $1<r<i$, let $\cM^r \cJ^{i-r}$ denote the monomials of weight $i$ which are a product of some $M \in \cM^r$ and $N \in \cJ^{i-r}$, and similarly define $\tilde{\cM}^r \tilde{\cJ}^{i-r}$. Then 
 \begin{equation} \begin{split} & G^+_{(n-i-1)} G^- = \mu_i \omega_i + \cdots,
\\ & \tilde{G}^+_{(n-i-1)} \tilde{G}^- = \tilde{\mu}_i \tilde{\omega}_i + \cdots, \end{split} \end{equation} where the remaining terms are linear combinations of elements of $\cM^i$, $\cJ^i$, and $\cM^r \cJ^{i-r}$ in the first line, and of elements of $\tilde{\cM}^i$, $\tilde{\cJ}^i$, and $\tilde{\cM}^r \tilde{\cJ}^{i-r}$ in the second line, respectively.

It is easy to check that the Jacobi identities of type $(J, G^+, G^-)$ together with \eqref{partialOPE:generalA} uniquely determine all coefficients of all monomials in $\cJ^i$, for all $i = 2,\dots, n-1$. Also, for all monomials of the form $MN \in \tilde{\cM}^r \tilde{\cJ}^{i-r}$ with $M \in \tilde{\cM}^{r}$ and $N \in \tilde{\cM}^{i-r}$, the coefficient of $MN$ appearing in $\tilde{G}^+_{(n-i-1)} \tilde{G}^-$ is uniquely determined in terms of the coefficient of $M$ in $\tilde{G}^+_{(n-r-1)} \tilde{G}^-$. Note, however, that the Jacobi identities do not provide any restrictions on the coefficients of monomials in $\cM^i$.

Recall that in $\cW^{-n}(\gs\gl_n, f_{\text{subreg}})$ and $\cA$, we have respectively,
\begin{equation} \begin{split} & G^+_{(n-3)} G^- = \mu_2 \omega_2 + \cdots,
\\ & \tilde{G}^+_{(n-3)} \tilde{G}^- = \tilde{\mu}_2 \tilde{\omega}_2 + \cdots,\end{split} \end{equation} Here the remaining terms lie in $\cJ^2$ and $\tilde{\cJ}^2$, respectively, and hence are completely determined. Therefore they must agree under \eqref{correspondence}.. 

We define $$\tilde{\omega}'_2 = \frac{\tilde{\mu}_2}{\mu_2} \tilde{\omega}_2 \in \cA.$$ Then 
$$\tilde{G}^+_{(n-3)} \tilde{G}^- = \tilde{\mu}_2 \frac{\mu_2}{\tilde{\mu}_2} \tilde{\omega}'_2 + \cdots = \mu_2 \tilde{\omega}'_2 + \cdots,$$ which now agrees with $G^+_{(n-3)} G^-$ under \eqref{correspondence}.

Inductively, we assume that for all $i \leq r$, we have defined new variables $\tilde{\omega}'_1,\dots, \tilde{\omega}'_r \in \cA$ which commute with $\tilde{J}$, such that all structure constants in $\tilde{G}^+_{(n-i-1)} \tilde{G}^-$ agree with those in $G^+_{(n-i-1)} G^-$. In particular, for $i\leq r+1$, the set $\tilde{\cM}^i$ has been replaced with the set $(\tilde{\cM}')^i$ of normally ordered monomials of weight $i$ in the new variables $\tilde{\omega}'_1,\dots, \tilde{\omega}'_r$. 

Now consider $G^+_{(n-r-2)} G^-$, which has weight $r+1$, and write
$$G^+_{(n-r-2)} G^- = \mu_{r+1} \omega_{r+1} + \sum_{j} a_j M^{r+1}_j + \cdots,$$ 
$$\tilde{G}^+_{(n-r-2)} \tilde{G}^- = \tilde{\mu}_{r+1} \tilde{\omega}_{r+1} + \sum_{j} b_j \tilde{M}^{r+1}_j + \cdots,$$
where $a_j, b_j$ are constants and the remaining terms which either lie in $\cJ^{r+1}$ or $\cM^{a} \cJ^{r+1-a}$ (respectively $\tilde{\cJ}^{r+1}$ or $(\tilde{\cM}')^{a} \tilde{\cJ}^{r+1-a}$), are completely determined by previous data.

We now define $$\tilde{\omega}'_{r+1} = \frac{\tilde{\mu}_{r+1}}{\mu_{r+1}} \bigg(\tilde{\omega}_{r+1} + \sum_j \frac{b_j}{\tilde{\mu}_{r+1}}\tilde{M}^{r+1}_j - \sum_j \frac{a_j}{\tilde{\mu}_{r+1}} \tilde{M}^{r+1}_j\bigg).$$ Then
\begin{equation} \begin{split} \tilde{G}^+_{(n-r-2)} \tilde{G}^- & = \tilde{\mu}_{r+1} \bigg( \frac{\mu_{r+1}}{\tilde{\mu}_{r+1}}  \tilde{\omega}'_{r+1} - \sum_j \frac{b_j}{\tilde{\mu}_{r+1}} \tilde{M}^{r+1}_j + \sum_j  \frac{a_j}{\tilde{\mu}_{r+1}}\tilde{M}^{r+1}_j\bigg) + \sum_{j} b_j \tilde{M}^{r+1}_j + \cdots,
\\ & = \mu_{r+1} \tilde{\omega}'_{r+1} + \sum_j  a_j \tilde{M}^{r+1}_j + \cdots. \end{split} \end{equation}
Therefore with this change of variables, the structure constants in $\tilde{G}^+_{(n-r-2)} \tilde{G}^-$ now agree with those of $G^+_{(n-r-2)} G^-$ with respect to monomials in $\{\tilde{J}, \tilde{\omega}'_2,\dots, \tilde{\omega}'_{r+1}\}$ and $\{J, \omega_2,\dots,\omega_{r+1}\}$, respectively. Inductively, we can choose new variables $\tilde{\omega}'_2,\dots, \tilde{\omega}'_{n-1}$ where this holds, and then by the above discussion the full OPE algebras agree. $\qquad \qquad \qquad \qquad \qquad \qquad \qquad \   \Box$
 
\bigskip

Combining Theorem \ref{casen=m} with Theorem 8.1 of \cite{CGL}, we immediately obtain
\begin{cor} The Zhu algebra $A(\cW^{-n}(\gs\gl_n, f_{\text{subreg}}))$ is isomorphic to the ring of invariant differential operators $\cD(n^2)^{SL_n \times SL_n}$, where $\cD(n^2)$ denotes the Weyl algebra of rank $n^2$. 
\end{cor} 

This allows the irreducible positive energy modules for $\cW^{-n}(\gs\gl_n, f_{\text{subreg}})$ to be studied via the representation theory of $\cD(n^2)^{SL_n \times SL_n}$. Even though $\cD(m)$ has no nontrivial finite-dimensional modules for all $m\geq 1$, it turns out that $\cD(n^2)^{SL_n \times SL_n}$ admits a class of finite-dimensional irreducible modules which were described in \cite{CGL} for $n=2, 3, 4$. The corresponding irreducible $\cW^{-n}(\gs\gl_n, f_{\text{subreg}})$-modules will have finite-dimensional graded components, and it is an interesting problem to classify them.

We have one more consequence of Theorem \ref{casen=m}.

\begin{cor} \label{cor:n=m} For all $n\geq 2$, $$\cS(n^2)^{\gs\gl_n[t]} \cong (V^{-n}(\gs\gl_n) \otimes \cW^{-n}(\gs\gl_n, f_{\text{subreg}})) / \cI,$$ where $\cI$ is the ideal generated by
 \begin{equation*}  \left\{
\begin{array}{ll} \{\omega_k - \nu_k|\ k = 2,\dots, n-1\} \cup \{:G^+ G^-:\  - P(J, \nu_2,\dots, \nu_n)\}, & n >2
\smallskip
\\ \{:G^+ G^-:\  - P(J, \nu_2)\} & n=2. \\
\end{array} 
\right. \end{equation*} 
Here $P$ is the same normally ordered polynomial appearing in \eqref{relation:omega_n}.
\end{cor}

\begin{proof} Since $\text{Com}(V^{-n}(\gs\gl_n), \cS(n^2)^{\gs\gl_n[t]}) \cong  \cW^{-n}(\gs\gl_n, f_{\text{subreg}})$, we have a surjective homomorphism 
$$\phi:V^{-n}(\gs\gl_n) \otimes \cW^{-n}(\gs\gl_n, f_{\text{subreg}}) \rightarrow \cS(n^2)^{\gs\gl_n[t]},$$ and $\text{ker}\ \phi$ clearly contains $\cI$. The fact that $\cI = \text{ker}\ \phi$ is apparent from the isomorphism
$$\text{gr}_f(\cS(n^2)^{\gs\gl_n[t]}) \cong \text{gr}(\cS(n^2))^{\gs\gl_n[t]} \cong \mathbb{C}[J_{\infty}(V\oplus V^*)]^{J_{\infty}(SL_n)} \cong \mathbb{C}[J_{\infty}((V\oplus V^*)/\!\!/SL_n)],$$ for $V = \mathbb{C}^{n^2}$, since $(V\oplus V^*)/\!\!/SL_n$ is a hypersurface with defining relation corresponding to \eqref{relation:omega_n}.
\end{proof}

\subsection{The case $m = n+1$}

For $n\geq 2$, it follows from Theorem \ref{main:sln} that $\cS(n(n+1))^{\gs\gl_n[t]}$ is an extension of $\tilde{V}^{-n}(\gg\gl_{n+1}) = \cH \otimes \tilde{V}^{-n}(\gs\gl_{n+1})$ by $2(n+1)$ fields of weight $\frac{n}{2}$ of the form 
$$D^{+,r} := D_{i_1,\dots, i_{r-1},\widehat{i_r}, i_{r+r},\dots, i_{n+1}},\qquad D^{-,r} :=D'_{i_1,\dots, i_{r-1},\widehat{i_r}, i_{r+r},\dots, i_{n+1}},\qquad r = 1,\dots, n+1.$$ Note that $\tilde{V}^{-n}(\gs\gl_{n+1}) = V^{-n}(\gs\gl_{n+1})$ since this algebra is simple by Theorem 0.2.1 of \cite{GK}; this follows from the fact that for $k = -n$ the shifted level $k + h^{\vee} = 1$. Also, the fields $\{D^{+,r}\}$ and $\{D^{-,r}\}$ transform under $\gg\gl_{n+1}$ as the standard module $\mathbb{C}^{n+1}$ and the dual module $(\mathbb{C}^{n+1})^*$, respectively.

We recall the family of $\cW$-algebras studied by the first author and Creutzig in \cite{CL3} which were called {\it hook-type}. For positive integers $n,m$, recall that $\gs\gl_{n+m}$ has the decomposition
$$\mathfrak{sl}_{n+m} = \mathfrak{sl}_n \oplus \mathfrak{gl}_m \oplus \big(\mathbb{C}^n \otimes (\mathbb{C}^m)^* \big)\ \oplus \big(( \mathbb{C}^n)^* \otimes \mathbb{C}^m\big).$$
Recall next that conjugacy classes of nilpotents $f\in \mathfrak{sl}_N$ correspond to partitions of $N$. 
For $N = n+m$, let $f_{n,m} \in \mathfrak{sl}_{n+m}$ be the nilpotent which is principal in $\mathfrak{sl}_n$ and trivial in $\mathfrak{gl}_m$. It corresponds to the hook-type partition $N = n + 1 + \dots + 1$, and the corresponding $\cW$-algebra $\cW^k(\gs\gl_{n+m}, f_{n,m})$ is a common generalization of the affine vertex algebra $V^k(\gs\gl_{n+1})$, the principal $\cW$-algebra $\cW^k(\gs\gl_n)$, the subregular $\cW$-algebra $\cW^k(\gs\gl_{n+1}, f_{\text{subreg}})$, and the minimal $\cW$-algebra $\cW^k(\gs\gl_{m+2}, f_{\text{min}})$.

It is convenient to replace $k$ by the shifted level $\psi = k+n+m$, and following \cite{CL3}, we use the notation
$\cW^{\psi}(n,m) = \cW^k(\gs\gl_{n+m}, f_{n+m})$. It has the following features:
\begin{enumerate}
\item Virasoro element $L^{\cW}$ and affine subalgebra $V^{\psi-m-1}(\gg\gl_m)$.
\item For $n\geq 3$, additional even fields $\omega_3,\dots, \omega_n$ of conformal weights $3,\dots, n$ which commute with $V^{\psi-m-1}(\gg\gl_m)$.
\item Fields $\{G^{\pm,r}|\ r = 1,\dots, m\}$ of weight $\frac{n+1}{2}$ such that $\{G^{+,r}\}$ and $\{G^{-,r}\}$ transform in the standard and dual $\gg\gl_m$-modules, and are primary with respect to the action of $V^{\psi -m-1}(\gg\gl_m)$.
\item Freely generated by the generators of $V^{\psi - m - 1}(\gg\gl_m)$ together with $L, \omega_3,\dots, \omega_n, G^{\pm,r}$.
\end{enumerate}

The coset $\cC^{\psi}(n,m) = \text{Com}(V^{\psi-m-1}(\gg\gl_m), \cW^{\psi}(n,m))$ has Virasoro element $L^{\cC} = L^{\cW} - L^{\gg\gl_m}$ of central charge $$c =  -\frac{(n \psi  - m - n -1) (n \psi - \psi - m - n +1 ) (n \psi +  \psi  -m - n)}{(\psi -1) \psi}.$$ Note that when $\psi = \frac{m+n}{n+1}$, the central charge of $L^{\cC}$ is zero. If $m+n$ and $n+1$ are relatively prime, this is a boundary admissible level for $\gs\gl_{n+m}$. These algebras have been studied by Creutzig in the setting of Argyres-Douglas theories in \cite{C}, as well as in \cite{ACGY}. It is known that $L^{\cC}$ is a singular vector in $\cW^{\psi}(n,m)$, and we have a conformal embedding $L_{\psi-m-1}(\gg\gl_m) \hookrightarrow \cW_{\psi}(n,m)$, where $\cW_{\psi}(n,m)$ denotes the simple quotient \cite{ACLM}. It is expected that $\{G^{\pm,r}\}$ survive in $\cW_{\psi}(n,m)$, which is then a nontrivial extension of $L_{\psi-m-1}(\gg\gl_m)$. Since $\omega_3,\dots, \omega_n$ commute with $L^{\gg}$, they must vanish in $\cW_{\psi}(n,m)$, so that $\cW_{\psi}(n,m)$ would then be strongly generated by the generators of $L_{\psi-m-1}(\gg\gl_m)$ together with $\{G^{\pm,r}\}$.

We now specialize to the case $\cW^2(n-1, n+1) = \cW^{2-2n}(\gs\gl_{2n}, f_{n-1,n+1})$. Then $L^{\cC}$ is a singular vector, and $V^{-n}(\gg\gl_{n+1})$ is conformally embedded in $\cW_{2-2n}(\gs\gl_{2n}, f_{n-1,n+1})$.

\begin{conj} For all $n\geq 2$, $\cS(n(n+1))^{\gs\gl_n[t]} \cong \cW_{2-2n}(\gs\gl_{2n}, f_{n-1,n+1})$.
\end{conj}

\subsection{The structure of $\text{Com}(L_{r}(\gs\gl_n), \cE(nr))$} 
As above, let $W =(\mathbb{C}^n)^{\oplus r}$, which we regard as the space of $n \times r$ matrices. Then $\cE(W) \cong \cE(nr)$ and the homomorphism $L_r(\gs\gl_n) \otimes L_n(\gs\gl_r) \otimes \cH \rightarrow \cE(nr)$ given by \eqref{conformalemb:slnslmbc} corresponds to the left and right actions of $\gs\gl_n$ and $\gs\gl_r$ on $W$. We use the generators $b^{ij}, c^{ij}$ for $i= 1,\dots , n$ and $j = 1, \dots, r$, satisfying $b^{ij}(z) c^{kl}(w) \sim \delta_{i,k} \delta_{j,l} (z-w)^{-1}$. The generator of $\cH$ is then $e = -\sum_{i=1}^n \sum_{j=1}^m :b^{ij} c^{ij}:$.

\begin{thm} \label{main:slnE} For all $n\geq 2$ and $r \geq 1$, the coset $$\text{Com}(L_{r}(\gs\gl_n), \cE(nr))$$ is an extension of $ L_n(\gg\gl_r) = \cH \otimes L_n(\gs\gl_r)$, and is strongly generated by the fields
\begin{equation} \label{gen:vnslr} Y^{st} = \sum_{k=1}^n :b^{ks} c^{kt}:\  \in L_n(\gg\gl_r),\qquad s,t = 1,\dots, r,\end{equation} together with following additional strong generators of weight $\frac{n}{2}$:
\begin{equation} \label{gen:oddDD'} D_{i_1,\dots, i_n},\qquad D'_{i_1,\dots, i_n}.\end{equation} Here $i_1,\dots, i_n$ are elements from the set $\{1,\dots, n\}$ which need not be distinct, and $D_{i_1,\dots, i_n}, D'_{i_1,\dots, i_n}$ are given by the same formula as the determinants in Theorem \ref{main:sln}, but without the signs.
\end{thm}

\begin{proof} We have $\text{gr}(\cE(nr)) \cong  L^V = \bigwedge \bigoplus_{j\geq 0} V_j$ where $V_j \cong \big(\mathbb{C}^n \oplus (\mathbb{C}^n)^*\big)^{\oplus r} \cong V$ for all $j$. By Theorems \ref{arcspaceinvt} and \ref{oddgen} (1), $\text{gr}(\cE(nr))^{\gs\gl_n[t]}$ is generated as a differential algebra by elements in the subalgebra 
$$\bra (L^V_0)^{SL_n}\ket \cong  \big(\bigwedge (V \oplus V^*)\big)^{SL_n}.$$ By classical invariant theory, this is generated by the quadratics corresponding to a pairing of a copy of $\mathbb{C}^n$ with a copy of $(\mathbb{C}^n)^*$, together with determinants (without signs) that depend on $n$ copies of $\mathbb{C}^n$ which need not be distinct, or on $n$ copies of $(\mathbb{C}^n)^*$ which need not be distinct. It is easy to see that the corresponding fields in $\cE(nr)$ actually are $\gs\gl_n[t]$-invariant, and that the quadratic fields generate the affine vertex algebra $L_{n}(\gg\gl_{r})$. Therefore the map 
$$\text{gr}_f(\cE(nr))^{\gs\gl_n[t]}) \hookrightarrow \text{gr}(\cE(nr))^{\gs\gl_n[t]}$$ is an isomorphism, which completes the proof.
\end{proof}

\subsection{The structure of $\text{Com}(V^{-m+r}(\gs\gl_n), \cS(nm) \otimes \cE(nr))$}
Recall the homomorphism
$$V^{-m+r}(\gs\gl_n) \otimes V^{n}(\gg\gl_{r|m}) \rightarrow \cS(nm) \otimes \cE(nr),$$ whose image is conformally embedded.

\begin{thm}  \label{main:slnSE} For all $n\geq 2$ and $m,r \geq 1$, the coset $$\text{Com}(V^{-m+r}(\gs\gl_n), \cS(nm) \otimes \cE(nr))$$ is an extension of $\tilde{V}^{n}(\gg\gl_{r|m})$. It is strongly generated by the generators of $\tilde{V}^{n}(\gg\gl_{r|m})$, namely
\begin{equation} \label{gen:vnsl(r|m)} X^{ij} = \sum_{k=1}^n :\beta^{ki} \gamma^{kj}:,\qquad Y^{st} = \sum_{k=1}^n : b^{ks} c^{kt}:,\qquad E^{is} = \sum_{k=1}^n :\beta^{ki} c^{ks}:,\qquad F^{sj} = \sum_{k=1}^n :b^{ks} \gamma^{kj}:,\end{equation} together with the fields
\begin{equation} \label{gen:mixedDD'} D_{i_1,\dots, i_s; j_{s+1},\dots, j_n},\qquad D'_{i_1,\dots, i_s; j_{s+1},\dots, j_n}.\end{equation} Here $s = 0,1,\dots, n$, $i_1,\dots, i_s$ are distinct elements from the set $\{1,\dots, m\}$, and $\{j_{s+1}, \dots, j_n\}$ are elements from the set $\{1,\dots, r\}$, not necessarily distinct. The fields $D_{i_1,\dots, i_s; j_{s+1},\dots, j_n}$ and $D'_{i_1,\dots, i_s; j_{s+1},\dots, j_n}$ are supersymmetric analogous of determinants with certain sign changes when the variables are odd. In the extreme case $s = n$, they are ordinary determinants in the even variables, and in the extreme case $s = 0$, all signs are positive.
\end{thm}

\begin{proof} We have $\text{gr}(\cS(nm) \otimes \cE(nr)) \cong S^V \otimes L^W$, where 
\begin{equation*} \begin{split} & S^V = \mathbb{C}[\bigoplus_{j\geq 0} V_j \oplus V^*_j],\qquad V_j \cong (\mathbb{C}^n)^{\oplus m} \cong V,
\\ & L^W = \bigwedge \bigoplus_{j\geq 0} (W_j \oplus W^*_j), \qquad W_j \cong (\mathbb{C}^n)^{\oplus r} \cong W.\end{split} \end{equation*}
By Theorems \ref{arcspaceinvt} and \ref{oddgen} (1), $\text{gr}(\cS(nm) \otimes \cE(nr))^{\gs\gl_n[t]}$ is generated as a differential algebra by elements in the subalgebra 
$$\bra (S^V_0\otimes L^W_0)^{SL_n}\ket \cong \big(\mathbb{C}[V\oplus V^*] \otimes \bigwedge (W \oplus W^*)\big)^{SL_n}.$$ By classical invariant theory, this is generated by the quadratics corresponding to a pairing of either an even or odd copy of $\mathbb{C}^n$, and an even or odd copy of $(\mathbb{C}^n)^*$, together with determinants (with appropriate signs) that depend on $n$ copies of $\mathbb{C}^n$ which can be either even or odd, with the even ones distinct, or on $n$ copies of $(\mathbb{C}^n)^*$ which can be either even or odd, with the even ones distinct. It is easy to see that the corresponding fields in $\cS(nm) \otimes \cE(nr)$ actually are $\gs\gl_n[t]$-invariant, and that the quadratic fields generate the affine vertex superalgebra $V^{n}(\gg\gl_{r|n}$). Therefore the map 
$$\text{gr}_f((\cS(nm) \otimes \cE(nr))^{\gs\gl_n[t]}) \hookrightarrow \text{gr}(\cS(nm) \otimes \cE(nr))^{\gs\gl_n[t]}$$ is an isomorphism, which completes the proof.
\end{proof}

\section{The case $\gg = \gg\gl_n$}
In this section, we consider the structure of $\cS(nm)^{\gg\gl_n[t]}$, $\cE(nr)^{\gg\gl_n[t]}$, and $(\cS(nm) \otimes \cE(nr))^{\gg\gl_n[t]}$. We first study $\cS(nm)^{\gg\gl_n[t]}$, and to motivate our results, we begin by recalling the case $n=1$. In this case, $V^{-m}(\gg\gl_n)$ is replaced with the Heisenberg algebra $\cH$ generated by $e$, which satisfies $e(z) e(w) \sim -m (z-w)^{-2}$, and we write $\cS(m)^{\gg\gl_1[t]} = \text{Com}(\cH, \cS(m))$.

\begin{thm} \label{thm:casegl1}  \begin{enumerate}
\item In the case $m=1$, $\text{Com}(\cH, \cS(1))$ is isomorphic to the simple Zamolodchikov $\cW_3$ algebra with central charge $c = -2$.
\item In the case $m=2$, $\text{Com}(\cH, \cS(2))$ is isomorphic to the simple rectangular $\cW$-algebra of $\gs\gl_4$ at level $-\frac{5}{4}$, which is an extension of $L_{-1}(\gs\gl_2)$ and has central charge $c = -3$.
\item In the case $m\geq 3$,  $\text{Com}(\cH, \cS(m))$ is isomorphic to the simple affine vertex algebra $L_{-1}(\gs\gl_m)$.
\end{enumerate}
\end{thm}
Note that (1) is due to Wang \cite{Wa}, (2) was proven by Creutzig, Kanade, Ridout and the first author in \cite{CKLR}, and (3) is due to Adamovi\'c and Per\v{s}e \cite{AP}.

If $n \geq 2$, the cases $1 \leq m < n$ are already understood by Theorem 4.4 of \cite{LSS2}. We have
$$\cS(nm)^{\gg\gl_n[t]} \cong  \left\{
\begin{array}{ll}
\mathbb{C} & m = 1,
\smallskip
\\ \tilde{V}^{-n}(\gs\gl_m) & 2\leq m < n .
\end{array} 
\right.
$$
Note that $\tilde{V}^{-n}(\gs\gl_m) \cong V^{-n}(\gs\gl_m)$ since the latter is simple when $n>m$.

For $n\geq 2$ and $m\geq n$, there is a similar pattern to the case $n=1$. In the cases $n \leq m < 2n+1$, $\cS(nm)^{\gg\gl_n[t]}$ is a nontrivial extension of $\tilde{V}^{-n}(\gs\gl_m)$, and for $m\geq 2n+1$ it is just $\tilde{V}^{-n}(\gs\gl_m)$. From now on, we assume $n\geq 2$ and we will consider the cases $m \geq 2n+1$, $n<m < 2 n+1$, and $m=n$ separately.

\subsection{The case $m \geq 2n+1$} The following result generalizes Theorem \ref{thm:casegl1} (3).
\begin{thm} \label{main:glnS} For all $n \geq 2$ and $m\geq 2n+1$, $\cS(nm)^{\gg\gl_n[t]} \cong \tilde{V}^{-n}(\gs\gl_m)$. \end{thm}

\begin{proof} This cannot be proven directly using Theorem \ref{arcspaceinvt} (1) because the map $$\text{gr}_f(\cS(nm)^{\gg\gl_n[t]}) \hookrightarrow \text{gr}(\cS(nm))^{\gg\gl_n[t]}$$ given by \eqref{injgamm1}  fails to be surjective. Instead, we will study $\cS(nm)^{\gg\gl_n[t]}$ using the structure of $\cS(nm)^{\gs\gl_n[t]}$ given by Theorem \ref{main:sln}. Since $\tilde{V}^{-m}(\gg\gl_n) \cong \cH \otimes \tilde{V}^{-m}(\gs\gl_n)$, we have $$\cS(nm)^{\gg\gl_n[t]} = \text{Com}(\cH, \cS(nm)^{\gs\gl_n[t]}).$$ Since $\cH$ is generated by the field $e$ given by \eqref{defchargebg}, $\text{Com}(\cH, \cS(nm))$ is just the subalgebra $\cS(nm)^0 \subseteq \cS(nm)$ of charge zero, and $\cS(nm)^{\gg\gl_n[t]} \subseteq \cS(nm)^0$.

Recall that the subalgebra $\tilde{V}^{-n}(\gg\gl_m) = \cH \otimes \tilde{V}^{-n}(\gs\gl_m) \subseteq (\cS(nm)^0)^{\gs\gl_n[t]}$ has strong generators
$$X^{ij} = \sum_{k=1}^n :\beta^{ki} \gamma^{kj}:,\qquad i,j = 1,\dots, m.$$ The additional generators $D_{i_1,\dots, i_n}, D'_{j_1,\dots, j_n}$ of $\cS(nm)^{\gs\gl_n[t]}$ have charges $-n, n$, respectively. Therefore as a module over $\tilde{V}^{-n}(\gg\gl_m)$, $(\cS(nm)^0)^{\gs\gl_n[t]}$ is generated by normally ordered monomials in $D_{i_1,\dots, i_n}$ and $D'_{j_1,\dots, j_n}$ and their derivatives, with the same number of $D, D'$.
By induction on length, any such monomial can be generated over $\tilde{V}^{-n}(\gg\gl_m)$ by the products $:\partial^k D_{i_1,\dots, i_n} \partial^l D'_{j_1,\dots, j_n}:$. In fact, each $:\partial^k D_{i_1,\dots, i_n} \partial^l D'_{j_1,\dots, j_n}:$ can be generated by $:D_{i_1,\dots, i_n}  \partial^tD'_{j_1,\dots, j_n}:$ for $t\geq 0$ under repeated action of $\partial$.

We will show that each product $:D_{i_1,\dots, i_n}  \partial^tD'_{j_1,\dots, j_n}:$ lies in $\tilde{V}^{-n}(\gg\gl_m)$. It follows that $(\cS(nm)^0)^{\gs\gl_n[t]} \subseteq \tilde{V}^{-n}(\gg\gl_m)$. Since $\cS(nm)^{\gg\gl_n[t]} \subseteq (\cS(nm)^0)^{\gs\gl_n[t]}$, and $\cS(nm)^{\gg\gl_n[t]}$ commutes with $\cH$, we obtain $\cS(nm)^{\gg\gl_n[t]} \subseteq \tilde{V}^{-n}(\gs\gl_m)$. This completes the proof that $\cS(nm)^{\gg\gl_n[t]} = \tilde{V}^{-n}(\gs\gl_m)$.

\noindent {\bf The case $t=0$.} We denote by $d_{i_1,\dots, i_n}$, $d'_{j_1,\dots, j_n}$, and $x^{ij}$ the images of the fields $D_{i_1,\dots, i_n}$, $D'_{j_1,\dots, j_n}$, and $X^{ij}$ in $\text{gr}_f(\cS(nm)^{\gs\gl_n})$, respectively. By classical invariant theory, there is a relation of degree $n$ (that is, degree $2n$ in the variables $\beta^{ij}_0, \gamma^{ij}_0$):
$$d_{i_1,\dots, i_n} d'_{j_1,\dots, j_n} - \left|\begin{matrix} x^{i_1 j_1}& \cdots &  x^{i_1j_n} \cr  \vdots  & & \vdots  \cr   x^{i_n j_1}  & \cdots &  x^{i_n j_n} \end{matrix} \right|.$$

In the vertex algebra setting, the corresponding normally ordered expression $$\omega = \  :D_{i_1,\dots, i_n} D'_{j_1,\dots, j_n} : - \left|\begin{matrix} X^{i_1 j_1}& \cdots &  X^{i_1 j_n} \cr  \vdots  & & \vdots  \cr   X^{i_n j_1}  & \cdots &  X^{i_n j_n} \end{matrix} \right|$$
need not vanish but it has degree $d< n$ and is invariant under $\gs\gl_n[t]$. Therefore the same holds for the image $\phi_d(\omega)$ in the degree $d$ part of $\text{gr}_f(\cS(nm)^{\gs\gl_n[t]})$. It follows that $\phi_d(\omega)$ can be expressed as a normally ordered polynomial in the generators of $\mathbb{C}[J_{\infty}((\mathbb{C}^n \oplus (\mathbb{C}^n)^*)^{\oplus m}]$, namely, $x^{ij}$, $d_{i_1,\dots, i_n}$, $d'_{j_1,\dots, j_n}$, and their derivatives. By degree considerations, $\phi_d(\omega)$ must depend only on the quadratics $x^{ij}$ and their derivatives, so we can subtract the corresponding normally ordered polynomial in $X^{ij}$ and their derivatives, and proceed by induction on $d$. Note that the same argument shows that for all $\{i_1,\dots, i_n\}$ and $\{j_1,\dots, j_n\}$ and all $t,s \geq 0$, 
\begin{equation} \label{DD'lowering} (D_{i_1,\dots, i_n})_{(s)} \partial^t D'_{j_1,\dots, j_n} \in \tilde{V}^{-n}(\gg\gl_m).\end{equation}

\noindent {\bf The case $t\geq 1$.} 
We assume inductively that  when $\{i_1,\dots, i_n\}$ and $\{j_1,\dots, j_n\}$ are disjoint, 
$$:D_{i_1,\dots, i_n} \partial^s D'_{j_1,\dots, j_n}:\  \in \tilde{V}^{-n}(\gg\gl_m)\ \text{for}\ s \leq t.$$ Equivalently, 
$:(\partial^i D_{i_1,\dots, i_n}) \partial^{s-i} D'_{j_1,\dots, j_n}:\  \in \tilde{V}^{-n}(\gg\gl_m)$ for all $s \leq t$ and $0\leq i \leq s$. We will then show that $D_{i_1,\dots, i_n} \partial^{t+1} D'_{j_1,\dots, j_n} \in \tilde{V}^{-n}(\gg\gl_m)$.

We need to consider normally ordered relations which are built from the classical relations
\begin{equation} \label{sln:classicalrelation} \sum_{k=0}^n (-1)^k x^{u,v_k} d'_{v_0,\dots, v_{k-1}, v_{k+1},\dots, v_n} = 0.
\end{equation}

Let $j_0\in \{1,\dots,m\}$ be an element which appears on neither list $\{i_1,\dots, i_n\}$ or $\{j_1,\dots, j_n\}$ (which always exists because $m\geq 2n+1$).
 Then we have 
$$\sum_{k=0}^n (-1)^k :X^{j_0 j_k}  D'_{j_0,\dots, j_{k-1}, j_{k+1},\dots, j_n}:\ =-\partial D'_{j_1,\dots,j_n}.$$
Taking normally ordered product with $\partial^tD_{i_1,\dots,i_n}$ on the right yields
$$- :(\partial D'_{j_1,\dots,j_n}) \partial^t D_{i_1,\dots,i_n}:\ = \sum_{k=0}^n (-1)^k  :(:X^{j_0 j_k}  D'_{j_0,\dots, j_{k-1}, j_{k+1},\dots, j_n}:)\partial^t D_{i_1,\dots, i_{n}} :.$$ 
Using \eqref{commutator}-\eqref{nonasswick} and \eqref{DD'lowering}, and the fact that $\{i_1,\dots, i_n\}$, $\{j_1,\dots, j_n\}$ and $\{j_0\}$ are disjoint, it follows that 
\begin{equation} \begin{split} & \sum_{k=0}^n (-1)^k  :(:X^{j_0 j_k}  D'_{j_0,\dots, j_{k-1}, j_{k+1},\dots, j_n}:)\partial^t D_{i_1,\dots, i_{n}} : 
\\ & =
\sum_{k=0}^n (-1)^k  :X^{j_0 j_k}  (\partial^t D_{i_1,\dots, i_{n}}) D'_{j_0,\dots, j_{k-1}, j_{k+1},\dots, j_n}:. \end{split} \end{equation}  
By inductive hypothesis, $$:X^{j_0 j_k} (\partial^t D_{i_1,\dots, i_{n}}) D'_{j_0,\dots, j_{k-1}, j_{k+1},\dots, j_n}:\ \in \tilde{V}^{-n}(\gg\gl_m),\ \text{for} \ k = 0,\dots, n.$$ 
By \eqref{commutator}, $ :(\partial D'_{j_1,\dots,j_n}) \partial^t D_{i_1,\dots,i_n}: = \ :( \partial^tD_{i_1,\dots,i_n})\partial D'_{j_1,\dots,j_n}:$, so we obtain $:( \partial^tD_{i_1,\dots,i_n})\partial D'_{j_1,\dots,j_n}:\ \in  \tilde{V}^{-n}(\gg\gl_m)$. Next, since
 $$\partial (:\partial^t D_{i_1,\dots,i_n} D'_{j_1,\dots,j_n}:)=\ :(\partial^{t+1}D_{i_1,\dots,i_n}) D'_{j_1,\dots,j_n}: \ + \ :(\partial^t D_{i_1,\dots,i_n}) \partial D'_{j_1,\dots,j_n} :,$$ and $\partial (:\partial^t D_{i_1,\dots,i_n} D'_{j_1,\dots,j_n}:) \in \tilde{V}^{-n}(\gg\gl_m)$ by inductive hypotheses, we get 
 $:(\partial^{t+1}D_{i_1,\dots,i_n}) D'_{j_1,\dots,j_n}: \ \in  \tilde{V}^{-n}(\gg\gl_m)$. Since
 $$:(\partial^{t+1} D_{i_1,\dots,i_n}) D'_{j_1,\dots,j_n}: \ = \sum_{i=0}^{t+1} (-1)^i  \binom{t+1}{i}\partial^{t+1-i} (:D_{i_1,\dots,i_n} \partial^{i} D'_{j_1,\dots,j_n}:),$$ we conclude that  $:D_{i_1,\dots,i_n} \partial^{t+1} D'_{j_1,\dots,j_n}: \ \in  \tilde{V}^{-n}(\gg\gl_m)$, as well. This completes the case where $\{i_1,\dots, i_n\}$ and $\{j_1,\dots, j_n\}$ are disjoint.

Finally, by \eqref{ncw} we have 
\begin{equation} \begin{split} X^{j_s, i_s}_{(1)} (:D_{i_1,\dots, i_n} \partial^{t+1} D'_{i_1,\dots, i_{s-1},j_s,\dots,j_n} :) 
& = (t+1):D_{i_1,\dots, i_n} \partial^{t} D'_{i_1,\dots, i_{s-1},i_s,j_{s+1},\dots,j_n}:
\\ & -(D_{i_1,\dots,i_{s-1},j_s,i_{s+1},\dots, i_n})_{(0)} \partial^{t+1} D'_{i_1,\dots, i_{s-1},j_s,\dots,j_n} .\end{split} \end{equation}
Since $(D_{i_1,\dots,i_{s-1},j_s,i_{s+1},\dots, i_n})_{(0)} \partial^{t+1} D'_{i_1,\dots, i_{s-1},j_s,\dots,j_n}  \ \in \tilde{V}^{-n}(\gg\gl_m)$ by \eqref{DD'lowering}, it follows by induction on $s$ that $:D_{i_1,\dots, i_n} \partial^{t} D'_{i_1,\dots, i_{s-1},i_s,j_{s+1},\dots,j_n}: \ \in  \tilde{V}^{-n}(\gg\gl_m)$ for all $s \leq n$. \end{proof}

\subsection{The case $n+1 \leq m < 2n+1$}
Recall the generators $\{X^{ij}|\ i,j = 1,\dots, m\}$ for $\tilde{V}^{-n}(\mathfrak{gl}_m)$, as well as the Heisenberg field $e = \sum_{i=1}^n \sum_{j=1}^{m} :\beta^{ij} \gamma^{ij}:$ which commutes with both $\tilde{V}^{-m}(\mathfrak{sl}_n)$ and $\tilde{V}^{-n}(\mathfrak{sl}_m)$. Note that the zero mode of $e$ induces an action of $U(1)$ on $\cS(nm)^{\mathfrak{sl}_n[t]}$, and that $(\cS(nm)^{\mathfrak{sl}_n[t]})^{U(1)} = \cS(nm)^{\mathfrak{sl}_n[t]} \cap \cS(nm)^0$.

\begin{thm} Fix $n \geq 2$ and $ n+1 \leq m < 2n+1$. Then
\begin{enumerate}
\item $(\cS(nm)^{\mathfrak{sl}_n[t]})^{U(1)}$ is generated by $\{X^{ij}\}$ together with one additional field $:(\partial D'_{2,\dots, n+1}) D_{2,\dots,n+1}:$, of weight $n+1$.
\item $(\cS(nm)^{\mathfrak{sl}_n[t]})^{U(1)}$ has a minimal strong generating set consisting of $\{X^{ij}\}$, together with $\binom{m}{n}^2$ additional fields on weight $n+1$, 
$$:\partial D'_{i_1,\dots, i_n} D_{j_1,\dots, j_n}:,$$ for all subsets $\{i_1,\dots, i_n\}$ and $\{j_1,\dots, j_n\}$ of $\{1,\dots, m\}$.
\end{enumerate}
 
\end{thm}
\begin{proof}

First, for all  $\{i_1,\dots, i_n\}$ and $\{j_1,\dots, j_n\}$, we have $D'_{i_1,\dots, i_n} D_{j_1,\dots,j_n}\in \tilde{V}^{-n}(\mathfrak{gl}_m)$. Next, via the action of $\gg\gl_m$ generated by the zero modes $X^{ij}_{(0)}$, all fields $\partial D'_{i_1,\dots, i_n} D_{j_1,\dots,j_n}$ lie in the subalgebra generated by $\{X^{ij}\}$ and $:\partial D'_{2,\dots, n+1} D_{2,\dots,n+1}:$.

Let $\cW(nm)$ denote the span of all normally ordered monomials in $X^{ij}$, $:\partial D'_{i_1,\dots, i_n} D_{j_1,\dots, j_n}:$, and their derivatives, where $\{i_1,\dots, i_n\}$ and $\{j_1,\dots, j_n\}$ range over all subsets of $\{1,\dots, m\}$. As in the proof of Theorem \ref{main:glnS}, $(\cS(nm)^{\mathfrak{sl}_n[t]})^{U(1)}$ is strongly generated by $X^{ij}$ together with $:\partial^k D'_{i_1,\dots, i_n}\partial^l D_{j_1,\dots, j_n}:$ for all $k,l \geq 0$, and all $\{i_1,\dots, i_n\}$, $\{j_1,\dots, j_n\}$. So to prove both statements, it suffices to show that all elements $:\partial^k D'_{i_1,\dots, i_n}\partial^l D_{j_1,\dots, j_n}:$ lie in $\cW(nm)$.

For fixed $k,l\geq 0$, $:\partial^k D'_{2,\dots, n+1}\partial^l D_{2,\dots,n+1}:$ can generate any $:\partial^k D'_{i_1,\dots, i_n}\partial^l D_{j_1,\dots,j_n}:$ under the action of $X^{ij}_{(0)}$. To show that $:\partial^k D'_{2,\dots, n+1}\partial^l D_{2,\dots,n+1}:\ \in \cW(nm)$, it suffices to show that $: D'_{2,\dots, n+1} \partial^t D_{2,\dots,n+1}:\ \in \cW(nm)$ for all $t \geq 1$. We will proceed by induction on $t$, so we assume that for all $s \leq t$, $:D'_{2,\dots n+1}\partial^s D_{2,\dots, n+1}:\ \in \cW(nm)$. Then for all $k+l\leq t$ and all $\{i_1,\dots, i_n\}$ and $\{j_1,\dots, j_n\}$, we have $:\partial^k D'_{i_1,\dots, i_n}\partial^l D_{j_1,\dots,j_n}:\ \in \cW(nm)$.

Note that for distinct $j_0,\dots, j_n$, we have 
$$\sum_{k=0}^n (-1)^k :X^{j_0 j_k}  D'_{j_0,\dots, j_{k-1}, j_{k+1},\dots, j_n}:\ =-\partial D'_{j_1,\dots,j_n}.$$
Taking normally ordered product with $\partial^tD_{j_1,\dots,j_n}$ on the right yields
\begin{equation} \label{secondcasegln} - :(\partial D'_{j_1,\dots,j_n}) \partial^t D_{j_1,\dots,j_n}:\ = \sum_{k=0}^n (-1)^k  :(:X^{j_0 j_k}  D'_{j_0,\dots, j_{k-1}, j_{k+1},\dots, j_n}:)\partial^t D_{j_1,\dots, j_{n}} :.\end{equation}
It follows that 
\begin{equation*} \begin{split} & \sum_{k=0}^n (-1)^k  :(:X^{j_0 j_k}  D'_{j_0,\dots, j_{k-1}, j_{k+1},\dots, j_n}:)\partial^t D_{j_1,\dots, j_{n}} : 
			\\ & =
		\sum_{k=0}^n (-1)^k  :X^{j_0 j_k}  (:D'_{j_0,\dots, j_{k-1}, j_{k+1},\dots, j_n}\partial^t D_{j_1,\dots, j_{n}} :): 
	\\  &+\sum_{k=0}^n(-1)^k\sum_{s\geq 0}\frac 1 {(s+1)!} :\partial^{s+1} X^{j_0j_k} ((D'_{j_0,\dots, j_{k-1}, j_{k+1},\dots, j_n})_{(s)}\partial^t D_{j_1,\dots, j_{n}}):
	\\  &+\sum_{k=0}^n(-1)^k\sum_{s\geq 0}\frac 1 {(s+1)!} :\partial^{s+1} D'_{j_0,\dots, j_{k-1}, j_{k+1},\dots, j_n} (X^{j_0j_k} _{(s)}\partial^t D_{j_1,\dots, j_{n}}):
	\\ &\sim \frac 1 {t+1}  \sum_{k=1}^n\sum_{s=0}^t\binom{t+1}{s+1} :\partial^{s+1} D'_{j_0,\dots, j_{k-1}, j_{k+1},\dots, j_n}\partial^{t-s} D_{j_0,\dots, j_{k-1}, j_{k+1},\dots, j_n} : 
	\\ & \sim  -\frac 1 {t+1}\sum_{k=1}^n  :D'_{j_0,\dots, j_{k-1}, j_{k+1},\dots, j_n}\partial^{t+1} D_{j_0,\dots, j_{k-1}, j_{k+1},\dots, j_n} :.
 \end{split} \end{equation*}  
Here $\sim$ means modulo $\cW(nm)$. Note that by our inductive hypothesis, the left side of \eqref{secondcasegln} is equal to $: D'_{j_1,\dots,j_n} \partial^{t+1} D_{j_1,\dots,j_n}:$ modulo $\cW(nm)$, and we use this hypothesis again in the last two lines. We then obtain $$(t+1) : D'_{j_1,\dots,j_n} \partial^{t+1} D_{j_1,\dots,j_n}: +\sum_{k=1}^n  :D'_{j_0,\dots, j_{k-1}, j_{k+1},\dots, j_n}\partial^{t+1} D_{j_0,\dots, j_{k-1}, j_{k+1},\dots, j_n} :\ \sim 0$$
We have $n+1$ equations if we exchange $j_0$ and $j_i$, so we can solve 
$$: D'_{j_1,\dots,j_n} \partial^{t+1} D_{j_1,\dots,j_n}:\ \sim 0.$$ Specializing to $\{j_1,\dots, j_n\} = \{2,\dots,n+1\}$, completes the proof.
\end{proof}

Since $e_{(k)} ( : \partial D'_{i_1,\dots,i_n} D_{j_1,\dots,j_n}:) \in  \tilde{V}^{-n}(\mathfrak{gl}_m)$ for all $k \geq 1$, there exists a field $\nu_{i_1,\dots, i_n; j_1,\dots, j_n} \in \tilde{V}^{-n}(\mathfrak{gl}_m)$ such that $ : \partial D'_{i_1,\dots,i_n} D_{j_1,\dots,j_n}: - \nu_{i_1,\dots, i_n; j_1,\dots, j_n}$ commutes with $e$, and hence lies in $\cS(nm)^{\gg\gl_n[t]}$. We obtain

\begin{cor} For all $n \geq 2$ and $n+1 \leq m < 2n+1$, 
\begin{enumerate}
\item $\cS(nm)^{\mathfrak{gl}_n[t]}$ is generated by $\{X^{ij}\}$ together with the field $ : \partial D'_{2,\dots n+1} D_{2,\dots,n+1}: - \nu_{2,\dots,n+1; 2,\dots,n+1}$ in weight $n+1$.
\item $\cS(nm)^{\mathfrak{gl}_n[t]}$ has a minimal strong generating set consisting of $\{X^{ij}\}$ together with $\binom{m}{n}^2$ fields in weight $n+1$,
$$: \partial D'_{i_1,\dots,i_n} D_{j_1,\dots,j_n} : - \nu_{i_1,\dots, i_n; j_1,\dots, j_n},$$ for all subsets $\{i_1,\dots,i_n\}$ and $\{j_1,\dots, j_n\}$ of $\{1,\dots, m\}$.
\end{enumerate}
\end{cor}

 \subsection{The case $m=n$}
Recall that $\cS(n^2)^{\gs\gl_n[t]} \cong (V^{-n}(\gs\gl_n) \otimes \cW^{-n}(\gs\gl_n, f_{\text{subreg}})) / \cI$, by Corollary \ref{cor:n=m}. The Heisenberg algebra $\cH$ is contained in $\cW^{-n}(\gs\gl_n, f_{\text{subreg}})$ and commutes with $V^{-n}(\gs\gl_n)$, so $$\cS(n^2)^{\gg\gl_n[t]} \cong (V^{-n}(\gs\gl_n) \otimes \cC^{-n}) / \cI, \ \text{where} \ \cC^{-n} \cong \text{Com}(\cH, \cW^{-n}(\gs\gl_n, f_{\text{subreg}})).$$ 

It is known that for generic level $k$, $\cC^k = \text{Com}(\cH, \cW^{k}(\gs\gl_n, f_{\text{subreg}}))$ is of type $\cW(2,3,\dots, 2n+1)$. This is the case $\cC^{\psi}(n-1,1)$ of Lemma 6.1 of \cite{CL3}, and in this notation $\psi = k+n$. What is not immediately clear is that the critical level $k = -n$ is generic in this sense.

The zero mode of the generator of $\cH$ integrates to an action of $U(1)$, and the orbifold $\cW^k(\gs\gl_n, f_{\text{subreg}})^{U(1)}$ is isomorphic to $\cH \otimes \cC^k$. Then $\cC^k$ is of type $\cW(2,3,\dots, 2n+1)$ if and only if $\cW^k(\gs\gl_n, f_{\text{subreg}})^{U(1)}$ is of type $\cW(1,2,3,\dots, 2n+1)$. 

As in \cite{CL3}, we denote the generators of $\cW^k(\gs\gl_n, f_{\text{subreg}})$ by $J, W^2,\dots, W^{n-1}, G^{\pm}$, where $J$ is the Heisenberg field, $W^2,\dots W^{n-1}$ commute with $J$, and $G^{\pm}$ satisfy $J(z) G^{\pm}(w) \sim \pm G^{\pm}(z-w)^{-1}$. Then $\cW^k(\gs\gl_n, f_{\text{subreg}})^{U(1)}$ is strongly generated by $J, L, W^3,\dots, W^{n-1}$, together with the fields 
$$U_{i,j} = \ :(\partial^i G^+)(\partial^j G^-):,\qquad \ i,j \geq 0,$$ which have weight $n+i+j$. These are not all necessary. First of all, we only need $\{U_{0,j}|\ j \geq 0\}$ because $\partial U_{i,j} = U_{i+1,j} + U_{i,j+1}$. Second, there is a normally ordered relation of the form
$$:U_{0,0} U_{1,1}: - :U_{0,1} U_{1,0}:\ = \lambda_{n,1}(k) U_{0,n+2} + P_1.$$ Here $P_1$ is a normally ordered polynomial in 
$J,L,W^3,\dots, W^{n-1}, U_{0,0},U_{0,1},\dots, U_{0,n+1}$, and their derivatives, and $\lambda_{n,1}(k)$ is a nonzero rational function of $k$. The precise formula of $\lambda_{n,1}(k)$ is given for $n= 3$ and $n=4$ in \cite{ACL1,CL2}, but is not known in general. In particular, whenever $\lambda_{n,1}(k) \neq 0$, $U_{0,n+2}$ is not needed in the strong generating set since it can expressed as a normally ordered polynomial in $J,L,W^3,\dots, W^{n-1}, U_{0,0},U_{0,1},\dots, U_{0,n+1}$, and their derivatives. Similarly, for all $m\geq 1$, there are relations
$$:U_{0,0} U_{1,m}: - :U_{0,m} U_{1,0}:\ = \lambda_{n,m}(k) U_{0,n+2} + P_m,$$ where $P_m$ is a normally ordered polynomial in 
$J,L,W^3,\dots, W^{n-1}, U_{0,0},U_{0,1},\dots, U_{0,n+1}$, and their derivatives, and $\lambda_{n,m}(k)$ is a rational function of $k$. The precise formula of $\lambda_{n,m}(k)$ is given for $n= 3$ and $n=4$ and all $m\geq 1$ in \cite{ACL1,CL2}. Although is not known in general, we claim that it is a nonzero rational function for all $n,m$. To see this, recall that in the notation of \cite{CL3}, the large level limit of $\cW^k(\gs\gl_n, f_{\text{subreg}}) = \cC^{\psi}(n-1,1)$ is the following free field algebra:
$$ \left\{
\begin{array}{ll}
\big(\bigotimes_{i=1}^{n-1} \cO_{\text{ev}}(1,2i)\big)  \otimes \cO_{\text{ev}}(2, n), & n \ \text{even},
\smallskip
\\ \big(\bigotimes_{i=1}^{n-1} \cO_{\text{ev}}(1,2i)\big)  \otimes \cS_{\text{ev}}(1, n), & n \ \text{odd}. \\
\end{array} 
\right.$$
It is easy to verify that the corresponding coefficient (which coincides with $\lim_{k\rightarrow \infty}\lambda_{n,m}(k)$ after suitably rescaling the generators), is nonzero, which proves the claim. Therefore when $k$ is generic, $U_{0,j}$ can be eliminated from our strong generating set for all $j \geq n+2$, so that $\cW^k(\gs\gl_n, f_{\text{subreg}})^{U(1)}$ is strongly generated by $J,L,W^3,\dots, W^{n-1}, U_{0,0},U_{0,1},\dots, U_{0,n+1}$. This is equivalent to the fact that $\text{Com}(\cH, \cW^{k}(\gs\gl_n, f_{\text{subreg}}))$ is of type $\cW(2,3,\dots, 2n+1)$, since $\cW^k(\gs\gl_n, f_{\text{subreg}})^{U(1)} \cong \cH \otimes \text{Com}(\cH, \cW^{k}(\gs\gl_n, f_{\text{subreg}}))$.

At the critical level $k = -n$, recall that $\cW^{-n}(\gs\gl_n, f_{\text{subreg}})$ contains central fields $W^2,\dots, W^{n-1}$ which can be identified with the generators of the center of $V^{-n}(\gs\gl_n)$. By Theorem \ref{casen=m}, we may identify $\cW^{-n}(\gs\gl_n, f_{\text{subreg}})$ with $\cS(n^2)^{\gs\gl_n[t] \oplus \gs\gl_n[t]}$, and we use the same notation $U_{i,j}$ to denote the fields $:(\partial^i D^+)(\partial^j D^-):$.

\begin{thm} \label{thm:subregrel} In $\cW^{-n}(\gs\gl_n, f_{\text{subreg}})$, we have the following relations for all $n\geq 3$ and $m\geq 1$:
\begin{equation} :U_{0,0} U_{1,m}: - :U_{0,m} U_{1,0}:\  =  \mu(n,m) U_{0,n+m +1}+ P_{n,m},\qquad \mu(n,m) = \frac{m (2 n + 1 + m)}{(n + 1) (n + m) (n + 1 + m)},\end{equation} where $P_{n,m}$ is a normally ordered polynomial in $J,W^2,\dots, W^{n-1}, U_{0,0}, U_{0,1},\dots, U_{0,n+m}$, and their derivatives.
\end{thm}
The proof is quite long and technical so it appears in Appendix \ref{appendix}. By induction starting with the case $m=1$, this implies that $P_{n,m}$ can be replaced with a normally ordered polynomial in $J,W^2,\dots, W^{n-1}, U_{0,0}, U_{0,1},\dots, U_{0,n+1}$, and their derivatives. Therefore Theorem \ref{thm:subregrel} implies that $\cW^{-n}(\gs\gl_n, f_{\text{subreg}})^{U(1)}$ has a minimal strong generating set $$\{J, W^2,\dots, W^{n-1}, U_{0,0},\dots, U_{0,n+1}\},$$ and hence is of type $\cW(1,2,3,\dots, 2n+1)$. Finally, it is easy to see that $U_{i,j} \in \cS(n^2)^{\gs\gl_n[t]\oplus \gs\gl_n[t]}$ can be corrected by adding an element $\nu_{i,j} \in \cH \otimes V^{-n}(\gs\gl_n)$ so that 
$$\tilde{U}_{i,j} = U_{i,j} + \nu_{i,j} \in \text{Com}(\cH, \cS(n^2)^{\gs\gl_n[t]\oplus \gs\gl_n[t]}).$$
We obtain
\begin{cor} $\cS(n^2)^{\gg\gl_n[t]}$ is of type $\cW(1^{n^2-1}, n+1, n+2,\dots, 2n+1)$. In particular, it is an extension of $V^{-n}(\gs\gl_n)$, and has additional strong generators $\{\tilde{U}_{0,i}|\ i = 1,2,\dots,n+1\}$, which have weights $n+1, n+2,\dots,2n+1$.
\end{cor}

\subsection{The structure of $\cE(nr)^{\gg\gl_n[t]}$}
Recall the homomorphism
$$L_r(\gs\gl_n) \otimes L_n(\gs\gl_r) \otimes \cH \rightarrow \cE(nr)$$ given by \eqref{conformalemb:slnslmbc}, whose image is conformally embedded. The following result is well known (see Theorem 4.1 of \cite{OS}), but we give an alternative proof.
\begin{thm} \label{main:glnE} For all $n\geq 2$ and $r \geq 1$,
$$\text{Com}(L_r(\gs\gl_n) \otimes\cH, \cE(nr))  = \cE(nr)^{\gg\gl_n[t]} \cong L_n(\gs\gl_r).$$
\end{thm}
\begin{proof} Recall from Theorem \ref{main:slnE} that $\cE(nr)^{\gs\gl_n[t]}$ is strongly generated by the fields 
$$Y^{st} = \sum_{k=1}^n :b^{sk} c^{tk}: \ \in L_n(\gg\gl_r) = \cH \otimes L_n(\gs\gl_r),\qquad s,t = 1,\dots, r,$$ together with the fermionic determinants $D_{i_1,\dots, i_n}, D'_{i_1,\dots, i_n}$ for all $1 \leq i_1 \leq \cdots \leq  i_n \leq r$. As above, $\cE(nr)^{\gg\gl_n[t]}$ lies in the subspace of charge zero, hence it suffices to prove that all fields $:D_{i_1,\dots, i_n} \partial^k D'_{j_1,\dots, j_n}:$ lie in the subalgebra $L_n(\gg\gl_r)$ generated by $X^{ij}$. The proof is similar to the proof of Theorem \ref{main:glnS}, and the details are omitted. \end{proof}

\subsection{The structure of $(\cS(nm) \otimes \cE(nr))^{\gg\gl_n[t]}$}
Finally, recall the homomorphism
$$V^{-m+r}(\gg\gl_n) \otimes V^n (\gs\gl_{r|m}) \rightarrow \cS(nm) \otimes \cE(nr)$$ given by \eqref{super:1}, whose image is conformally embedded.

\begin{thm} \label{main:glnSE} For all $n\geq 2$ and $m,r \geq 1$, 
$$\text{Com}(\tilde{V}^{-m+r}(\gg\gl_n), \cS(nm) \otimes \cE(nr)) = (\cS(nm) \otimes \cE(nr))^{\gg\gl_n[t]} \cong \tilde{V}^n (\gs\gl_{r|m}).$$
\end{thm}

\begin{proof} By Theorem \ref{main:slnSE}, $(\cS(nm) \otimes \cE(nr))^{\gs\gl_n[t]}$ is strongly generated by the generators of $\tilde{V}^{n}(\gg\gl_{r|m})$, together with the fields $D_{i_1,\dots, i_s; j_{s+1},\dots, j_n}$ and $D'_{i_1,\dots, i_s; j_{s+1},\dots, j_n}$. As in the proof of Theorem \ref{main:glnS}, it suffices to show that all the fields 
$$:D_{i_1,\dots, i_s; j_{s+1},\dots, j_n} \partial^t D'_{i'_1,\dots, i'_u; j'_{u+1},\dots, j'_n}:,\qquad t \geq 0,$$ lie in $\tilde{V}^{n}(\gg\gl_{r|m})$. 

First, for $t=0$, the same argument as the case $t=0$ of Theorem \ref{main:glnS} shows that for all $\{i_1,\dots, i_s; j_{s+1},\dots, j_n\}$ and $\{i'_1,\dots, i'_u; j'_{u+1},\dots, j'_n\}$,
 \begin{equation} \label{DD'reductionodd}  \begin{split}  & :D_{i_1,\dots, i_s; j_{s+1},\dots, j_n} D'_{i'_1,\dots, i'_u; j'_{u+1},\dots, j'_n}: \ \in \tilde{V}^{n}(\gg\gl_{r|m}),
 \\ & (D_{i_1,\dots, i_s; j_{s+1},\dots, j_n})_{(j)} \partial^t D'_{i'_1,\dots, i'_u; j'_{u+1},\dots, j'_n} \in \tilde{V}^{n}(\gg\gl_{r|m}),\ \text{for all} \ j,t \geq 0. \end{split} \end{equation}

For $t > 0$, we begin with the case where $\{i_1,\dots, i_s\}$ and $\{i'_1,\dots, i'_u\}$ are disjoint, and $\{j_{s+1},\dots, j_n\}$ and $\{j'_{u+1},\dots, j'_n\}$ are disjoint. Assume inductively that $$:D_{i_1,\dots, i_s; j_{s+1},\dots, j_n} \partial^j D'_{i'_1,\dots, i'_u; j'_{u+1},\dots, j'_n}:\  \in \tilde{V}^{n}(\gg\gl_{r|m}),\ \text{for}\ j \leq t.$$ 
In the notation of Theorem \ref{oddgen}, let $V = (\mathbb{C}^n \oplus (\mathbb{C}^n)^*)^{\oplus m}$ and $W = (\mathbb{C}^n \oplus (\mathbb{C}^n)^*)^{\oplus r}$, and  let $$x^{ij},\  y^{st}, \ e^{is}, \ f^{sj}, \ d_{i_1,\dots, i_s; j_{s+1},\dots, j_n}, \ d'_{i_1,\dots, i_s; j_{s+1},\dots, j_n}\  \in S^{V} \otimes L^W \cong \text{gr}(\cS(nm) \otimes \cE(nr)),$$ denote the elements which correspond to the fields \eqref{gen:vnsl(r|m)} and \eqref{gen:mixedDD'}. Consider the classical relation
\begin{equation} \begin{split} & \sum_{k=1}^{u} (-1)^k f^{j'_{n+1} i'_k} d'_{i'_1,\dots, i'_{k-1}, i'_{k+1},\dots, i'_u; j'_{u+1},\dots, j'_n,j'_{n+1}} 
\\ & + (-1)^{u+1} \sum_{l=1}^{n+1-u} y^{j'_{n+1} j'_{u+l}} d'_{i'_1,\dots,i'_u, j'_{u+1},\dots, j'_{u+l-1}, j'_{u+l+1},\dots, j'_n,j'_{n+1}} = 0.\end{split} \end{equation}
Here $\{j'_{u+1},\dots, j'_{n}, j'_{n+1}\}$ need not be distinct.

The corresponding normally ordered relation is
\begin{equation} \begin{split} & \sum_{k=1}^{u} (-1)^k :F^{j'_{n+1} i'_k} D'_{i'_1,\dots, i'_{k-1}, i'_{k+1},\dots, i'_u; j'_{u+1},\dots, j'_n,j'_{n+1}}:
\\ & +(-1)^{u+1}  \sum_{l=u+1}^{n+1-u}  Y^{j'_{n+1} j'_{u+l}} D'_{i'_1,\dots,i'_u, j'_{u+1},\dots, j'_{u+l-1}, j'_{u+l+1},\dots, j'_n,j'_{n+1}}: \ +(-1)^{u}  : \partial D'_{i'_1,\dots, i'_u; j'_{u+1},\dots, j'_{n}}: \ = 0.\end{split} \end{equation}
Taking the normally ordered product on the right by $\partial^t D_{i_1,\dots, i_s; j_{s+1},\dots, j_n}$, and applying the same argument as the proof of Theorem \ref{main:glnS}, we see that 
$$:D_{i_1,\dots, i_s; j_{s+1},\dots, j_n} \partial^t D'_{i'_1,\dots, i'_u; j'_{u+1},\dots, j'_n}: \ \in \tilde{V}^{n}(\gg\gl_{r|m})$$ for all $t$ when $\{i_1,\dots, i_s\}$ and $\{i'_1,\dots, i'_u\}$ are disjoint, and $\{j_{s+1},\dots, j_n\}$ and $\{j'_{u+1},\dots, j'_n\}$ are disjoint.

Suppose first that $s \leq u$. By \eqref{ncw} we have 
\begin{equation} \begin{split} & X^{i'_a, i_a}_{(1)} (:D_{i_1,\dots, i_s; j_{s+1},\dots, j_n} \partial^{t+1} D'_{i_1,\dots, i_{a-1}, i'_a,\dots,i'_u; j'_{u+1},\dots, j'_n}: ) 
\\ & = (t+1):D_{i_1,\dots, i_s; j_{s+1},\dots, j_n} \partial^{t} D'_{i_1,\dots, i_{a-1}, i_a, i'_{a+1},\dots,i'_u; j'_{u+1},\dots, j'_n}
\\ & -(D_{i_1,\dots,i_{a-1}, i'_a, i_{a+1},\dots, i_s; j_{s+1},\dots, j_n} )_{(0)} \partial^{t+1} D'_{i_1,\dots, i_{a-1}, i'_a,\dots,i'_u; j'_{u+1},\dots, j'_n}.\end{split} \end{equation}

Since $(D_{i_1,\dots,i_{a-1}, i'_a, i_{a+1},\dots, i_s; j_{s+1},\dots, j_n} )_{(0)} \partial^{t+1} D'_{i_1,\dots, i_{a-1}, i'_a,\dots,i'_u; j'_{u+1},\dots, j'_n} \ \in  \tilde{V}^{n}(\gg\gl_{r|m})$ by \eqref{DD'reductionodd}, it follows by induction on $a$ that 
$:D_{i_1,\dots, i_s; j_{s+1},\dots, j_n} \partial^{t} D'_{i_1,\dots, i_{a-1}, i_a, i'_{a+1},\dots,i'_u; j'_{u+1},\dots, j'_n}: \ \in \tilde{V}^{n}(\gg\gl_{r|m})$ for all $a \leq s$. 

By \eqref{ncw} again, we have 
\begin{equation} \begin{split} & F^{ j'_{u+b+1},j_{s+b}}_{(1)} \big(:(\partial^{t+1} D_{i_1,\dots, i_s; j'_{u+1},\dots, j'_{u+b-1},j_{s+b},\dots, j_n}) D'_{i_1,\dots, i_{a}, i'_{a+1},\dots,i'_u; j'_{u+1},\dots, j'_n}: \big) 
\\ & = (t+1) :(\partial^{t} D_{i_1,\dots, i_s; j'_{u+1},\dots, j'_{u+b-1}, j'_{u+b}, j_{s+b+1},\dots, j_n}) D'_{i_1,\dots, i_{a}, i'_{a+1},\dots,i'_u; j'_{u+1},\dots, j'_n}: .\end{split} \end{equation}
Again, it follows by induction on $b$ that  $ :(\partial^{t} D_{i_1,\dots, i_s; j'_{u+1},\dots,  j'_{u+b}, j_{s+b+1},\dots, j_n}) D'_{i_1,\dots, i_{a}, i'_{a+1},\dots,i'_u; j'_{u+1},\dots, j'_n}: \ \in \tilde{V}^{n}(\gg\gl_{r|m})$ for all $b \leq n-s$. Finally, the case $u<s$ can be proven in the same way by reversing the roles of $D$ and $D'$. \end{proof}

\section{The case $\gg = \gs\gp_{2n}$}
\subsection{The structure of $\cS(nm)^{\gs\gp_{2n}[t]}$}
Recall the homomorphism $V^{-\frac{m}{2}}(\gs\gp_{2n}) \otimes V^{-2n}(\gs\go_m) \rightarrow \cS(nm)$ given by \eqref{conformalemb:sp2nsombg}, whose image $\tilde{V}^{-\frac{m}{2}}(\gs\gp_{2n}) \otimes \tilde{V}^{-2n}(\gs\go_m)$ is conformally embedded.

\begin{thm} \label{thm:spbg} For all $n, m \geq 1$, $\cS(nm)^{\gs\gp_{2n}[t]} \cong \tilde{V}^{-2n}(\gs\go_{m})$.
\end{thm}

\begin{proof} The case $m \leq 2n+2$ is given by Theorem 5.1 of \cite{LSS2}. In the general case, it follows from Theorem \ref{arcspaceinvt} (2) that the generators of $\text{gr}(\cS(nm))^{\gs\gp_{2n}[t]}$ as a differential algebra correspond to the generators of $\mathbb{C}[(\mathbb{C}^{2n})^{\oplus m}]^{Sp_{2n}}$, and are the generators of $V^{-2n}(\gs\go_m)$. If follows that the map 
$$\text{gr}_f(\cS(nm)^{\gs\gp_{2n}[t]}) \hookrightarrow \text{gr}(\cS(nm))^{\gs\gp_{2n}[t]},$$ is surjective. So the generators of  the classical invariant ring $\mathbb{C}[(\mathbb{C}^{2n})^{\oplus m}]^{Sp_{2n}}$ give rise to a generating set for $\text{gr}_f(\cS(nm)^{\gs\gp_{2n}[t]})$ as a differential algebra. By Lemma \ref{lem:filtration}, the corresponding fields strongly generate $\cS(nm)^{\gs\gp_{2n}[t]}$ as a vertex algebra.
\end{proof}

\begin{cor} For all $n, m \geq 1$, 
\begin{enumerate}
\item The Zhu algebra $A(\cS(nm)^{\gs\gp_{2n}[t]})$ is isomorphic to the ring of invariant differential operators $\cD(nm)^{\gs\gp_{2n}}$. 
\item The Zhu commutative algebra $R_{\cS(nm)^{\gs\gp_{2n}[t]}}$ is isomorphic to $\mathbb{C}[(\mathbb{C}^{2n})^{\oplus m}]^{Sp_{2n}}$.
\item $\tilde{V}^{-2n}(\gs\go_{m})$ is classically free.
\end{enumerate}
\end{cor}

\subsection{Actions on $\cE(2nr)$}
Next, taking $V = (\mathbb{C}^{2n})^{\oplus r}$, recall the conformal embedding
$L_{r}(\gs\gp_{2n}) \otimes L_n(\gs\gp_{2r}) \rightarrow \cE(2nm)$ given by \eqref{sp2n-sp2mE}. The following result is well known (see Proposition 2 of \cite{KP} as well as the appendix of \cite{ORS}), and we give an alternative proof.
\begin{thm} \label{thm:spbc} For all $n,r\geq 1$, we have
$$\text{Com}(L_{r}(\gs\gp_{2n}), \cE(2nr)) \cong L_{n}(\gs\gp_{2r}).$$
\end{thm}

\begin{proof} The argument is the same as the proof of Theorem \ref{main:slnE}, and follows from Theorem \ref{arcspaceinvt} and Lemma \ref{lem:filtration}.
\end{proof}

This following corollary generalizes Theorem 15.21 of \cite{EH1}, which is the case $r = 1$.

\begin{cor} \label{cor:classicalfreeness} For all integers $r, n\geq 1$, $L_n(\gs\gp_{2r})$ is classically free.
\end{cor}

\begin{proof} Let $f: L_r(\gs\gp_{2n}) \rightarrow \cE(2nr)$ be the above map. It follows from Theorem \ref{thm:spbc} and Lemma \ref{lem:filtration} that 

$$\text{gr}_f(\cE(2nr)^{\gs\gp_{2n}[t]}) \cong \text{gr}(\cE(2nr)^{\gs\gp_{2n}[t]}) \cong \text{gr}^F(\cE(2nr)^{\gs\gp_{2n}[t]}) \cong  \bigwedge \big(\bigoplus_{j\geq 0}(V_j \oplus V^*_j)\big)^{J_{\infty}(Sp_{2n})}.$$ By Theorem \ref{oddgen} (2), all relations among the generators of $\bigwedge \big(\bigoplus_{j\geq 0}(V_j \oplus V^*_j)\big)^{J_{\infty}(Sp_{2n})}$ are consequences of the relations in $\bigwedge[(V_0 \oplus V^*_0)]^{Sp_{2n}}$ and their derivatives, so the same statement holds in $\text{gr}^F(\cE(2nr)^{\gs\gp_{2n}[t]})\cong \text{gr}^F(L_n(\gs\gp_{2r}))$. Equivalently, $L_n(\gs\gp_{2r})$ is classically free.
\end{proof}

\begin{remark} Corollary \ref{cor:classicalfreeness} implies that Theorem 10.2.1 of \cite{EH2}, namely the vanishing of the first chiral homology $H_1^{\text{ch}}(V)$, holds for $V = L_n(\gs\gp_{2r})$. This generalizes Corollary 12.2 (b) of \cite{EH2}.
\end{remark}

\subsection{Actions on $\cS(nm) \otimes \cE(2nr)$}
Finally, recall the homomorphism
$$V^{-\frac{m}{2}+r}(\gs\gp_{2n}) \otimes V^{n}(\go\gs\gp_{m|2r}) \rightarrow \cS(nm) \otimes \cE(2nr)$$ given by \eqref{super:2}, whose image  $\tilde{V}^{-\frac{m}{2}+r}(\gs\gp_{2n})\otimes \tilde{V}^{n}(\go\gs\gp_{m|2r})$ is conformally embedded.

\begin{thm} \label{thm:spbgbc}
$\text{Com}(\tilde{V}^{-\frac{m}{2}+r}(\gs\gp_{2n}), \cS(nm) \otimes \cE(2nr)) \cong \tilde{V}^{n}(\go\gs\gp_{m|2r})$.
\end{thm}

\begin{proof} The argument is the same as the proof of Theorem \ref{main:slnSE}.
\end{proof}

\section{Level-rank dualities involving affine vertex superalgebras}
\subsection{Type $A$ case} Recall the embeddings 
$$\tilde{V}^{-m}(\gg\gl_n) \otimes \tilde{V}^{-n}(\gs\gl_m) \rightarrow \cS(nm),\qquad \tilde{V}^{-n+r}(\gg\gl_m) \otimes \tilde{V}^{m}(\gs\gl_{r|n}) \rightarrow  \cS(mn) \otimes \cE(mr).$$
As in \cite{CLR}, we use the notation
\begin{equation} \begin{split} & A^{-n}(\mathfrak{sl}_{m}) = \cS(nm)^{\gg\gl_n[t]},
\\ & A^m(\mathfrak{sl}_{r|n}) =  (\cS(mn) \otimes \cE(mr))^{\gg\gl_m[t]},
\end{split} \end{equation} 
since they are extensions of $\tilde{V}^{-n}(\mathfrak{sl}_{m})$ and $\tilde{V}^m(\mathfrak{sl}_{r|n})$, respectively. We denote by $\tilde{V}^{-n+r}(\mathfrak{sl}_m)$ the image of $V^{-n+r}(\mathfrak{sl}_m)$ under the diagonal map $V^{-n+r}(\mathfrak{sl}_m) \rightarrow A^{-n}(\mathfrak{sl}_{m}) \otimes L_r(\mathfrak{sl}_m)$. In \cite{CLR}, it was proven that
$$\text{Com}(\tilde{V}^{-n+r}(\mathfrak{sl}_m), A^{-n}(\mathfrak{sl}_{m}) \otimes L_r(\mathfrak{sl}_m))\cong \text{Com}(\tilde{V}^{-m}(\mathfrak{sl}_n) \otimes L_m(\mathfrak{sl}_r) \otimes \cH, A^m(\mathfrak{sl}_{r|n})).$$
Since $A^{-n}(\mathfrak{sl}_{m}) = \tilde{V}^{-n}(\gs\gl_m)$ for all $m < n$ and $m\geq 2n+1$, and that $A^m(\mathfrak{sl}_{r|n}) = \tilde{V}^m(\mathfrak{sl}_{r|n})$ for all $m,r \geq 1$, we have the following improvement of this result.

\begin{thm} 
\label{thm:CLRimproved} For all positive integers $r,n,m$, we have
$$
\text{Com}(\tilde{V}^{-n+r}(\mathfrak{sl}_m), A^{-n}(\mathfrak{sl}_{m}) \otimes L_r(\mathfrak{sl}_m)) \cong
\text{Com}(\tilde{V}^{-m}(\mathfrak{sl}_n) \otimes L_m(\mathfrak{sl}_r) \otimes \cH, {\widetilde V}^m(\mathfrak{sl}_{r|n})).
$$
Moreover, if $m<n$ or if $m \geq 2n +1$, then we have 
$$
\text{Com}(\tilde{V}^{-n+r}(\mathfrak{sl}_m), V^{-n}(\mathfrak{sl}_{m}) \otimes L_r(\mathfrak{sl}_m))\cong
\text{Com}(\tilde{V}^{-m}(\mathfrak{sl}_n) \otimes L_m(\mathfrak{sl}_r) \otimes \cH, {\widetilde V}^m(\mathfrak{sl}_{r|n})).
$$
\end{thm}
If $n = m$, $A^{-n}(\mathfrak{sl}_{n})$ cannot be replaced with $\tilde{V}^{-n}(\mathfrak{sl}_{n})$. Although $:D^+ D^-:$ can be expressed as a normally ordered polynomial in the generators of $V^{-n}(\gg\gl_n)$ and their derivatives, it is straightforward to check that $:D^+ \partial D^-:$ does not have this property. However, since $:D^+ \partial D^-:$ is invariant under $\gs\gl_n[t]$, it will lie in 
$\text{Com}(\tilde{V}^{-n+r}(\mathfrak{sl}_n), A^{-n}(\mathfrak{sl}_{n}) \otimes L_r(\mathfrak{sl}_n)) $ but not in $\text{Com}(\tilde{V}^{-n+r}(\mathfrak{sl}_n), V^{-n}(\mathfrak{sl}_{n}) \otimes L_r(\mathfrak{sl}_n))$. In the range $n < m<2n+1$ it is possible that the left hand side is unchanged by replacing $A^{-n}(\mathfrak{sl}_{m})$ with $\tilde{V}^{-n}(\mathfrak{sl}_{m})$, but we are unable to answer this question at present.

\subsection{Type $C$ case} 
Recall the homomorphism, 
$$V^{-\frac{n}{2}}(\gs\gp_{2m})\otimes V^{-2m}(\gs\go_{n}) \rightarrow \cS(mn),$$ whose image $\tilde{V}^{-\frac{n}{2}}(\gs\gp_{2m})\otimes \tilde{V}^{-2m}(\gs\go_{n})$ is conformally embedded. Therefore the coset
$$A^{-\frac{n}{2}}(\mathfrak{sp}_{2m}) = \text{Com}(\tilde{V}^{-2m}(\gs\go_{n}), \cS(mn)),$$ is an extension of $\tilde{V}^{-\frac{n}{2}}(\mathfrak{sp}_{2m})$.
Similarly, recall the homomorphisms
$$L_{m}(\gs\gp_{2r}) \otimes L_r(\gs\gp_{2m}) \rightarrow \cE(2mr),\qquad V^{-\frac{n}{2}+r}(\gs\gp_{2m}) \otimes V^{m}(\go\gs\gp_{n|2r}) \rightarrow \cS(mn) \otimes \cE(2mr),$$ and that by Theorems \ref{thm:spbc} and \ref{thm:spbgbc}, we have
$$\text{Com}(L_r(\gs\gp_{2r}), \cE(2mr)) \cong L_{m}(\gs\gp_{2r}),\qquad \text{Com}(\tilde{V}^{-\frac{n}{2}+r}(\gs\gp_{2m}), \cS(mn) \otimes \cE(2mr)) \cong \tilde{V}^{m}(\go\gs\gp_{n|2r}).$$ 

\begin{thm} 
\label{thm:levelranktypeC}
Let $r, n, m $ be positive integers. Then
$$\text{Com}(V^{-\frac{n}{2}+r}(\mathfrak{sp}_{2m}), A^{-\frac{n}{2}}(\mathfrak{sp}_{2m}) \otimes L_r(\mathfrak{sp}_{2m}))\cong
\text{Com}(\tilde{V}^{-2m}(\mathfrak{so}_{n}) \otimes L_m(\mathfrak{sp}_{2r}), \tilde{V}^{m}(\mathfrak{osp}_{n|2r})).$$
Moreover, if $m<\frac{n}{2}$, then we have
$$\text{Com}(\tilde{V}^{-\frac{n}{2}+r}(\mathfrak{sp}_{2m}), V^{-\frac{n}{2}}(\mathfrak{sp}_{2m}) \otimes L_r(\mathfrak{sp}_{2m}))\cong
\text{Com}(\tilde{V}^{-2m}(\mathfrak{so}_{n}) \otimes L_m(\mathfrak{sp}_{2r}), \tilde{V}^{m}(\mathfrak{osp}_{n|2r})).$$
\end{thm}

\begin{proof} The proof of the first statement is similar to the argument of \cite[Thm. 13.1]{ACL2}. We have
\begin{equation}\nonumber
\begin{split}
&\text{Com}(\tilde{V}^{-\frac{n}{2}+r}(\mathfrak{sp}_{2m}), A^{-\frac{n}{2}}(\mathfrak{sp}_{2m}) \otimes L_r(\mathfrak{sp}_{2m}))\cong \\
&\qquad\qquad\qquad\cong \text{Com}(\tilde{V}^{-\frac{n}{2}+r}(\mathfrak{sp}_{2m}),  \text{Com}(\tilde{V}^{-2m}(\mathfrak{so}_{n}), \cS(mn)) \otimes L_r(\mathfrak{sp}_{2m}))\\
&\qquad\qquad\qquad\cong \text{Com}(\tilde{V}^{-\frac{n}{2}+r}(\mathfrak{sp}_{2m}),  \text{Com}(\tilde{V}^{-2m}(\mathfrak{so}_{n}) \otimes L_m(\mathfrak{sp}_{2r}), \mathcal S(mn) \otimes \mathcal E(2mr)))\\
&\qquad\qquad\qquad\cong  \text{Com}(\tilde{V}^{-\frac{n}{2}+r}(\mathfrak{sp}_{2m}) \otimes \tilde{V}^{-2m}(\mathfrak{so}_{n}) \otimes L_m(\mathfrak{sp}_{2r}), \mathcal S(mn) \otimes \mathcal E(2mr))  \\
&\qquad\qquad\qquad\cong  \text{Com}(\tilde{V}^{-2m}(\mathfrak{so}_{n}) \otimes L_m(\mathfrak{sp}_{2r}), \text{Com}(\tilde{V}^{-\frac{n}{2}+r}(\mathfrak{sp}_{2m}), \mathcal S(mn) \otimes \mathcal E(2mr))) \\
&\qquad\qquad\qquad\cong  \text{Com}(\tilde{V}^{-2m}(\mathfrak{so}_{n}) \otimes L_m(\mathfrak{sp}_{2r}), \tilde{V}^{m}(\mathfrak{osp}_{n|2r})).
\end{split}
\end{equation}

For the second statement, we claim that $A^{-\frac{n}{2}}(\mathfrak{sp}_{2m})  \cong  V^{-\frac{n}{2}}(\mathfrak{sp}_{2m})$ for $m < \frac{n}{2}$. This is clear because $\text{gr}(\cS(nm)) \cong \mathbb{C}[J_{\infty}((\mathbb{C}^{n})^{\oplus 2m})]$ where $\mathbb{C}^{n}$ is the standard $SO_{n}$-module. Since $m< \frac{n}{2}$, $(\mathbb{C}^{n})^{\oplus 2m} /\!\!/ SO_{n}$ is an affine space, so the map \eqref{equ:invariantringmap} is surjective. In this case, the invariants are quadratic and correspond to the generators of $V^{-\frac{n}{2}}(\gs\gp_{2m})$, so the map \eqref{injgamm2} is surjective as well. This completes the proof. \end{proof}

Unfortunately, since we are unable to describe $A^{-\frac{n}{2}}(\mathfrak{sp}_{2m})$ when $m \geq \frac{n}{2}$, this statement cannot be improved at present.

\subsection{Type $D$ case}

Recall the homomorphism
$$V^{-m}(\gs\gp_{2n})\otimes V^{-2n}(\gs\go_{2m}) \rightarrow \cS(2mn),$$ whose image $\tilde{V}^{-m}(\gs\gp_{2n})\otimes \tilde{V}^{-2n}(\gs\go_{2m})$ is conformally embedded. Recall that by Theorem \ref{thm:spbg}, $\text{Com}(\tilde{V}^{-m}(\gs\gp_{2n}), \cS(2mn)) \cong \tilde{V}^{-2n}(\mathfrak{so}_{2m})$. Similarly, we have a conformal embedding
$$L_r(\gs\go_{2m}) \otimes L_{2m}(\gs\go_r) \rightarrow \cF(2mr),$$ so the coset
$$A_r(\gs\go_{2m}) = \text{Com}(L_{2m}(\gs\go_r), \cF(2mr)),$$ is an extension of $L_r(\gs\go_{2m})$.

Finally, we have a homomorphism
$$V^{-2n+r}(\gs\go_{2m}) \otimes V^{-m}(\go\gs\gp_{r|2n}) \rightarrow \cS(2nm) \otimes\cF(2mr)$$ whose image is conformally embedded. Therefore
$$A^{-m}(\go\gs\gp_{r|2n}) = \text{Com}(\tilde{V}^{-2n+r}(\gs\go_{2m}), \cS(2nm) \otimes\cF(2mr)),$$ is an extension of the image $\tilde{V}^{-m}(\go\gs\gp_{r|2n})$.

\begin{thm} 
\label{thm:levelranktypeD}
Let $r, n, m $ be positive integers. Then
$$\text{Com}(\tilde{V}^{-2n+r}(\mathfrak{so}_{2m}), \tilde{V}^{-2n}(\mathfrak{so}_{2m}) \otimes A_r(\mathfrak{so}_{2m}))\cong
\text{Com}(\tilde{V}^{-m}(\mathfrak{sp}_{2n}) \otimes L_{2m}(\mathfrak{so}_{r}), A^{-m}(\mathfrak{osp}_{r|2n})).$$
\end{thm}

\begin{proof} Again, this is similar to the argument of \cite[Thm. 13.1]{ACL2}. We have
\begin{equation}\nonumber
\begin{split}
&\text{Com}(\tilde{V}^{-2n+r}(\mathfrak{so}_{2m}), \tilde{V}^{-2n}(\mathfrak{so}_{2m}) \otimes A_r(\mathfrak{so}_{2m}))\cong 
\\ &\qquad\qquad\qquad\cong \text{Com}(\tilde{V}^{-2n+r}(\mathfrak{so}_{2m}),  \text{Com}(\tilde{V}^{-m}(\mathfrak{sp}_{2n}), \cS(2mn)) \otimes A_r(\mathfrak{so}_{2m})) \\
&\qquad\qquad\qquad\cong \text{Com}(\tilde{V}^{-2n+r}(\mathfrak{so}_{2m}),  \text{Com}(\tilde{V}^{-m}(\mathfrak{sp}_{2n}) \otimes L_{2m}(\mathfrak{so}_{r}), \cS(2mn) \otimes \mathcal F(2mr)))\\
&\qquad\qquad\qquad\cong  \text{Com}(\tilde{V}^{-2n+r}(\mathfrak{so}_{2m}) \otimes \tilde{V}^{-m}(\mathfrak{sp}_{2n}) \otimes L_{2m}(\mathfrak{so}_{r}), \cS(2mn) \otimes \cF(2mr))  \\
&\qquad\qquad\qquad\cong  \text{Com}(\tilde{V}^{-m}(\mathfrak{sp}_{2n}) \otimes L_{2m}(\mathfrak{so}_{r}), \text{Com}(\tilde{V}^{-2n+r}(\mathfrak{so}_{2m}), \cS(2mn) \otimes \cF(2mr))) \\
&\qquad\qquad\qquad\cong  \text{Com}(\tilde{V}^{-m}(\mathfrak{sp}_{2n}) \otimes L_{2m}(\mathfrak{so}_{r}), A^{-m}(\mathfrak{osp}_{r|2n}))
.\end{split}
\end{equation}
 \end{proof}

Unfortunately, since we are unable to describe $A^{-m}(\mathfrak{osp}_{r|2n})$, this statement cannot be improved at present.

\subsection{Type $B$ case} 
Recall the homomorphism, 
$$V^{-2n+1}(\gs\go_{2m+1}) \otimes V^{-m-\frac{1}{2}}(\mathfrak{osp}_{1|2n}) \rightarrow \cS(n(2m+1)) \otimes \cF(2m+1),$$ whose image $\tilde{V}^{-2n+1}(\gs\go_{2m+1}) \otimes \tilde{V}^{-m-\frac{1}{2}}(\mathfrak{osp}_{1|2n})$ is conformally embedded. Therefore the coset
$$A^{-2n+1}(\mathfrak{so}_{2m+1}) = \text{Com}(\tilde{V}^{-m-\frac{1}{2}}(\mathfrak{osp}_{1|2n}), \cS(n(2m+1)) \otimes \cF(2m+1)),$$ is an extension of $\tilde{V}^{-2n+1}(\mathfrak{so}_{2m+1})$.

Similarly, we have a conformal embedding
$$L_r(\gs\go_{2m+1}) \otimes L_{2m+1}(\gs\go_r) \rightarrow \cF(r(2m+1)),$$ so the coset
$$A_r(\gs\go_{2m+1}) = \text{Com}(L_{2m+1}(\gs\go_r), \cF(r(2m+1))),$$ is an extension of $L_r(\gs\go_{2m+1})$.

Finally, we have a homomorphism
$$V^{-2n+1+r}(\gs\go_{2m+1}) \otimes V^{-m-\frac{1}{2}}(\go\gs\gp_{r+1|2n}) \rightarrow \cS(n(2m+1)) \otimes\cF(2m+1) \otimes \cF(r(2m+1))$$ whose image is conformally embedded. Therefore
$$A^{-m-\frac{1}{2}}(\go\gs\gp_{r+1|2n}) = \text{Com}(\tilde{V}^{-2n+1+r}(\gs\go_{2m+1}), \cS(n(2m+1)) \otimes\cF(2m+1) \otimes \cF(r(2m+1))),$$ is an extension of the image $\tilde{V}^{-m-\frac{1}{2}}(\go\gs\gp_{r+1|2n})$. Note finally that $\tilde{V}^{-m-\frac{1}{2}}(\go\gs\gp_{r+1|2n})$ admits an action of $V^{-m-\frac{1}{2}}(\go\gs\gp_{1|2n}) \otimes L_{2m+1}(\gs\go_r)$.

For notational convenience below, we will write 
$\cS(n(2m+1)) \otimes\cF(2m+1) \otimes \cF(r(2m+1))$ in the form $\cS(n(2m+1)) \otimes \cF((r+1)(2m+1))$.

\begin{thm} 
\label{thm:levelranktypeB}
Let $r, n, m $ be positive integers. Then
\begin{equation} \begin{split} & \text{Com}\big(\tilde{V}^{-2n+1+r}(\mathfrak{so}_{2m+1}), A^{-2n+1}(\mathfrak{so}_{2m+1}) \otimes A_r(\mathfrak{so}_{2m+1}) \big)
\\ & \cong
\text{Com}\big(\tilde{V}^{-m - \frac{1}{2}}(\mathfrak{osp}_{1|2n}) \otimes L_{2m+1}(\gs\go_{r}), A^{-m- \frac{1}{2}}(\mathfrak{osp}_{r+1|2n})\big).\end{split} \end{equation}
\end{thm}

\begin{proof} This is again similar to the proof of \cite[Thm. 13.1]{ACL2}. We have

\begin{equation}
\begin{split}
& \text{Com}(\tilde{V}^{-2n+1 + r}(\mathfrak{so}_{2m+1}), A^{-2n+1}(\mathfrak{so}_{2m+1}) \otimes A_r(\mathfrak{so}_{2m+1})) \\
& \cong \text{Com}\big(\tilde{V}^{-2n+1 + r}(\mathfrak{so}_{2m+1}),  \text{Com} \big(\tilde{V}^{-m-\frac{1}{2}}(\mathfrak{osp}_{1|2n}), \cS(n(2m+1)) \otimes \cF(2m+1)\big)\otimes A_r(\mathfrak{so}_{2m+1})\big)
\\ & \cong \text{Com}\big(\tilde{V}^{-2n+1 + r}(\mathfrak{so}_{2m+1}),  \text{Com}\big(\tilde{V}^{-m-\frac{1}{2}}(\mathfrak{osp}_{1|2n}) \otimes L_{2m+1}(\gs\go_r), \cS(n(2m+1)) \otimes \cF((r+1)(2m+1))\big)\big)
 \\ & \cong  \text{Com}\big(\tilde{V}^{-2n+1 + r}(\mathfrak{so}_{2m+1}) \otimes \tilde{V}^{-m-\frac{1}{2}}(\mathfrak{osp}_{1|2n}) \otimes  L_{2m+1}(\gs\go_r), \mathcal \cS(n(2m+1)) \otimes \cF((r+1)(2m+1))\big)
 \\ & \cong  \text{Com}\big(\tilde{V}^{-m-\frac{1}{2}}(\mathfrak{osp}_{1|2n}) \otimes L_{2m+1}(\gs\go_r), \text{Com}(\tilde{V}^{-2n+1 + r}(\mathfrak{so}_{2m+1}), \mathcal \cS(n(2m+1)) \otimes\cF((r+1)(2m+1))\big)
 \\ & \cong  \text{Com}\big(\tilde{V}^{-m-\frac{1}{2}}(\mathfrak{osp}_{1|2n}) \otimes L_{2m+1}(\gs\go_r), A^{-m-\frac{1}{2}}(\go\gs\gp_{r+1|2n})\big).
\end{split}
\end{equation}
\end{proof}

It is an interesting question whether Theorem \ref{thm:levelranktypeB} remains true if $A^{-2n+1}(\mathfrak{so}_{2m+1})$, $A_r(\mathfrak{so}_{2m+1})$, and $A^{-m-\frac{1}{2}}(\go\gs\gp_{r+1|2n})$ are replaced with $\tilde{V}^{-2n+1}(\mathfrak{so}_{2m+1})$, $\tilde{V}_r(\mathfrak{so}_{2m+1})$, and $\tilde{V}^{-m-\frac{1}{2}}(\go\gs\gp_{r+1|2n})$, but we are not able to answer this question using the methods in this paper.

\appendix

\section{} \label{appendix}

In this Appendix we prove Theorem \ref{thm:subregrel}. Recall that $\cW^{-n}(\mathfrak{sl}_n, f_{\text{subreg}})$ is isomorphic to $\cS(n^2)^{\gs\gl_n[t] \oplus \gs\gl_n[t]}$, which has strong generators 
$$D^+ = \left|\begin{matrix} \beta^{11}& \cdots &  \beta^{1n} \cr  \vdots  & & \vdots  \cr   \beta^{n1}  & \cdots &  \beta^{nn} \end{matrix} \right|, \qquad D^-  =  \left|\begin{matrix} \gamma^{11} & \cdots & \gamma^{1n} \cr  \vdots  & & \vdots  \cr  \gamma^{n1} & \cdots & \gamma^{nn} \end{matrix} \right|,\qquad J = \sum_{i,j = 1}^n :\beta^{ij} \gamma^{ij}:,$$ together with central elements $\omega_2,\dots, \omega_{n-1}$ of conformal weights $2,\dots, n-1$. Let $I \subseteq \cS(n^2)^{\gs\gl_n[t] \oplus \gs\gl_n[t]}$ be the ideal generated by $\omega_2,\dots, \omega_{n-1}$. For elements $A,B \in \cS(n^2)^{\gs\gl_n[t] \oplus \gs\gl_n[t]}$, we say
$$A\simeq B, \quad \text{if} \ A-B\in I.$$

We have the following relations
\begin{equation}
\begin{split} & D^{+}_{(n-1)}D^{-}\simeq n! 1
\\ & D^{+}_{(n-2)}D^{-}\simeq n!\frac {J_{(-1)}} n 1,
\\ & D^{+}_{(n-3)}D^{-}\simeq n!(\frac 1 {2!}(\frac {J_{(-1)}} n)^2+\frac {J_{(-2)}}{2n})1,
\\ & D^{+}_{(n-4)}D^{-}\simeq n!(\frac 1 {3!} (\frac {J_{(-1)}}  n)^3+\frac {J_{(-1)}} n \frac {J_{(-2)}} {2n} +  \frac{J_{(-3)} }{3n})1.\end{split} \end{equation}
Similarly, we have
\begin{equation}
\begin{split} & D^{-}_{(n-1)}D^{+}\simeq (-1)^nn! 1,
\\ & D^{-}_{(n-2)}D^{+}\simeq (-1)^nn!\frac {J_{(-1)}} {-n} 1,
\\ & D^{-}_{(n-3)}D^{+}\simeq (-1)^nn!(\frac 1 {2!}(\frac {J_{(-1)}} {-n})^2+\frac {J_{(-2)}}{-2n})1,
\\ & D^{-}_{(n-4)}D^{+}\simeq (-1)^nn!(\frac 1 {3!} (\frac {J_{(-1)}}  {-n})^3+\frac {J_{(-1)}} {-n} \frac {J_{(-2)}} {-2n} +  \frac{J_{(-3)} }{-3n})1.\end{split} \end{equation}

Next, let $P_k$ be the polynomial in $x_1,\dots,x_k$ given by
$$P_k(x_1,\dots,x_k)=\sum_{\sum i s_i=k}  \prod_{i=1}^k \frac {x_i^{s_i}}{s_i!}.$$
If $Q=\exp(\sum x_it^i)$, then $Q=\sum P_k t^k$. Let $x_i=\frac{J_{(-i)}}{i n}$.
One can verify that 
\begin{equation} \begin{split} & \frac 1 {n!}D^{+}_{(n-1-k)}D^{-}\simeq p_k, \ \text{where} \ p_k =P_k(\frac {J_{(-1)}}{n},\dots, \frac {J_{(-k)}}{kn})1,
\\ &  \frac{(-1)^n}{n!}D^{-}_{(n-1-k)}D^{+}\simeq \bar p_k,\ \text{where} \ \bar p_k =P_k(\frac {J_{(-1)}}{-n},\dots, \frac {J_{(-k)}}{-kn})1.
\end{split} \end{equation}
Next, observe that if $J_{(-i)}$ is replaced by $n$,
$$Q(n)=\exp(-\log (1-t))=\sum t^s.$$
We obtain
\begin{lemma}

\noindent \begin{enumerate}
\item For all $k\geq 0$, $p_k(n)=1$.
\item If $J_{(-i)}$ is replaced by $-n$, we have
$$p_0(-n)=1,\quad p_1(-n)=-1, \quad p_k(-n)=0, \quad \text{for}\  k>1.$$ 
\item Similarly,
$$\bar p_k(-n)=1, \quad \bar p_0(n)=1,\quad \bar p_1(n)=-1,\quad \bar p_k(n)=0,\quad \text{for} \  k>1.$$
\end{enumerate}
\end{lemma}
We also need the following computations. For $s\geq -1$,
\begin{equation} \begin{split}  (p_k )_{(s)} D^+ &  =(p_{k-s-1})p_{s+1}(n)(-1)^{s+1}D^+,
\\  (p_k )_{(s)} D^- & =(p_{k-s-1})p_{s+1}(-n)(-1)^{s+1}D^-.
\end{split} \end{equation}

Recall the fields
$$U_{i,j} = \ :(\partial^i D^+)(\partial^j D^-):,\qquad \ i,j \geq 0,$$ which have weight $n+i+j$. For fields $A,B$ of weight $2n+m+1$, we say that $A\sim B$ if $A-B$ is a normally ordered polynomial in $J,W^2,\dots, W^{n-1}, U_{0,0}, U_{0,1},\dots, U_{0,n+m}$, and their derivatives; equivalently, the coefficient of $U_{0,n+m+1}$ in $A-B$ is zero. We have
\begin{lemma} For $k-s_0-s_1=n+m+3$, we have
\begin{equation} \label{eq:app2}  \begin{split} & D_{\left(s_{0}\right)}^{-} D_{\left(s_{1}\right)}^{+} p_{k}\sim D_{\left(s_{0}\right)}^{-} D_{\left(s_{1}-k\right)}^{+}-D_{\left(s_{0}-1\right)}^{-} D_{\left(s_{1}-k_{1}+1\right)}^{+},
\\ & D_{\left(s_{0}\right)}^{+} D_{\left(s_{1}\right)}^{-} \bar p_{k}\sim D_{\left(s_{0}\right)}^{+}
 D_{\left(s_{1}-k\right)}^{-}-D_{\left(s_{0}-1\right)}^{+} D_{\left(s_{1}-k_{1}+1\right)}^{-} . \end{split} \end{equation}
 \end{lemma}
\begin{proof}
	\begin{equation*}
		\begin{split} 
			 D_{(s)}^{+} \frac{J_{(-k)}}{n}-\frac{J_{(-k)}}{n} D_{(s)}^{+} & =-\frac{1}{n}\left(J_{(0)} D^{+}\right)_{(s-k)}=D^{+}_{(s-k)},
			\\  D^{-}_{(s)} \frac{J_{(-k)}}{n}-\frac{J_{(-k)}}{n} D_{(s)}^{-} & =-\frac{1}{n}\left(J_{(0)} D^{-}\right)_{(s-k)}=-D_{(s-k)}^{-}.\\
			D^{+}_{(s)} \frac{\left(x_{k} t\right)^{m}}{m !}&= \sum_{s=0}^{m} \frac{\left(x_{k} t\right)^{n}}{n !} D_{(s-(m-n) k)}^{+} \frac{1}{(m-s) !}\left(\frac{1}{k} t\right)^{m-n},
			\\ D^{+}_{(s)} e^{x_{k} t}&= e^{x_{k} t} \sum_{m=0}^{\infty} D^{+}_{(s-m k)} \frac{1}{m !}\left(\frac{1}{k} t\right)^{m},\\
			D^{-}_{(s)} e^{x_{k} t}&= e^{x_{k} t} \sum_{m=1}^{\infty} D_{\left(s-m_{k}\right)}^{-} \frac{1}{m !}\left(\frac{-1}{k} t\right)^{m},\\
			D_{(s)}^{+} Q&= Q \sum_{m=0}^{\infty} D_{(s-m)}^{+} p_{m}(n) t^{m}=Q \sum_{m=0}^{\infty} D_{(s-m)}^{+} t^{m},\\
			D_{(s)}^{-}  Q&= Q \sum_{m=0}^{\infty} D_{(s-m)}^{-} \bar{p}_{m}(n) t^{m}=Q\left(D_{(s)}^{-}-tD_{(s-1)}^{-}\right).\end{split} \end{equation*}
		Thus 
		$$D_{\left(s_{0}\right)}^{-} D_{\left(s_{1}\right)}^{+} Q=Q\left(D_{(s_0)}^{-}-tD_{(s_0-1)}^{-}\right)\sum_{m=0}^{\infty} D_{(s_1-m)}^{+} t^{m}.$$
			$$D_{\left(s_{0}\right)}^{+} D_{\left(s_{1}\right)}^{-} Q=Q\sum_{m=0}^{\infty} D_{(s_0-m)}^{+} t^{m}\left(D_{(s_1)}^{-}-tD_{(s_1-1)}^{-}\right).$$
		In the above equations, the  coefficients of $t^k$ give the equations in the lemma.
\end{proof}
We now consider 
\begin{equation*} \begin{split} :U_{1,m}U_{0,0}:-:U_{0,m}U_{1,0}:  \ \sim & \ :(:\partial^mD^-\partial D^+:)(:D^+D^-:):-:(:\partial^mD^- D^+: )(:\partial D^+D^-:):
\\  \sim & \sum_{k\geq 0}\partial^mD^-_{(-2-k)}D^+_{(-1)}\partial D^+_{(k)} D^- - \sum_{k\geq 0}\partial^mD^-_{(-2-k)}\partial D^+_{(-1)}D^+_{(k)} D^- 
\\ & +\sum_{k\geq 0}\partial D^+_{(-2-k)}\partial^m D^-_{(k)} D^+_{(-1)} D^- - \sum_{k\geq 0}D^+_{(-2-k)}\partial^m D^-_{(k)}\partial D^+_{(-1)} D^-.
\end{split} \end{equation*} 
We write 
\begin{equation*} \begin{split} & A = \sum_{k\geq 0}\partial^mD^-_{(-2-k)}D^+_{(-1)}\partial D^+_{(k)} D^- - \sum_{k\geq 0}\partial^mD^-_{(-2-k)}\partial D^+_{(-1)}D^+_{(k)} D^-,
\\ & B = \sum_{k\geq 0}\partial D^+_{(-2-k)}\partial^m D^-_{(k)} D^+_{(-1)} D^- - \sum_{k\geq 0}D^+_{(-2-k)}\partial^m D^-_{(k)}\partial D^+_{(-1)} D^-.\end{split} \end{equation*}
In the remainder of this Appendix, we shall compute the contributions of $A$ and $B$ to the coefficient of  $:D^{+} \partial^{n+m+1} D^{-}:$. First, we have
\begin{eqnarray*}
		A & \sim& {n !}\sum_{k=1}^{n} \frac{-k(m+k+1) !}{(k+1) !} D_{(-2- k-m)}^{-} D^{+} p_{n-k}-n ! \sum_{k=0}^{n-1} \frac{(m+k+1) !}{(k+1) !} D_{(-2-k-m)} D_{(-2)}^{+} p_{n-1-k} \\
		& \sim &n ! \sum_{k=1}^{n} \frac{-k(m+k+1) !}{(k+1) !}\left(D_{(-2-k-m)}^{-} D_{(-1-n+k)}^{+}- (1-\delta_k^{n})D_{(-3-k-m)}^{-} D_{(-n+k)}^{+}\right)  \\
		&&-n ! \sum_{k=0}^{n-1} \frac{(m+k+1) !}{(k+1) !}\left(D_{(-2-k-m)}^{-} D_{(-1-n+k)}^{+}- (1-\delta_k^{n-1}) D_{(-3-k-m)}^{-} D_{(-n+k)}^{+}\right) \ \text{by}\ \eqref{eq:app2} \\
		& \sim &n ! \sum_{k=1}^{n} \frac{-k}{(k+1) !(n-k) !} :\partial^{k+n+1} D^{-} \partial^{n-k} D^{+}: \\
		&& -n ! \sum_{k=1}^{n-1} \frac{-k}{(k+1) !(n-k-1) !} \frac{1}{2+k+m} :\partial^{k+m+2} D^{-} \partial^{n-k-1} D^{+}: \\
		&&-n ! \sum_{k=0}^{n-1} \frac{1}{(k+1) !(n-k) !} \partial^{k+n+1} : D^{-} \partial^{n-k} D^{+} :\\
		&&+n ! \sum_{k=0}^{n-2} \frac{1}{(k+1) !(n-k-1) !} \frac{1}{2+k+m} : \partial^{k+m+2} D^{-} \partial^{n-k+1} D^{+}: \\
		&\sim &\left(\sum_{k=0}^{n} \frac{n !(-1)^{n-k}}{k !(n-k) !}+\frac{1}{n+1}+\sum_{k=0}^{n-1} \frac{n !}{k !(n-k-1) !} \frac{(-1)^{n-k-1}}{2+k+m}-\frac{1}{n+m+1}\right) :D^{+} \partial^{m+n+1} D^{-} :\\
		&=&\left(\frac{1}{n+1}-\frac{1}{n+m+1}+\sum_{k=0}^{n-1} \frac{n !}{k !(n-k-1) !} \frac{(-1)^{n-k-1}}{2+k+m}\right) :D^{+} \partial^{n+m+1} D^{-}:.
\end{eqnarray*}
The contribution from $B$ is more difficult to compute, and we need the following preliminary calculations.
For $k>m$,
\begin{equation} \label{eq:calc2} \begin{split}
D^-_{(k-m)} (D^{+}_{(-1)} D^{-})& = \sum_{s=0}^{k-m}\left(\begin{array}{c}k-m \\ s\end{array}\right) n ! (\bar{p}_{n-1-s})_{(k-m-1-s)} D^{-} 
\\ & = \sum_{s=0}^{k-m}\left(\begin{array}{c}k-m \\ s\end{array}\right) n ! \bar p_{n+m-k-1}(-1)^{k-m-s} D^{-} 
\\ & = 0, 
\\ D^{-}_{(k-m)} (D_{(-2)}^{+} D^{-}) &= \sum_{s=0}^{k-m}\left(\begin{array}{c}
		k-m \\
		s
	\end{array}\right) (\bar{p}_{n-1-s})_{(k-m-2-s)} D^{-}
	\\ &= \sum_{s=0}^{k-m-1}\left(\begin{array}{c}
		k-m \\
		s
	\end{array}\right) \bar p_{n+m-k}(-1)^{k-m-s-1} D^{-}+\partial \bar p_{n+m-1-k} D^- 
	\\ & =\left(\bar{p}_{n+m-k}+\partial \bar{p}_{n+m-1-k}\right) D^{-}.
\end{split} \end{equation}
We also need
\begin{equation} \label{eq:calc3} \begin{split}
		\bar{p}_{k} D^{-} & =\sum_{l=0}^{k} \frac{(-1)^{l}}{l !} \partial^{l}\left(D^{-} \bar p_{k-l}\right) 
		\\ \partial \bar{p}_{k} D^{-}-D^{-} \partial \bar{p}_{k} & =-\sum_{t=0}^{k-1} \frac{(-1)^{t+2}}{(t+2) !} \partial^{t+2}\left(D^{-} \bar p_{k-t-1}\right).
\end{split} \end{equation}
\begin{eqnarray*}
B &\sim &\sum_{k\geq m}(-1)^m\frac {k!(2+k)}{(k-m)!}D^+_{(-3-k)}D^{-}_{(k-m)}D^+_{(-1)}D^{-}- \sum_{k\geq m}(-1)^m\frac {k!}{(k-m)!} D^+_{(-2-k)}D^{-}_{(k-m)}D^+_{(-2)}D^{-}\\
&\sim &(-1)^{m+n}m!n!(2+m)D^+_{(-3-m)}\bar p_{n-1}D^{-}-(-1)^{m+n}m!n!D^+_{(-2-m)}\partial \bar p_{n-1} D^{-} 
\\
&& -\sum_{k=m+1}^{n+m}(-1)^{m+n}\frac {k!n!}{(k-m)!}D^+_{(-2-k)} (\partial \bar p_{n+m-1-k}+\bar p_{n+m-k})D^{-} \\
&=& (-1)^{m+n}m!n!(2+m)D^+_{(-3-m)}\bar p_{n-1}D^{-}\\
&& -\sum_{k=m+1}^{n+m}(-1)^{m+n}\frac {k!n!}{(k-m)!}D^+_{(-2-k)} \bar p_{n+m-k}D^{-} \\
&& -\sum_{k=m}^{n+m-2}(-1)^{m+n}\frac {k!n!}{(k-m)!}D^+_{(-2-k)} \partial \bar p_{n+m-1-k}D^{-}.
\end{eqnarray*}
Note that we have used both \eqref{eq:calc2} and \eqref{eq:calc3} in this calculation. The first term of $B$ yields
\begin{equation*} \begin{split} & (-1)^{m+n}m!n!(2+m)D^+_{(-3-m)}\bar p_{n-1}D^{-}
\\ & = (-1)^{m+n}m!n!(2+m) D^+_{(-3-m)}  \sum_{l=0}^{n-1} \frac{(-1)^l}{l!} \partial^l (D^- \bar p_{n-1-l})
\\ & \sim -(-1)^nn!\frac{1}{m+1} \sum_{l=0}^{n-2} \frac{(-1)^{l+1}}{l !(n-l-2) !(m+l+3)}  :D^+\partial^{m+n+1} D^{-}:.\end{split} \end{equation*}
For the second term of $B$, we compute
\begin{equation*} \begin{split}
	  - & \sum_{k=m+1}^{n+m}(-1)^{m+n}\frac {k!n!}{(k-m)!}D^+_{(-2-k)} \bar p_{n+m-k}D^{-} 
	 \\  = & \sum_{k=m+1}^{n+m} \frac{(-1)^{m+n}n!}{(k-m) !(k+1)} \sum_{l=0}^{m+n-k} :(\partial^{k+1} D^{+}) \bigg(\frac{(-1)^{l}}{l !} \partial^{l}\left(D^{-} \bar p_{n+m-k-l}\right)\bigg):
	 \\  \sim &  \sum_{k=m+1}^{n+m} \frac{(-1)^{m+n}n!}{(k-m) !(k+1)} \sum_{i=0}^{m+n-k} \frac{1}{l !} :(\partial^{k+l+1} D^{+}) D^{-} \bar p_{n+m-k-l}:
	\\ = & \sum_{k=m+1}^{n+m} \frac{(-1)^{m+n}n!}{(k-m) !(k+1)}\left(\sum_{l=0}^{m+n-k} \frac{1}{l !(m+n-k-l) !} :(\partial^{k+l+1} D^{+}) \partial^{m+n-k-l} D^{-}:\right.
	\\ & \left.-  \sum_{l=0}^{m+n-k-1} \frac{1}{l !(m+n-k-l-1) !(k+l+2)} :(\partial^{k+1+2} D^{+}) \partial^{m+n-k-l-1} D^{-}:\right) \end{split} \end{equation*}	
	\begin{equation*} \begin{split}
	\\  \sim &  \sum_{k=m+1}^{m+n} \frac{(-1)^{m+n}n!}{(k-m) !(k+1)}\left(\sum_{l=0}^{m+n-k} \frac{(-1)^{k+l+1}}{l !(m+n-k-l) !}\right.
	\\ & \left.-\sum_{l=0}^{m+n-k-1} \frac{(-1)^{k+l} }{l !(m+n-k-l-1) !(k+l+2)}{D^{+}}\right) :D^+\partial^{m+n+1}{D}^{-}:
	\\ = & (-1)^n n!\left(\frac{(-1)^{n+1}}{n !(m+n+1)}-\sum_{s=1}^{n} \frac{1}{s !(m+s+1)} \sum_{l=0}^{n-s-1} \frac{(-1)^{l+s}}{l !(n-s-l-1) !(m+s+l+2)}\right)
	\\ &  :D^+\partial^{m+n+1} D^{-}:.
\end{split} \end{equation*}

For the third term of $B$, we compute
\begin{equation*} \begin{split}
	\ & \sum_{k=m}^{n+m-2} \frac{(-1)^{m+n}n! k !}{(k-m) !} D_{(-2-k)}^{+} \partial \bar p_{n+m-1-k} D^{-}
	\\ = &\sum_{k=m}^{n+m-2} \frac{(-1)^{m+n}n! k!}{(k-m) !}\bigg(D_{(-2-k)}^{+} D^{-} \partial \bar p_{n+m-1-k}-\sum_{t=0}^{m+n-k-2} \frac{(-1)^{t+2}}{(t+2) !} D^{+}_{(-2-k)} \partial^{t+2}\bigg(D^{-} \bar p_{n+m-k-t-2}\bigg)\bigg)
	\\ = &\sum_{k=m}^{n+m-2} \frac{(-1)^{m+n}n!}{(k-m) !(k+1)}\bigg(-:(\partial^{k+2} D^{+}) D^{-} \bar p_{n+m-1-k}:-:(\partial^{k+1} D^{+}) \partial D^{-} \bar p_{n+m-1-k}:\bigg)
	\\ &-\sum_{k=m}^{n+m-2} \frac{(-1)^{m+n}n!}{(k-m) !(k+1)} \sum_{t=0}^{m+n-k-2} \frac{1}{(t+2) !} :(\partial^{k+t+3} D^{+}) D^{-} p_{n+m-k-t-2}:
	\end{split} \end{equation*}
	\begin{equation*} \begin{split}
	 =&\sum_{k=m}^{n+m-2}\frac{(-1)^{m+n}n!}{(k-m) !(k+1)}\bigg(-:(\partial^{k+2} D^{+}) \frac{\partial^{n+m-1-k} D^{-}}{(n+m-k-1)!}:+\frac{1}{k+3} :(\partial^{k+3} D^{+})\frac{\partial^{n+m-k-2} D^{-}}{(n+m-k-2) !}:
	\\ & -:(\partial^{k+1} D^{+}) \frac{\partial^{n+m-k} D^{-}}{(n+m-k) !}:+\frac{1}{k+2} :(\partial^{k+2} D^{+}) \frac{\partial^{n+m-k-1} D^{-}}{(n+m-k-1)!}:
	\\ &- \sum_{t=0}^{m+n-k-2} \frac{1}{(t+2) !}\bigg(:(\partial^{k+t+3} D^{+}) \frac{\partial^{n+m-k-t-2} D^{-}}{(n+m-k-t-2) !}: 
	\\ & -\frac{1}{k+t+4} :(\partial^{k+t+4} D^{+}) \frac{\partial^{n+m-k-t-3} D^{-}}{(n+m-k-t-3) !}:\bigg)\bigg) \end{split} \end{equation*}
	\begin{equation*} \begin{split}
	 \sim & (-1)^nn! \sum_{s=0}^{n-2} \bigg(-\frac{(-1)^{s}}{s !(m+s+1)(n-s-1) !}+\frac{(-1)^{s+1}}{s !(m+s+1)(m+s+3)(n-s-2) !}
	\\ & - \frac{(-1)^{s+1}}{s !(m+s+1)(n-s) !}+\frac{(-1)^{s}}{s !(m+s+1)(n-s-1) !(m+s+2)}\bigg) :D^+\partial^{m+n+1} D^{-}: 
	\\ & +(-1)^nn!\sum_{s=0}^{n-2} (\frac{(-1)^{s+1}}{s !(m+s+1)(n-s) !}+\frac{(-1)^{s}}{s !(m+s+1)(n-s-1) !}) :D^+\partial^{m+n+1} D^{-}:
	\\ &+(-1)^nn!\sum_{s=0}^{n-2} \frac{1}{s !(m+s+1)} \sum_{t=0}^{n-s-3} \frac{(-1)^{s+t}}{(t+2) !(m+s+t+4)(n-s-t-3) !} :D^+\partial^{m+n+1} D^{-}: 
		\end{split} \end{equation*}
	\begin{equation*} \begin{split}
	 \sim & (-1)^nn!\sum_{s=0}^{n-2} \frac{1}{s !(m+s+1)} \sum_{t=0}^{n-s-1} \frac{(-1)^{s+t}}{t !(m+s+t+2)(n-s-t-1) !} :D^+\partial^{m+n+1} D^{-}:. 
	\end{split} \end{equation*}

Combining these contributions, we get
\begin{equation*} \begin{split}
	B\sim &-(-1)^nn!\frac{1}{m+1} \sum_{l=0}^{n-2} \frac{(-1)^{l+1}}{l !(n-l -2) !(m+l+3)}:D^+\partial^{m+n+1} D^{-}:
	\\ &-(-1)^nn!\left(\frac{(-1)^{n+1}}{n !(m+n+1)}-\sum_{s=1}^{n} \frac{1}{s !(m+s+1)} \sum_{l=0}^{n-s-1} \frac{(-1)^{l+s}}{l !(n-s-l-1) !(m+s+l+2)}\right) 
	\\ & :D^+\partial^{m+n+1} D^{-}:
	 \\ &-(-1)^nn!\sum_{s=0}^{n-2} \frac{1}{s !(m+s+1)} \sum_{t=0}^{n-s-1} \frac{(-1)^{s+t}}{t !(m+s+t+2)(n-s-t-1) !} :D^+\partial^{m+n+1} D^{-}: 
	 \\ &=\sum_{l=0}^{n-1}\frac{(-1)^{n-l}n!}{l!(n-l-1)!(m+l+2)}+\frac{m}{(m+n)(m+n+1)} :D^+\partial^{m+n+1} D^{-}:.
	\end{split} \end{equation*}
	
Finally, this yields
\begin{equation*} \begin{split} A+B & \sim (\frac 1{n+1}-\frac1 {n+m+1}+\frac m{(n+m)(n+m+1)}) :D^+\partial^{m+n+1} D^{-}: 
\\ & =\frac{m(m+2n+1)}{(n+1)(n+m)(n+m+1)} :D^+\partial^{m+n+1} D^{-}:,	\end{split} \end{equation*} which completes the proof of Theorem \ref{thm:subregrel}.

\end{document}